\documentclass[twocolumn]{article}    
\usepackage{cite}
\usepackage[a4paper, portrait, left=.7in, right=.7in, bottom=1.1in, top=1.1in]{geometry}
\usepackage[T1]{fontenc}
\usepackage{algorithm}
\usepackage{algpseudocode}
\usepackage{nccmath}
\usepackage{graphicx}
\usepackage{subcaption}
\usepackage{textcomp}
\usepackage{bm}
\usepackage{hyperref}

\hypersetup{ colorlinks = true, citecolor = blue, linkcolor=blue }
\usepackage{mathtools}
\usepackage{booktabs}
\usepackage{import}
\usepackage{commath}
\usepackage[overload]{empheq}
\usepackage{multirow}
\usepackage{gensymb}
\usepackage[normalem]{ulem}
\usepackage{enumitem}

\usepackage{amsmath,amsfonts,amssymb}
\usepackage[usenames,dvipsnames,svgnames]{xcolor}
\usepackage{tikz}

\usetikzlibrary{fit,positioning}
\tikzstyle{every node}=[font=\scriptsize]
\tikzstyle{block} = [draw, rectangle]
\tikzstyle{sum} = [draw, circle, node distance=1cm, inner sep=2pt]
\tikzstyle{input} = [coordinate]
\tikzstyle{output} = [coordinate]
\tikzstyle{pinstyle} = [pin edge={to-,thin,black}]

\tikzstyle{chart_node_algorithm} = [rectangle, rounded corners, text width=0.25\textwidth, minimum height=1cm,text centered, fill=LightSteelBlue!10, draw=LightSteelBlue, thick]

\tikzstyle{chart_node_lqr} = [rectangle, rounded corners, text width=0.25\textwidth, minimum height=1cm,text centered, fill=DarkOrange!10, draw=DarkOrange, thick]

\tikzstyle{chart_node_theory} = [rectangle, rounded corners, text width=0.25\textwidth, minimum height=1cm,text centered, draw=MediumAquamarine, thick, fill=MediumAquamarine!10]

\newcommand\new[1]{{\color{black}#1}}

\usepackage{amsthm}

\usetikzlibrary{shapes,arrows, calc}
\theoremstyle{plain}
\newtheorem{theorem}{Theorem}
\newtheorem{proposition}[theorem]{Proposition}
\newtheorem{lemma}[theorem]{Lemma}

\theoremstyle{definition}

\newtheorem{assumption}[theorem]{Assumption}
\theoremstyle{remark}
\newtheorem{remark}[theorem]{Remark}
\newcommand{\zvec}{\mathbf z}
\newcommand{\Pin}{\Pi_{\textrm{in}}}
\newcommand{\Pinx}{\Pi_{\textrm{in},\mathbf{x}}}
\newcommand{\Pout}{\Pi_{\textrm{out},\mathbf{x}}}

\newcommand{\cJ}{\mathcal J}

\newcommand{\innerprod}[1]{\left\langle #1 \right\rangle}
\newcommand{\Ntilde}{\tilde N}
\newcommand{\Ktilde}{\tilde K}

\newcommand{\bd}[1]{\cB(#1)}
\newcommand{\uveco}{\uvec^{\textrm{opt}}}
\newcommand{\tuveco}{\tilde{\uvec}^{\textrm{opt}}}
\newcommand{\cTinv}{\mathcal T^{-1}}
\newcommand{\normp}[1]{(\norm{ #1} + 1)}
\newcommand{\cJhat}{\hat\cJ}
\newcommand{\cT}{\mathcal T}
\newcommand{\cR}{\mathcal R}
\newcommand{\PX}{P_X}
\newcommand{\cG}{\mathcal G}
\newcommand{\cD}{\mathcal D}

\newcommand{\Stildein}{\tilde S_{\textrm{in}}}
\newcommand{\Stildeout}{\tilde S_{\textrm{out}}}
\newcommand{\zhat}{\hat z}

\newcommand{\R}{\mathbb R}
\newcommand{\HS}{\mathrm{HS}}
\newcommand{\A}{\mathbf A}
\newcommand{\B}{\mathbf B}
\newcommand{\K}{\mathbf K}

\newcommand{\xvec}{\mathbf x}

\newcommand{\Hspace}{\mathcal H}
\newcommand{\uvec}{\mathbf u}
\newcommand{\Xcal}{\mathcal X}

\newcommand{\tildezvec}{\tilde{\mathbf z}}
\newcommand{\wvec}{\mathbf w}

\newcommand{\Zx}{Z_{\xvec}}
\newcommand{\Su}{S_{\uvec}}
\newcommand{\Sx}{S_{\xvec}}

\newcommand{\Atilde}{\tilde A}
\newcommand{\Btilde}{\tilde B}

\newcommand{\Ptilde}{\tilde P}

\usepackage{xspace}

\newcommand\de{:=}

\newcommand\ie{i.e.\ }
\newcommand\wrt{w.r.t.\ }
\newcommand\irange[1]{\brk*{#1}}

\newcommand{\st}{\text{ s.t. }\xspace}

\newcommand{\tiid}{i.i.d.\ }

\newcommand\numberthis{\addtocounter{equation}{1}\tag{\theequation}}

\newcommand\restr[2]{{
  \left.\kern-\nulldelimiterspace 
  #1 
  \vphantom{\big|} 
  \right|_{#2} 
  }}

\DeclareMathOperator*{\argmin}{arg\,min}

\DeclareMathOperator*{\spa}{span}

	\usepackage{bm}
	\newcommand{\V}[1]{\bm{#1}} 

\DeclareMathOperator*{\ran}{ran}

\newcommand{\kron}{\otimes}

\newcommand{\vvec}[1]{\Bigr[\def\arraystretch{0.8}\begin{array}{cc}#1\end{array}\Bigr]}

\newcommand{\bN}{\mathbb{N}}

\newcommand{\bR}{\mathbb{R}}

\DeclarePairedDelimiter{\prt}{(}{)}
\DeclarePairedDelimiter{\brk}{[}{]}
\DeclarePairedDelimiter{\cb}{\{}{\}}
\let\norm\relax
\DeclarePairedDelimiter{\norm}{\lVert}{\rVert}
\DeclarePairedDelimiter{\n}{\lVert}{\rVert}
\DeclarePairedDelimiter{\ip}{\langle}{\rangle}

\DeclarePairedDelimiter{\nHS}{\lVert}{\rVert_{\textup{HS}}}

\usepackage{pgffor}
\foreach \x in {a,...,z}{
  \expandafter\xdef\csname V\x \endcsname{\noexpand\ensuremath{\noexpand\V{\x}}}
}
\foreach \x in {A,...,Z}{
  \expandafter\xdef\csname V\x \endcsname{\noexpand\ensuremath{\noexpand\V{\x}}}
  \expandafter\xdef\csname c\x \endcsname{\noexpand\ensuremath{\noexpand\mathcal{\x}}}
  \expandafter\xdef\csname f\x \endcsname{\noexpand\ensuremath{\noexpand\mathfrak{\x}}}
}

\DeclareUnicodeCharacter{03B1}{\alpha}  	 
\DeclareUnicodeCharacter{0391}{\Alpha}       
\DeclareUnicodeCharacter{03B2}{\beta}        
\DeclareUnicodeCharacter{0392}{\Beta}        
\DeclareUnicodeCharacter{03B3}{\gamma}       
\DeclareUnicodeCharacter{0393}{\Gamma}       
\DeclareUnicodeCharacter{03B4}{\delta}       
\DeclareUnicodeCharacter{0394}{\Delta}       
\DeclareUnicodeCharacter{03B5}{\epsilon}     
\DeclareUnicodeCharacter{0395}{\Epsilon}     
\DeclareUnicodeCharacter{03B6}{\zeta}        
\DeclareUnicodeCharacter{0396}{\Zeta}        
\DeclareUnicodeCharacter{03B7}{\eta}         
\DeclareUnicodeCharacter{0397}{\Eta}         
\DeclareUnicodeCharacter{03B8}{\theta}       
\DeclareUnicodeCharacter{0398}{\Theta}       
\DeclareUnicodeCharacter{03B9}{\iota}        
\DeclareUnicodeCharacter{0399}{\Iota}        
\DeclareUnicodeCharacter{03BA}{\kappa}       
\DeclareUnicodeCharacter{039A}{\Kappa}       
\DeclareUnicodeCharacter{03BB}{\lambda}      
\DeclareUnicodeCharacter{039B}{\Lambda}      
\DeclareUnicodeCharacter{03BC}{\mu}          
\DeclareUnicodeCharacter{039C}{\Mu}          
\DeclareUnicodeCharacter{03BD}{\nu}          
\DeclareUnicodeCharacter{039D}{\Nu}          
\DeclareUnicodeCharacter{03BE}{\xi}          
\DeclareUnicodeCharacter{039E}{\Xi}          
\DeclareUnicodeCharacter{03BF}{\omega}       
\DeclareUnicodeCharacter{039F}{\Omega}       
\DeclareUnicodeCharacter{03C0}{\pi}          
\DeclareUnicodeCharacter{03A0}{\Pi}          
\DeclareUnicodeCharacter{03C1}{\rho}         
\DeclareUnicodeCharacter{03A1}{\Rho}         
\DeclareUnicodeCharacter{03C3}{\sigma}       
\DeclareUnicodeCharacter{03A3}{\Sigma}       
\DeclareUnicodeCharacter{03C4}{\tau}         
\DeclareUnicodeCharacter{03A4}{\Tau}         
\DeclareUnicodeCharacter{03C5}{\upsilon}     
\DeclareUnicodeCharacter{03A5}{\Upsilon}     
\DeclareUnicodeCharacter{03D5}{\phi}         
\DeclareUnicodeCharacter{03C6}{\varphi}      
\DeclareUnicodeCharacter{03A6}{\Phi}         
\DeclareUnicodeCharacter{03C7}{\xi}          
\DeclareUnicodeCharacter{03A7}{\Xi}          
\DeclareUnicodeCharacter{03C8}{\psi}         
\DeclareUnicodeCharacter{03A8}{\Psi}         
\DeclareUnicodeCharacter{03C9}{\omega}       
\DeclareUnicodeCharacter{03A9}{\Omega}       
\DeclareUnicodeCharacter{1D62}{_i}       	 %

\DeclareUnicodeCharacter{2248}{\approx}      
\DeclareUnicodeCharacter{2264}{\leq}         
\DeclareUnicodeCharacter{2265}{\geq}         
\DeclareUnicodeCharacter{227C}{\preccurlyeq} 
\DeclareUnicodeCharacter{227D}{\succcurlyeq} 
\DeclareUnicodeCharacter{2286}{\subseteq}    
\DeclareUnicodeCharacter{2260}{\neq}   		 
\DeclareUnicodeCharacter{1D40}{^T}           
\DeclareUnicodeCharacter{2200}{\forall}      
\DeclareUnicodeCharacter{220A}{\in}          
\DeclareUnicodeCharacter{00B7}{\cdot}        
\DeclareUnicodeCharacter{00D7}{\times}       
\DeclareUnicodeCharacter{00B7}{\cdot}        
\DeclareUnicodeCharacter{2016}{\Vert}        
\DeclareUnicodeCharacter{2026}{\dots}        
\DeclareUnicodeCharacter{2020}{\dagger}      
\DeclareUnicodeCharacter{27C2}{\perp}        %
\DeclareUnicodeCharacter{2192}{\rightarrow}  
\DeclareUnicodeCharacter{21A6}{\mapsto}       
\DeclareUnicodeCharacter{21D4}{\iff} 		 
\DeclareUnicodeCharacter{211D}{\mathbb{R}}   
\DeclareUnicodeCharacter{221E}{\infty}   
\DeclareUnicodeCharacter{00B2}{^2}           
\DeclareUnicodeCharacter{2080}{_0}           
\DeclareUnicodeCharacter{2081}{_1}           
\DeclareUnicodeCharacter{2082}{_2}           
\DeclareUnicodeCharacter{2083}{_3}           
\DeclareUnicodeCharacter{2084}{_4}           
\DeclareUnicodeCharacter{2085}{_5}           
\DeclareUnicodeCharacter{2086}{_6}           
\DeclareUnicodeCharacter{2087}{_7}           
\DeclareUnicodeCharacter{2088}{_8}           
\DeclareUnicodeCharacter{2089}{_9}           

\DeclareUnicodeCharacter{2212}{--}           

\newcommand{\nnew}[1]{{\color{black}#1}}

\usepackage{xparse}

\newcommand\CL{\cL}
\newcommand\CLHo{\CL(\cH₁)}

\newcommand{\kr}{k} 					
\newcommand\supk{κ²}
\newcommand\supfmap{κ}
\newcommand{\G}{G_γ}				
\newcommand{\nG}{\tilde G_γ}				
\renewcommand{\A}{A_γ}						
\renewcommand{\B}{B_γ}						
\renewcommand{\K}{K}						
\newcommand{\nA}{\tilde A_γ}				
\newcommand{\nB}{\tilde B_γ}				
\newcommand{\nK}{\tilde K}				
\newcommand{\nN}{\tilde N}				
\newcommand{\nP}{\tilde P}				

\newcommand{\Bout}{U_{\textup{out},\mathbf{x}}}	
\newcommand{\Hout}{\tilde\cH_{\textup{out}}} 
\newcommand{\Hin}{\tilde\cH_{\textup{in}}} 
\newcommand{\Kmout}{K_{m,\textup{out}}} 

\newcommand\rfK{H_γ}	
\newcommand\ldm[1]{\tilde{x}_{#1}}	

\DeclareDocumentCommand\ldmi{g}{\tilde{\mathbf{x}}^{\text{in}}\IfNoValueF{#1}{_{#1}}}	
\DeclareDocumentCommand\ldmo{g}{\tilde{\mathbf{x}}^{\text{out}}\IfNoValueF{#1}{_{#1}}}	

\title{\bfseries{Linear quadratic control of nonlinear systems with Koopman operator learning and the Nystr\"om method}}
\usepackage{authblk}

\author[1,5]{Edoardo Caldarelli}
\author[3,6]{Antoine Chatalic}
\author[1]{Adri\`a Colom\'e}
\author[3]{Cesare Molinari}
\author[1,2]{\\Carlos Ocampo-Martinez}
\author[1]{Carme Torras}
\author[3,4,5]{Lorenzo Rosasco}

\affil[1]{Institut de Rob\`otica i Inform\`atica Industrial, CSIC -- UPC, Barcelona, Spain}  
\affil[2]{Automatic Control Department (ESAII), Universitat Politècnica de Catalunya - BarcelonaTECH, Spain}  
\affil[3]{MaLGa Center -- DIBRIS -- Universit\`a di Genova, Genoa, Italy}             
\affil[4]{CBMM -- Massachusets Institute of Technology, Cambridge, MA, USA}        
\affil[5]{Istituto Italiano di Tecnologia, Genoa, Italy}        
\affil[6]{CNRS, Univ. Grenoble-Alpes, GIPSA-lab, France} 
\affil[ ]{\textbf{Correspondence to:} \texttt{edoardo.caldarelli@iit.it}}
\date{}
\begin{document}
	
	\maketitle
	
	\begin{abstract}        
		In this paper, we study how the Koopman operator framework can be combined with kernel methods to effectively control nonlinear dynamical systems. While kernel methods have typically large computational requirements, we show how random subspaces (Nystr\"om approximation) can be used to achieve  huge computational savings while preserving accuracy.
		Our main technical contribution  is deriving theoretical guarantees on the effect of the Nystr\"om approximation.  More precisely, we study the linear quadratic regulator problem, showing that the approximated Riccati operator \new{converges at the rate $m^{-1/2}$}, and the regulator objective, for the associated solution of the optimal control problem, \new{converges at the rate $m^{-1}$}, where $m$ is the random subspace size. Theoretical findings are complemented by numerical experiments corroborating our results.

	\end{abstract}
	


\begin{figure*}
	\centerline{
		\begin{tikzpicture}[auto, node distance=2cm,>=latex']
			
			\node [chart_node_algorithm] (Simulate) {Nonlinear dynamical system $\mathbf x_{k+1} = f(\mathbf x_k, \mathbf u_k)$.};

			\node [chart_node_algorithm, right= of Simulate,xshift=-1.5cm] (Lift) {Lift system to RKHS, $\infty$-dim.\ system $z_{k+1} = \G \begin{bmatrix}z_k\ \mathbf u\end{bmatrix}^T$.};
			
			\node [chart_node_algorithm, right= of Lift,xshift=-1.5cm] (Sketch) {Finite-dim.\ approximation with Nyström method, $\tilde z_{k+1} 
				=  \nG \begin{bmatrix}\tilde z_k\  \mathbf u\end{bmatrix}^T$.};
			
			\node [chart_node_lqr, below=of Sketch, yshift=1.5cm] (Identify) {\textbf{System identification:} Learn $\nG$ with sketched kernel ridge regression.};
			
			\node [chart_node_lqr, left= of Identify, xshift=1.5cm] (LQR) {\textbf{Control:} Solve the LQR problem with $\nG$.};	
			
			\node [chart_node_theory, below = of Identify,yshift=1.6cm] (Convergence_dynamics) {Fix training set. For any regularization $\gamma$, $\lVert \G - \nG\rVert \lesssim \mathcal O(1/\sqrt m)$.};
			
			\node [chart_node_theory, below= of LQR,yshift=1.6cm] (Convergence_LQR) {The optimal control with $\nG$ is close to the one we would obtain with $\G$.};
			
			\node (A) [left= of Simulate]  {};
			\node (B) [below=of A.west, anchor=west]  {};
			\node (C) [below = of B.west, anchor=west]  {};
			
			\node[fit= (A) (Simulate) (Lift) (Sketch), rounded corners, draw, dashed, inner sep=0.05cm,align=left, very thin] (modelling){\textbf{Modelling}};
			\node[fit= (B) (Identify) (LQR), rounded corners, draw, dashed, inner sep=0.05cm, align=left, very thin] (learning){\textbf{Learning \& control}};
			\node[fit= (C) (Convergence_dynamics) (Convergence_LQR), rounded corners, draw, dashed, inner sep=0.05cm, align=left, very thin, yshift=0.1cm] (theory){\textbf{Guarantees}};

			
			\draw [->] (Simulate) -- (Lift);
			\draw [->] (Lift) -- (Sketch);
			\draw [->] (Sketch) -- (Identify);
			\draw [->] (Identify) -- (LQR);
			\draw [->] (LQR) -- ++ (-4,0) -| node [pos=0.99] {} (Simulate) node[midway, xshift=2.9cm, yshift=-0.2cm] (m) {Optimal gain for state feedback};
			
			\draw [->] (Identify) -- (Convergence_dynamics);
			\draw [->] (LQR) -- (Convergence_LQR);
			\draw [->] (Convergence_dynamics) -- (Convergence_LQR);
			\draw [dashed, ->, color=Crimson] (Lift) -- (LQR) node[midway] (m) {\textcolor{red}{{Inefficient}}};
			
	\end{tikzpicture}}
	\caption{Summary: given some controls and corresponding state trajectories of a nonlinear dynamical system, we use kernels 
	to build a linear, data-driven model of the system. Kernels yield \new{a computationally inefficient} representation of the state space, \new{due to the inversion of the kernel matrix}, which we render \new{computationally tractable} using the Nyström method. 
	}
	\label{fig:summary}
\end{figure*}
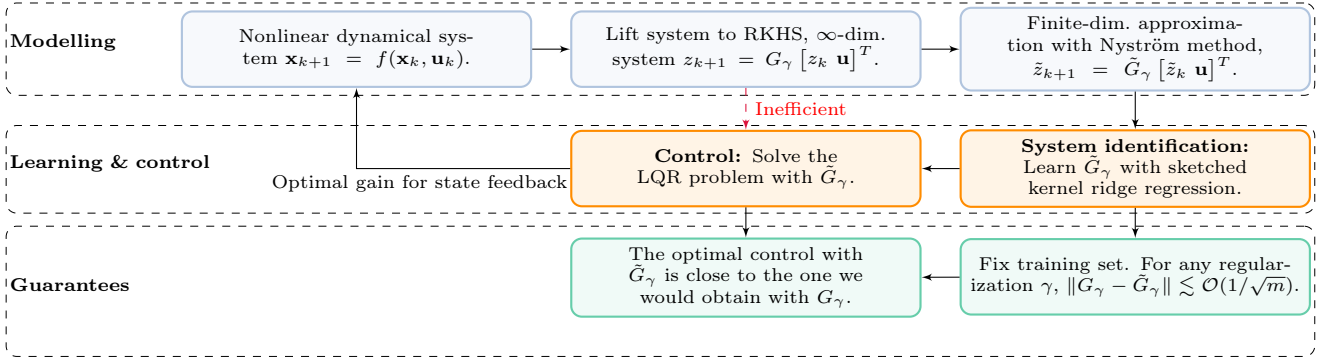
%
\section{Introduction}
Nonlinear dynamical systems are ubiquitous, and pose a great challenge  in terms of system identification and control.  
A powerful approach to deal with nonlinear dynamical systems is provided by the \emph{Koopman operator} framework~\cite{koopman1931hamiltonian,bevanda2021koopman,mezic2021koopman,brunton2022modern}.
In this approach, the nonlinear dynamical system is transformed through a set of nonlinear functions, called \emph{observables}, so that the dynamics of the transformed states are \emph{linear}, 
and can be used to reconstruct the states of the system. 
The Koopman operator, at the basis of this technique, was originally introduced in the context of autonomous dynamical systems~\cite{koopman1931hamiltonian}. However, as shown in the seminal work of Korda et al.~\cite{korda2018linear}, the Koopman operator formulation allows to apply linear control techniques, which are well understood and efficient to compute.
As discussed by~\cite{brunton2022modern,otto2021KoopmanOperatorsEstimation}, the control input can be included in the Koopman framework by either defining a family of Koopman operators (one for each value of the control input)~\cite{peitz2019koopman,nuske2023finite}, or by directly extending the state space to include the control input as an additional state of the system.  The latter perspective is the one we consider in this paper.
A key challenge  to apply the Koopman operator approach is to choose the suitable space of observable 
functions, both with and without  considering the control input ~\cite{korda2018linear}. 
Common choices include splines~\cite{korda2018linear}, polynomial or Fourier bases~\cite{abraham2017model}, and neural networks~\cite{shi2022deep,yin2022embedding,hao2024deep}, see also ~\cite{korda2020optimal,kaiser2021data,gibson2022koopman}. 

In this paper, we consider observables in a \emph{reproducing kernel Hilbert space} (RKHS)  \cite{aronszajn1950theory}.  Band limited functions, splines and Sobolev spaces are then special cases~\cite{berlinet2011reproducing}.  This choice was already mentioned in~\cite{korda2018linear} and has been recently  analyzed in 
\cite{klus2020kernel,kostic2022learning,das2020koopman,khosravi2023RepresenterTheoremLearning,giannakis2021learning,philipp2024error,bevanda2024koopman}
 for the  forecasting and analysis of dynamical systems. 
Kernel methods are popular in machine learning \cite{scholkopf2002learning}, as an RKHS has the advantage of being a possibly infinite dimensional space and corresponds to universal approximators \cite{steinwart2008support}, while the associated estimators are nonparametric \cite{wasserman2006all} and their computation reduces to  finite-dimensional numerical problems \cite{scholkopf2001generalized}. 
When efficiency is needed, further approximations are however required.  
In this paper,  we use the so-called Nystr\"om method,
\new{ which can be interpreted as a dimensionality reduction technique. More precisely, we approximate the dynamics of the observable functions in the RKHS using projections on random and data-dependent finite-dimensional subspaces} of functions \cite{williams2000using,nystroem1930UeberPraktischeAufloesung}. 
The effect of this approximation has been characterized for supervised machine learning, see, e.g., \cite{rudi2015less,musco2017RecursiveSamplingNystrom} and references therein, and more recently  for dynamical system identification, see~\cite{degennaro2019ScalableExtendedDynamic,meanti2023EstimatingKoopmanOperators}. \new{Besides the Nyström method, as an alternative kernel approximation, Nüske et al.\ \cite{nuske2023efficient} study the case of random Fourier features for dynamical system identification.}

 In this paper, we combine the above ideas and develop an efficient and accurate  control approach for nonlinear dynamical systems, based on using the Koopman operator framework together with kernels and the Nystr\"om method. Our main technical contribution is the analysis of the approximation due to the Nystr\"om method but,  unlike previous works,  we consider a control setting. 
We focus on the linear quadratic regulator (LQR) method, which has appealing theoretical properties, such as an analytic form of the optimal solution, and easily allows to deal with multi-input-multi-output systems.   \new{Our main
	result studies the impact of the Nyström approximation on the optimal control problem. Namely, in a fixed design analysis, we
	quantify the error introduced by the Nyström approximation to an empirical, kernel-based estimator of the nonlinear dynamics. We further show how this error propagates to the solution of the optimal control problem, if such an empirical estimator is employed in an LQR.} Instrumental to our analysis are recent results on  learning and controlling \emph{linear} systems \cite{simchowitz2018learning,mania2019certainty,dean2020sample}.  
We complement our theoretical analysis with some numerical experiments. 
The combination of the Koopman approach and the LQR~\cite{moyalan2023data}  has been applied to challenging robot learning problems~\cite{abraham2019active,yin2022embedding}, also involving soft robots~\cite{haggerty2023control}. Here, we assess the proposed control pipeline on a classic control benchmark, i.e., the Duffing oscillator \cite{korda2020optimal}, and on the compelling problem of identifying the dynamics of cloth \cite{coltraro2022inextensible,luque2022model,amadio2023controlled,zheng2022mixtures}.
To conclude, we observe that concurrent to our work, Driessen~\cite{driessen2023koopman} also considers kernel-based methods and their Nyström approximation for control purposes, but without any theoretical guarantees, and learning a bilinear~\cite{korda2018linear, moyalan2023data} data-driven system. 

As shown in Fig.~\ref{fig:summary}, our contributions are the following:
\begin{itemize}[topsep=0pt]
	\item we show how the Koopman operator framework for controlled dynamical systems can be used to design linear control by lifting the state-space representation to an RKHS;
	\item we show how the Nyström method can be exploited to derive efficient computations;
	\item we prove finite sample bounds for the convergence of the Nyström data-driven system to the infinite dimensional state-space representation in RKHS;
	\item when using an LQR, we 
	show how these error rates due to the Nyström approximation translate to the associated Riccati operator, and to the optimal control sequence computed by solving the LQR;
	\item we provide a publicly available and open-source implementation of the proposed Nyström-based system identification and optimal control algorithms\footnote{Available at \url{https://github.com/LCSL/nys-koop-lqr}.\looseness=-1} that we test on some illustrative examples.
\end{itemize}
{The structure of this paper is as follows. 
In Section~\ref{s:background}, we report the main technical background of our work.
In Section~\ref{s:koopman_id}, we discuss how to combine reproducing kernels and the Nyström method to obtain linear predictors for nonlinear systems. In Section~\ref{s:kernel_lqr} we discuss how to use such a predictor in an LQR. In Section~\ref{s:theory}, we present the core theoretical contributions of our work, while Section~\ref{s:experiments} contains a numerical evaluation of our identification algorithm and the associated LQR problem.}

\section{Background and notation}
\label{s:background}
 \nnew{In this paper, we consider 
 	a nonlinear controlled discrete dynamical system, 
 	which is approximated for control purposes by a surrogate dynamical system whose dynamics are linear both in the lifted state and in the input variable. This type of heuristic approximation, previously considered for instance by Korda et al.~\cite{korda2018linear}, has been empirically shown to be suited to nonlinear systems that are \emph{affine} in the control input, and have no coupled terms between state and input variables. While this may be seen as a restriction of the class of dynamical systems that can be modeled with our approximation, we can observe that it is already a fairly general model, see, e.g., the dynamical systems we study in Section 6, or the soft robot studied in \cite{haggerty2023control}.}
 

 We denote the state vector as $\mathbf x\in \mathbb R^d$, the output as $\mathbf y\in \mathbb R^{n_y}$, and the control input vector as $\mathbf u\in\mathbb R^{n_u}$. 
 For a sequence of control inputs $(\uvec_t)_{t\in\mathbb N}$, an initial state $\xvec_0$, and a function $f:\R^{d + n_u}\to\R^d$, we are interested in studying 
 deterministic nonlinear dynamical systems described by the difference equation
\begin{equation}
	\mathbf x_{t + 1} = f(\mathbf x_t, \mathbf u_t)\label{eq:controlled_dynsys},
\end{equation}
where $t \in\mathbb N$ denotes the discrete-time instant. \new{As mentioned before, the nonlinear map is considered to be affine in the control variable, and without coupled terms, i.e., for $g:\R^d\to\R^d$, $\mathfrak B:\R^{n_u}\to \R^{d}$,
$
	f\nnew{(\xvec, \uvec)} = g(\xvec) + \mathfrak B \uvec.
$
}The state and the input can be merged in an augmented \emph{finite-dimensional} state, denoted by $\wvec = [\mathbf x^T, \mathbf u^T]^T\in \R^{d + n_u}$. We will use both notations interchangeably throughout the article.
For a suitable space $\cA$ of observable functions $\xi$, the Koopman operator $\cK:\cA\to\cA$ can be defined as
\begin{align}
	(\mathcal K\xi)(\wvec_t) 
	&=\xi(\wvec_{t+1}), 
	\quad \forall \xi∊\cA. \label{eq:koopman_op_with_time}
\end{align}
\new{%
The definition in \eqref{eq:koopman_op_with_time} follows the Koopman with inputs and control (KIC) formalism~\cite{proctor2018generalizing}.
Note that when the control is prescribed by a state-feedback gain, this definition is related to the standard definition of the Koopman operator for the associated autonomous dynamics.
An alternative definition of the Koopman operator involving sequences of controls propagated using the shift operator has also been proposed~\cite{korda2018linear}. 
However, all definitions of the Koopman operator with control suffer from a common drawback, which is that they implicitly assume that a space of observable $\cA$ invariant under the action of the Koopman operator indeed exists.
}


As we will discuss, a fundamental component of our learning framework is the notion of RKHS \cite{aronszajn1950theory}. For an input space $\Xcal$, an RKHS $H$ is a Hilbert space of scalar functions on $\cX$ for which there exists a $k:\Xcal \times \Xcal \rightarrow \R$, the \emph{reproducing kernel}, so that, for \new{any $\chi \in \Xcal$ and $f \in H$, it holds} $k(\chi, \cdot)\in H$ and $f(\chi) = \innerprod{f, k(\chi, \cdot)}_{H}$. The latter notation denotes the inner product in the RKHS. An RKHS is a potentially infinite-dimensional space with universality properties.
Some examples of reproducing kernels are given by the Gaussian kernel, or by the Matérn family of kernels, which can be linked to Sobolev spaces. 
Note that the Koopman operator is linear in the observable $ξ$. It can thus be seen as a convenient way to build a linear approximation of the original dynamical system, 
at the cost of manipulating the lifted state $ξ(\mathbf w)$.
\new{
In this paper, we take inspiration from the Koopman formalism in order to devise a data-driven approximation of the dynamics in the RKHS suitable to use with an LQR.
We focus on the error induced by the use of the Nyström approximation, however we do not relate the learned dynamics in the RKHS to the definition of the Koopman operator with control in \eqref{eq:koopman_op_with_time}. 
}

\paragraph*{Notations} 
In the following, we use the notation $\n{·}$ for the operator norm. We denote the Euclidean norm of a vector as $\norm{\cdot}_2$. 
For any Hilbert space $\mathbb H$ we denote $\HS(\mathbb H)$ the Hilbert space of Hilbert-Schmidt operators on $\mathbb H$ and $\nHS{·}$ the associated Hilbert-Schmidt norm.
We denote $A^*$ the adjoint of an operator $A$, and $A^\dagger$ the pseudo-inverse of $A$. Besides, $\sigma_{\textrm{min}}(A)$ denotes the smallest singular value of $A$. 
\section{Koopman system identification}

\label{s:koopman_id}

As discussed previously, we focus on approximations of the original dynamical system which are linear
in the control input. 
In the following, we will show how to model such a linearity in the control input. We will also detail the regression problems that are solved in order to retrieve the data-driven, kernel-based dynamics, and show how such kernel-based dynamics can be transformed to vector-valued dynamics by leveraging the Nyström method.

\subsection{Choosing the Koopman lifting function}
	In order to define a suitable state and control transformation, we consider the RKHS $\cH₁$ associated to a stationary positive definite kernel $k:\R^d\times\R^d\to\R$. We let $\psi: \R^d\rightarrow \cH_1$, $\psi(x)\de k(x,·)$.	We define $\cH \coloneqq \cH_1\times \R^{n_u}$
	and choose as our state and input transformation
	\begin{align}
		\phi:\R^{d+n_u}\rightarrow\cH,\ 
		\phi(\wvec) &\coloneqq {\vvec{\psi(\xvec)\\  \uvec}}.
		\label{e:def_phi}
	\end{align}
	The choice of the stationary kernel allows to lift the data through an infinite-dimensional nonlinear transformation. Moreover, $\phi$ is linear in the control input variable, which allows to use linear control techniques on the system of interest. However, $\phi$ is infinite-dimensional for many standard choices of kernel functions. 
	We will show in Section~\ref{s:nystrom}
	that a related finite-dimensional lifting function can be obtained by
	using a Nyström approximation of the kernel $k$.

\subsection{Regression problem and corresponding solutions}  

We now explain how the function $\phi$ introduced in \eqref{e:def_phi}  
can be leveraged
to obtain a linear, data-driven surrogate dynamical system to be used in place of the original nonlinear one when designing the control law. The difference equation of this data-driven model can be estimated by least-squares regression, as we detail in the following.\looseness=-1

	Having a dataset of $n$ training pairs\footnote{\new{We use this notation for simplicity, however in practice one could also put together samples obtained by sampling multiple different trajectories.}} $((\wvec_i, \wvec_{i+1}))_{i=1,\dots,n}$ including control inputs and corresponding states, \new{\ie $\wvec_i\de[\xvec_iᵀ, \uvec_iᵀ]ᵀ$}, we define the sampling operators for the system's state ($\Sx$) and control input ($\Su$) as
	\begin{align}
		S_{\xvec} &:\Hspace_1 \rightarrow\R^n,\
		S_{\xvec} l\new{\de} \tfrac{1}{\sqrt{n}} [l(\xvec_1), \dots, l(\xvec_n)]^T,
		\label{e:def_Sx} \\
		S_{\mathbf u} &:\R^{n_u} \rightarrow\R^n,\
		S_{\uvec} \mathbf u \new{\de} \tfrac{1}{\sqrt{n}}[\langle \mathbf u_1, \mathbf u\rangle, \dots,\langle \mathbf u_n, \mathbf u\rangle]^T.
		\label{e:def_Su} 
	\end{align}%
	\new{The operator $\Sx$ returns a renormalized vector of the evaluations of its input $l$ at the points $(\xvec_i)_{1≤i≤n}$, while the operator $\Su$ corresponds to sampling the linear function $\ip{\mathbf u,·}$ associated to its input at the locations $(\mathbf u_i)_{1≤i≤n}$.}  
	Note that with these definitions, we can define the following compound sampling operator:
	\begin{equation}
		S:\cH_1\times\R^{n_u}\rightarrow\R^n,\ 
		S{\vvec{l\\ \uvec}}  = S_{\xvec} \new{l}+ S_{\uvec} \new{\uvec}.\label{e:def_S}
	\end{equation}
	Moreover, as discussed by~\cite{korda2018linear}, when considering Koopman operator regression for controlled systems, the regression output can be limited to be the one-step-ahead state, i.e., we are not interested in forecasting the evolution of the control input variable. Thus, the sampling operator \new{for the output training points $(\xvec_2,…,\xvec_{n+1})$} can be written as
\begin{equation}
\Zx  :\Hspace_1 \rightarrow\R^n,\
\Zx l = \tfrac{1}{\sqrt{n}} [l(\xvec_{2}), \dots, l(\xvec_{n+1})]^T.
\label{e:def_Zx}
\end{equation}
\paragraph*{Empirical risk minimization} 
The dynamics can be estimated by solving the following 
problem \nnew{over the set of linear operators from $\cH$ to $\cH_1$}:
\begin{align}
	\label{eq:koopman_regression_full}
	\G	\de \arg\min_{W:\mathcal H\rightarrow\mathcal H_1}\mathcal R(W) + \gamma\lVert W\rVert^2_{\HS}
\\\text{where}\quad
	\cR(W) \de \frac1n \sum_{i=1}^{n}\n*{ \psi(\xvec_{i+1})-W ϕ(\wvec_i)}_{\cH_1}² 
	\label{e:def_risk}.
\end{align}
Here $γ>0$ is a regularization parameter, and we recall that $ϕ$ is defined in \eqref{e:def_phi}. Note that this regression problem is not strictly speaking a Koopman regression problem, as the regression input and output spaces are different.
The objective in \eqref{eq:koopman_regression_full} is \nnew{continuous, coercive and strictly convex}, and thus admits a unique \nnew{minimizer}. 
As shown in Appendix~\ref{a:risk}, the risk function can be rewritten as
$\cR(W)= \nHS{\Zx - SW^*}²$ and thus the solution of \eqref{eq:koopman_regression_full} can be expressed as 
\begin{align}
	 \G &= \Zx^*(SS^* + \gamma I)^{-1}S.
	 \label{e:def_G}
\end{align}
\nnew{Although we started the exposition with the definition of the Koopman operator for simplicity, the operator $\G$ is more similar to a regularized conditional mean embedding~\cite{muandet2017KernelMeanEmbedding} or an embedded Perron-Frobenius operator~\cite{klus2020EigendecompositionsTransferOperators}.}

For any initial state $\xvec_0∊\bR^d$, this operator defines the following linear dynamics \nnew{in $\cH_1$}:
\begin{align}[left = \empheqlbrace\,]
	z_0 &= \psi(\xvec_0),
	\label{e:initial_state}\\
	z_{t+1} &=  \G {\vvec{z_t\\ \uvec_t}}.
	\label{eq:dynamics_compound}
\end{align}
\paragraph*{Affine dynamics} 
Note that the operator $\G:\cH→\cH₁$ defined in \eqref{eq:koopman_regression_full} can be decomposed in two operators $\A:\cH₁ → \cH₁$ and $\B:\R^{n_u}→\cH₁$, controlling respectively the parts of the dynamics due to the state and to the control input.
More precisely, defining
\begin{align}
	\A &= \Zx^*(SS^* + \gamma I)^{-1} S_{\xvec},
		\label{e:def_A}\\
	\B &= \Zx^*(SS^* + \gamma I)^{-1} S_{\uvec},
		\label{e:def_B}
\end{align}
the dynamics \eqref{eq:dynamics_compound} are equivalent to the autoregressive linear model
\begin{align}
	z_{t+1} &= \A z_t+ \B \uvec_t
	\label{eq:dynamics}.
\end{align}

\subsection{Nyström approximation} 
\label{s:nystrom}

Given that the lifted state $z$ is typically infinite-dimensional,
we are now interested in designing a finite-dimensional approximation of the dynamics in \eqref{eq:dynamics_compound}, which would be more useful for practical control purposes.
In this section, we thus approximate the nonlinear kernel $\kr$ using a Nyström approximation~\cite{williams2000using}. The approximation is based on the choice of two sets $\ldmi{1}, …,\ldmi{m}$ and $\ldmo{1},…,\ldmo{m}$ of $m$ points, called the input and output landmarks.
Multiple approaches have been studied for landmark selection, but in this paper we consider the simple setting where the landmarks are either all drawn uniformly from the dataset, or the input landmarks are drawn uniformly from the dataset and the output landmarks are then taken one step ahead in time \wrt the input ones. 
 Now, let us define
	$\Hin \de \spa\{\psi(\ldmi{1}), …, \psi(\ldmi{m})\}$, 
	$\Hout\de\spa\{\psi(\ldmo{1}), \dots, \psi(\ldmo{m})\}$, 
and let $\Pinx:\Hspace_1→\Hspace_1, \Pout:\Hspace_1\rightarrow\Hspace_1$ be the orthogonal projectors, respectively onto $\Hin$ and $\Hout$.
Moreover, we define $\Pin:\Hspace\rightarrow\Hspace$ 
as $\Pin \phi(\wvec) = \vvec{\Pinx \psi(\xvec) \\ \uvec}$; as the control variable is already finite-dimensional, we indeed only need to project the lifted state in order to obtain a finite-dimensional approximation. 
Following~\cite{meanti2023EstimatingKoopmanOperators}, we define the Nyström approximation of $\G$ as follows:
\begin{align}
	\nG
	&\de \arg\min_{W:\cH→\cH_1} \cR(\Pout W\Pin) + γ\nHS{W} ^2\nonumber\\
	&=\Pout\Zx^* (S\Pin S^* + \gamma I)^{-1} S\Pin.
	\label{eq:a_b_compound_expr_nystr}
\end{align}
The operator $\nG$ can be used to define linear dynamics approximating the ones from~\eqref{eq:dynamics},
namely for any initial condition $\xvec_0∊\bR^d$:
\begin{align}[left = \empheqlbrace\,]
	\tilde z_0 
		&= \Pout\psi(\xvec_0),\label{e:n_initial_state}\\
	\tilde z_{t+1} 
	&= \nG {\vvec{\tilde z_t\\ \uvec_t}}.
		\label{eq:nystrom_dynamics_compound}
\end{align}
The projection in the initial condition guarantees that all the states visited during the evolution of the system belong to $\Hout$. The dynamics in \eqref{eq:nystrom_dynamics_compound} are still linear in the control input, and the operator $\nG$ could be decomposed into operators $(\nA,\nB)$, approximating the operators $(\A,\B)$ defined in \eqref{e:def_A} and \eqref{e:def_B}.
We choose however to manipulate only $\nG$ in the following to keep expressions more concise.
\paragraph*{Vector-valued representation} 
The dynamics in~\eqref{eq:nystrom_dynamics_compound} are defined for evolving functions: despite being all restricted to a finite-dimensional subspace, the iterates $(\tilde z_t)_{t∊\bN}$ still belong to a functional space. 
In order to retrieve a vector-valued state representation that is practically computable, 
we will thus look at the dynamics of the coordinates of these evolving functions in a basis of the subspace to which they belong.
More precisely, we can define $\Stildeout$ as
\begin{equation}\medmath{
	\Stildeout:\Hspace_1\rightarrow\R^m,\
		\Stildeout g = [g(\ldmo{1}), …,g(\ldmo{m})]ᵀ.
		\label{e:def_Stout}}
\end{equation}
Note that $\new{\Kmout\de}\Stildeout\Stildeout^*\in\R^{m\times m}$ is the Gram matrix of the stationary kernel $\kr$ computed at the output Nyström centers $\tilde\xvec_i^{\textrm{out}}$.
Defining $\Bout=\Stildeout^* (\Kmout^†)^{1/2}$, it holds $\Bout\Bout^*=\Pout$. 

Note that for any initial condition of the form~\eqref{e:n_initial_state}, the Nyström states $\tilde z_t$ defined by \eqref{eq:nystrom_dynamics_compound} belong to $\Hout$ for all $t\in\mathbb N$.
Thus, 
\nnew{define $\tildezvec_t$ as the solution of the following} finite-dimensional autoregressive dynamics
\begin{align}[left = \hspace{-0.5cm}\empheqlbrace]
	\medmath{\tildezvec_0} 
		&\medmath{\de \Bout^*\psi(\xvec_0), }
		\label{e_n_vec_initial_state}
		\\
	\medmath{\tildezvec_{t+1} }
		&\medmath{\!=\! \Bout^*\Zx^* (S\Pin S^*\! +\! \gamma I)^{-1} S\Pin\!{\vvec{\Bout \tildezvec_t\\ \uvec_t}}\!.} 
	\label{eq:nystrom_dynamics_vector_partial}
\end{align}
\nnew{Then, for 
\begin{align}
	\tilde z_t
	&\de \Bout \tildezvec_t,
	\label{eq:parameterized_nystrom_dynamics}
\end{align}
$\tilde z_t$ fulfills \eqref{e:n_initial_state} and \eqref{eq:nystrom_dynamics_compound}.}

As we detail in Appendix~\ref{appendix:computable_expressions}, the dynamics \eqref{eq:nystrom_dynamics_vector_partial} can be rewritten in terms of matrix products that can be computed efficiently, and only requires to invert an $m×m$ matrix.
Note that the dynamics of $\tilde z$ could similarly be expressed in any orthonormal basis of $\Hout$, however working with $\Bout$ naturally yields expressions where the kernel matrix $\Kmout$ appears, which is convenient for implementation.
\paragraph*{State reconstruction} The lifted state $\tildezvec$ can be used to reconstruct the original state $\xvec$, as discussed in~\cite[Section 6]{brunton2022modern}. This goal can be achieved, \nnew{e.g.}, by using a least-squares estimate, or by augmenting the lifted state with the original one. \nnew{Here}, we consider a regularized least squares estimate for the state reconstruction matrix $C$. For a given regularization parameter $λ > 0$, we define
\begin{equation}
	C = \nnew{\argmin_{M\in\R^{d\times m}}}\new{\frac{1}{n}} \sum_{i=1}^{n}\norm{\xvec_{i+1} - M\tildezvec_{i+1}}_2^2 + λ \norm{M}_{\HS}^2.
	\label{eq:reconstruction}
\end{equation}
A closed-form expression for matrix $C$ is also included in Appendix~\ref{appendix:computable_expressions}. 
\new{Note that the training control variables appear in \eqref{eq:reconstruction} only through the definition of the dynamics, and it would also be possible to use instead an arbitrary dataset of sampled states.}

\section{Kernels and Koopman LQR}
\label{s:kernel_lqr}
Once we have estimated the state-space representation of the dynamical system of interest, we can use linear predictive control techniques, as the Koopman approach transforms a non-linear system in a linear one~\cite{korda2018linear}. 
In the following, we focus on the LQR, 
and more particularly on the infinite horizon LQR~\cite{mania2019certainty} 
due to its theoretically appealing properties. 
The key idea is that the optimization problem at the basis of the LQR is rendered tractable by using a finite dimensional embedding of the state and inputs, which can be achieved by sketching techniques, as discussed in Section~\ref{s:koopman_id}.  
\paragraph*{LQR for exact dynamics}  
We first consider using the exact kernel $\kr$ and the transformation $ϕ$. 
We consider a control objective that is quadratic in the lifted state and control inputs, via the weighting operators $Q: \mathcal H_1\rightarrow\mathcal H_1$ and $ R:\R^{n_u}\rightarrow\R^{n_u}$. 
Then, the LQR strategy requires to solve the following optimal control problem, for which a time-invariant analytical solution is available: 
\begin{align}
	\label{eq:lqr_koopman_infinite_dimensional}
	\min_{\mathbf u_0, \uvec_1,\dots}\lim_{T\rightarrow \infty}&\sum_{i=0}^{T}\langle  z_i, Q z_i\rangle_{\mathcal H_1} +
	\langle \mathbf u_i,  R\mathbf u_i\rangle_{\R^{n_u}}\\
	\mathrm{s.t.\ }  &\eqref{e:initial_state}, \eqref{eq:dynamics_compound}.
	\nonumber
\end{align}
\paragraph*{LQR for approximated dynamics} 
When the dynamics are approximated by using the Nyström approach as in \eqref{eq:nystrom_dynamics_compound}, the optimal control problem becomes
\begin{align}
	\label{eq:lqr_koopman_nystrom_infinite_dimensional}
	\min_{\mathbf u_0, \uvec_1,\dots}\lim_{T\rightarrow \infty}\sum_{i=0}^{T}&\langle  \tilde z_i, Q \tilde z_i\rangle_{\mathcal H_1} + 
	\langle \mathbf u_i,  R\mathbf u_i\rangle_{\R^{n_u}}\\
	\nonumber
	\mathrm{s.t.\ }  &\eqref{e:n_initial_state}, \eqref{eq:nystrom_dynamics_compound}.
\end{align}
According to~\eqref{eq:parameterized_nystrom_dynamics}, the problem in~\eqref{eq:lqr_koopman_nystrom_infinite_dimensional} can equivalently be rewritten in the basis $\Bout$ for weighting matrices $\tilde Q\in \R^{m\times m}$,
\begin{equation}
\tilde Q=(\Kmout^\dagger)^{1/2}\Stildeout Q\Stildeout^*(\Kmout^\dagger)^{1/2},\label{eq:Q_definition}
\end{equation} and $R$, as follows:
\begin{align}
	\label{eq:lqr_koopman_nystrom_finite_dimensional}
	\min_{\mathbf u_0, \uvec_1,\dots}\lim_{T\rightarrow \infty}
		\sum_{i=0}^{T}&\langle  \tildezvec_i, \tilde Q \tildezvec_i\rangle_{\mathcal \R^m} 
	+ \ip{ \mathbf u_i,  R\mathbf u_i }_{\R^{n_u}}\\
	\nonumber
	\mathrm{s.t.\ } &\eqref{e_n_vec_initial_state}, \eqref{eq:nystrom_dynamics_vector_partial}.%
\end{align}
\nnew{Based on \eqref{eq:Q_definition}, problems \eqref{eq:lqr_koopman_nystrom_infinite_dimensional} and \eqref{eq:lqr_koopman_nystrom_finite_dimensional} are equivalent, and yield the same optimal control sequence in form of state feedback~\cite{hager1976convergence}}. Practically, one can choose $\tilde Q=C^*Q'C$ for some $Q':\R^d\to\R^d$, and $C$ is the reconstruction matrix defined in \eqref{eq:reconstruction} in order to define an objective that can be interpreted as a penalization of the states. Denoting the optimal gain resulting from~\eqref{eq:lqr_koopman_nystrom_infinite_dimensional} as $\tilde K: \cH_1\rightarrow \R^{ n_u}$, the optimal control law is 
$
	\uvec_k = \tilde K\tilde z_k.
$
From~\eqref{eq:parameterized_nystrom_dynamics}, we have that the state-feedback input can be rewritten as
$
	\uvec_k = \tilde K\Bout\tilde \zvec_k,
$
with $ \tilde K\Bout:\R^m\rightarrow \R^{n_u}$ being the solution of the LQR problem in~\eqref{eq:lqr_koopman_nystrom_finite_dimensional}. When dealing with the true nonlinear dynamics in Section~\ref{s:experiments}, we consider a state-feedback control law of the form
\begin{equation}
	\uvec_k = \tilde K\Pout\psi(\xvec_k),\label{e:ctrl_policy}
\end{equation}
where $\xvec_k$ is the true state of the system, to perform control in closed loop.
\nnew{\paragraph*{Overall pipeline} The content of this section can be summarized in a whole pipeline that can be used for linear identification and control of nonlinear systems. Such a pipeline consists of the following steps:
\begin{enumerate}[topsep=0pt]
	\item sample Nyström input and output landmarks $\ldmi{1}, …,\ldmi{m}$ and $\ldmo{1},…,\ldmo{m}$ uniformly from the training set;
	\item compute the operators in~\eqref{eq:nystrom_dynamics_vector_partial} to obtain a data-driven linear system;
	\item use these linear dynamics to solve problem~\eqref{eq:lqr_koopman_nystrom_finite_dimensional} and get the optimal state-feedback gain.
\end{enumerate} }
\section{Theoretical analysis}
\label{s:theory}
In this section, we perform a theoretical analysis of the proposed system identification method, and assess its effect on the LQR problem. In particular, we will compare the Nyström-based approach with the data-driven model based on the exact kernel. Overall, we show that the Nyström-based model of~\eqref{eq:nystrom_dynamics_compound} is a provably accurate approximation of the dynamics in~\eqref{eq:dynamics_compound}, that can be safely used for control purposes, in place of the intractable infinite-dimensional model of ~\eqref{eq:dynamics_compound}.
\paragraph*{Layout} In Section~\ref{s:hypotheses}, we state the assumptions at the basis of our analysis. In Section~\ref{s:bounds_A_B}, we bound the error, introduced by the Nyström approximation, on operators $\A$ and $\B$ from~\eqref{eq:dynamics_compound} (compound in the transition operator $\G$). In Section~\ref{s:bounds_riccati}, we show that the Riccati operator obtained with the Nyström dynamics is close to the one obtained with the exact kernel. Finally, in Section~\ref{s:convergence_lqr_objective}, we use the aforementioned results to show that, when plugging the optimal control from~\eqref{eq:lqr_koopman_nystrom_infinite_dimensional} into the dynamics of~\eqref{eq:dynamics_compound}, the LQR objective function is close to the one obtained when solving~\eqref{eq:lqr_koopman_infinite_dimensional} directly.
\subsection{Hypotheses}
\label{s:hypotheses}
In this section, we introduce the hypotheses of our theoretical derivations in the next sections.
\begin{assumption}[Bounded kernel]
	\label{a:bounded_kernel}
	The stationary kernel $k$ is bounded, i.e., \new{there exists a positive constant $\supfmap<∞$ such that} $k(\xvec, \xvec')\leq \supk$ \new{for any $\xvec,\xvec'∊\bR^d$}.
\end{assumption}

Note that under Assumption \ref{a:bounded_kernel}, it holds in particular
$
	\norm{\Zx} \leq κ $
	and$
	\norm{\Sx} \leq κ.
$
Indeed, for any $x$, it implies $\n{ψ(x)}=\sqrt{k(x,x)}≤κ$, and thus for any $g∊\cH_1$ with $\n{g}≤1$, it holds
\begin{align*}
	\n{\Zx g}_2
	&= n^{-1/2}\n{[\ip{g, ψ(\xvec_{2})}, \dots, \ip{g, ψ(\xvec_{n+1})}]ᵀ}_2\\
	 &≤ n^{-1/2}\left(\sum_{i=2}^{n+1} \n{g}²κ² \right)^{1/2} 
	 ≤ κ, 
\end{align*}
and a similar argument holds for $\Sx$.

The following assumptions are standard and guarantee that the LQR problems in~\eqref{eq:lqr_koopman_infinite_dimensional} and~\eqref{eq:lqr_koopman_nystrom_infinite_dimensional} are well-posed and admit an analytical solution in the form of a static state-feedback gain~\cite{hager1976convergence}. Minimal assumptions on $\A$, $\B$ from~\eqref{e:def_A} and~\eqref{e:def_B}, and $C$ from~\eqref{eq:reconstruction}, for stabilizability and detectability to hold, could be derived by leveraging the theoretical deployments reported, e.g., in~\cite{bensoussan2007representation}. 
\begin{assumption}[Stabilizability, \new{\cite{hager1976convergence}}]
	\label{a:stabilizability}
	The dynamical systems in~\eqref{eq:dynamics_compound} and~\eqref{eq:nystrom_dynamics_compound} are \emph{stabilizable}, i.e.,
	$
		\exists M:\Hspace_1\rightarrow \R^{n_u} \textrm{\ such that\ }{\rho(\A + \B M)<1}$, and $
		\exists \tilde M:\Hspace_1\rightarrow \R^{n_u} \textrm{\ such that\ }{\rho(\nA + \nB\tilde M)<1},\nonumber
	$
	where $\rho(W)$ is the spectral radius of operator $W$. 
\end{assumption}
Stabilizability assumes that the dynamical systems can be driven to zero with a suitable state-feedback gain. This is a weaker assumption than controllability (i.e., assuming that the system can be driven to any location via the choice of suitable feedback gain), that is often used in the analysis of the LQR algorithm~\cite{mania2019certainty}. However, the controllability assumption in $\cH_1$ is not satisfied by the Nyström dynamics, as the $\tilde z$ functions always live in the subspace $\Hout$, as we discussed in Section~\ref{s:nystrom}. 
\begin{assumption}[Detectability, \new{\cite{hager1976convergence}}]
	\label{a:detectability}
	The dynamical systems in~\eqref{eq:dynamics_compound} and~\eqref{eq:nystrom_dynamics_compound} are \emph{detectable}, i.e.,
	\begin{align}
		&\exists t,s,b,d\geq 0 \textrm{\ such that\ }\norm{\A^tz}_{\cH_1}\geq b\norm{z}_{\cH_1}\nonumber\\
		&\quad\Rightarrow \innerprod{z, \sum_{i=0}^s\A^{i*}Q\A^iz}_{\cH_1}\geq d\innerprod{z, z}_{\cH_1},\nonumber
	\end{align}
\begin{align}
		&\exists t,s,b,d\geq 0 \textrm{\ such that\ }\norm{\nA^t\tilde z}_{\cH_1}\geq b\norm{\tilde z}_{\cH_1}\nonumber\\
		&\quad\Rightarrow \innerprod{\tilde z, \sum_{i=0}^s\nA^{i*}Q\nA^i\tilde z}_{\cH_1}\geq d\innerprod{\tilde z, \tilde z}_{\cH_1}.\nonumber
	\end{align}
\end{assumption}
Detectability assumes that, if unstable dynamics happen in the linear system, these must be observed (and taken into account in the design of the LQR objective function).

\new{
\begin{assumption}
	Let $Q$, $R$ be the weights of the LQR problem \eqref{eq:lqr_koopman_infinite_dimensional}, and let $P$ be a solution of the discrete algebraic Riccati equation
	\begin{equation}
		 P =  A^*PA -  A^*PB( R + B^*  P B)^{-1}B^* P A+  Q.\label{e:dare_asm}
	\end{equation}
Then, $\sigma_{\textrm{min}}(P)\geq1$.
\end{assumption}
As discussed by Mania et al.~\cite{mania2019certainty}, this technical assumption can be fulfilled by re-scaling $Q$ and $R$ accordingly. Indeed, if $P$ is a solution of \eqref{e:dare_asm}, then $\eta P, \eta >0$ is also a solution, provided that $Q$ and $R$ are multiplied by $\eta$.}
\subsection{Accuracy of the Nyström approximation of the transition operator}
\label{s:bounds_A_B}

We now upper-bound the error (in operator norm) induced by the Nyström approximation on the transition operator $\G$.
Although other works studied sketched estimators of the Koopman operator,  
our bound notably differs from \cite{ahmad2023sketch,meanti2023EstimatingKoopmanOperators} who consider different norms (Hilbert-Schmidt and operator norm but on different spaces) and dynamical systems without control.

\begin{theorem}[Convergence rate for $\nG - \G$]\label{r:bound_G_nG_opnorm}
	Under Assumption~\ref{a:bounded_kernel}, for any $γ>0$ , it holds with probability $1-δ$ that
	\begin{align}
		\n{\nG-\G}
		&≤ \prt*{\frac{κ}{\nnew{γ}}+\new{\frac{1}{\gamma^{1/2}}}} 4\supfmap\sqrt{\frac{3}{m}\log\prt*{\frac{8m}{5δ}}} \nonumber\\
		&\quad	+ \frac{48\supfmap^3}{γ^{3/2}} \frac{1}{m}\log\prt*{\frac{8m}{5δ}}.
		\label{e:bound_G_nG_opnorm}
	\end{align}
\end{theorem}
\begin{proof}
	The proof for this result is provided in Appendix~\ref{appendix:nystrom_rate}.
\end{proof}

\color{black}

Note that, if we rewrite the approximate dynamics $\tilde z_{t+1} = \nG \phi(\wvec_t)$ from \eqref{eq:nystrom_dynamics_compound} in the autoregressive form as
\begin{align}
	\tilde z_{t+1} 
		&= \nA ψ(\xvec_t)+ \nB \uvec_t, 
	\label{e:autoregressive_dynamics}
\end{align}
Theorem~\ref{r:bound_G_nG_opnorm} automatically translates in bounds on the approximation of $\A$ and $\B$. Indeed, under the assumptions of Theorem \ref{r:bound_G_nG_opnorm}, the right-hand side of
	\eqref{e:bound_G_nG_opnorm} is also an upper bound on $\n{\A-\nA}$ and $\n{\B-\nB}$.
	Formally, the operators $\nA,\nB$ in \eqref{e:autoregressive_dynamics} can be defined from $\nG$ as $
		\nA 
			= \nG [I_{\cH_1}, 0_{\cH_1\rightarrow\bR^{n_u}}]^* : \cH_1\rightarrow \cH_1 $ and $
		\nB 
			= \nG [0_{\bR^{n_u}\rightarrow \cH_1}, I_{\bR^{n_u}}]^* : \bR^{n_u} \rightarrow \cH_1.
	$
	As a consequence, it holds $\n{\A-\nA}=\n{(\G-\nG) [I_{\cH_1}, 0_{\cH_1\rightarrow\bR^{n_u}}]^*}≤\n{\G-\nG}$, and a similar argument holds for $B$.

\renewcommand{\A}{A}						
\renewcommand{\B}{B}						
\renewcommand{\nA}{\tilde A}				
\renewcommand{\nB}{\tilde B}				
In the following, we will use for simplicity the notations $\A\de \A_γ,\B\de \B_γ,\nA\de \nA_γ,\nB \de \nB_γ$. However, all these operators implicitly depend on the choice of the regularization parameter $γ$.

\subsection{Convergence analysis for the Riccati operator}
\label{s:bounds_riccati}
In this section, we show that the Riccati operators for problems in~\eqref{eq:lqr_koopman_infinite_dimensional} and~\eqref{eq:lqr_koopman_nystrom_infinite_dimensional} are $\epsilon$-close in operator norm, provided that $\norm{\G - \nG}\leq \epsilon$. Note that, according to Theorem~\ref{r:bound_G_nG_opnorm}, we can set $\epsilon$ as function of $m$, namely
$
	\epsilon=\prt*{\frac{κ}{\nnew{γ}}+\new{\frac{1}{\gamma^{1/2}}}} 4\supfmap\sqrt{\frac{3}{m}\log\prt*{\frac{8m}{5δ}}} 
	+ \frac{48\supfmap^3}{γ^{3/2}} \frac{1}{m}\log\prt*{\frac{8m}{5δ}}.
$
This fact means that we transfer the guarantees for $\G$ from Section~\ref{s:bounds_A_B} to guarantees on the fundamental building block of the LQR solution, i.e., the Riccati operator.\looseness=-1

Let
\begin{align}
	\medmath{F(P, A, B)} &\medmath{= P - A^*[P - PB(R + B^* P B)^{-1}B^* P ]A}
	-Q\nonumber\\
	&\medmath{= P - A^*P(I + B R^{-1}B^*P)^{-1}A - Q.}\nonumber
\end{align}
Then, as proven by~\cite[Theorem 9]{hager1976convergence}, under the assumptions in Section~\ref{s:hypotheses}, the LQR problems in~\eqref{eq:lqr_koopman_infinite_dimensional} and~\eqref{eq:lqr_koopman_nystrom_infinite_dimensional} admit analytical solutions given by the static state-feedback gains
\begin{align}
	\K:\cH_1\to\ \R^{n_u},\ \K &= -(R+\BᵀP\B)^{-1}\B^* P\A  
		\label{e:def_K},\\
	\nK:\cH_1\to\ \R^{n_u},\ \nK &= -(R+\nBᵀ \nP \nB)^{-1} \nB^* \nP \nA 
		\label{e:def_nK},
\end{align}
where $P,\Ptilde:\cH₁→\cH₁$ are the unique self-adjoint, positive semi-definite operators obtained by solving the following discrete-time algebraic Riccati equations ~\cite{hager1976convergence}:
\begin{align}
	F(P,A,B)=0,
		\label{e:def_P}\\
	F(\nP,\nA,\nB)=0
		\label{e:def_nP}.
\end{align}
A fundamental step is now to derive an error rate for the solutions, upper bounding the quantity $\norm{P - \Ptilde}$. To do so, we can state and prove an analogous proposition to~\cite[Proposition 2]{mania2019certainty}, based on the fundamental results in~\cite{konstantinov1993perturbation}, and generalized to the case of operator dynamics.

\begin{lemma}[Convergence rate for $\Ptilde - P$]
	\label{lemma:convergence_p_operator}
	Let the systems in~\eqref{eq:lqr_koopman_infinite_dimensional} and~\eqref{eq:lqr_koopman_nystrom_infinite_dimensional} be stabilizable and detectable. Let $\sigma_{\textit{min}}(P)$ be the smallest singular value of the Riccati operator associated to~\eqref{eq:lqr_koopman_infinite_dimensional}, and let $\Ptilde$ be the Riccati operator for~\eqref{eq:lqr_koopman_nystrom_infinite_dimensional}. Assume $R$ is positive definite, and $\sigma_{\textrm{min}}(P) \geq 1$. Let $L=A + BK$, where $K$ is the stabilizing Riccati gain defined in~\eqref{e:def_K}, let $\rho(L)$ be the spectral radius of $L$, and let
	$
		\tau(L, \new{\zeta})=\sup\{\norm {L^k}\new{\zeta}^{-k}, k\geq 0\},\nonumber
	$
	with $\new\zeta$ being a constant such that $\rho(L) \leq\new\zeta < 1$.
	Then, for $\norm{A - \Atilde}\leq \epsilon$, $\norm{B - \Btilde}\leq \epsilon$, and
\begin{align*}
		\epsilon< \min\left\{\norm B,\right.&\left.\frac{1}{12} \frac{1}{\normp{L}^{2} +\normp{P}}\frac{(1 - \new\zeta^2)^2}{\tau(L,\new\zeta)^4}\right.\\
		&\cdot\normp A^{-2} \normp P^{-2} \\
		&\cdot\left.\normp B^{-3}\cdot\normp {R^{-1}}^{-2}\right\},
	\end{align*}
it holds that the error on the Riccati operator due to the Nyström approximation is upper bounded as follows:
\begin{align*}
	\norm{P-\tilde P}&\leq 6\epsilon \frac{\tau(L, \new\zeta)^2}{1 - \new\zeta^2}\normp A^2 \normp P^2 \\
	&\quad\cdot\normp B \normp {R^{-1}}.
\end{align*}
\end{lemma}
\begin{remark}
	When talking about stabilizability, we are considering the data-driven dynamics described by~\eqref{eq:dynamics_compound} and~\eqref{eq:nystrom_dynamics_compound}. This fact means that we assume these systems are stabilizable in $\cH_1$, not in the original state space $\R^{d}$.
\end{remark}
\begin{proof}
The proof of this result is close to the one of~\cite[Proposition 2]{mania2019certainty} which covers the setting where $A$ and $B$ are matrices. Some technical adjustments are required due to the fact that we work with operators. We report a sketch of the proof in Appendix~\ref{appendix:rates_riccati}.
\end{proof}
\subsection{Convergence analysis of the LQR objective function}
\label{s:convergence_lqr_objective}

In the previous subsections, we showed that $\norm{A - \Atilde}\leq \epsilon$, $\norm{B- \Btilde}\leq \epsilon$ for a sufficiently large number of Nyström landmarks $m$. We further showed that this implies that $\norm{P - \Ptilde}\leq \mathcal O(\epsilon)$. Now, we are interested in assessing how suitable the Nyström dynamics are for solving the LQR problem, instead of using the intractable, infinite-dimensional dynamics associated to an exact kernel. In order to do so, we first solve the LQR problem from \eqref{eq:lqr_koopman_nystrom_infinite_dimensional} based on the Nyström approximation of the dynamics. Then, we plug the optimal control sequence retrieved in this way in the exact kernel dynamics of~\eqref{eq:dynamics_compound}, and compare the LQR objective function with this control sequence, and the control sequence we would obtained by solving the problem in~\eqref{eq:lqr_koopman_infinite_dimensional}. This type of error analysis resembles the one in~\cite{mania2019certainty}, in the context of certainty equivalence. In order to do so, let us define 
two control sequences ($\tuveco_0,\tuveco_1, \dots)$, and ($ \uveco_0, \uveco_1, \dots)$.
 Moreover, for a given initial condition $\hat z_0 = z_0$, let
\begin{align}
	\cJ &\coloneqq \lim_{T\to\infty}\sum_{i=0}^{T}\langle  z_i,  Q z_i\rangle_{\mathcal \cH_1} 
	+ \ip{  \uveco_i,  R\uveco_i }_{\R^{n_u}},\nonumber \\
	&\quad z_{i+1} = A z_i + B \uveco_i, \uveco_i=\K z_i
		\label{eq:obj},
	\end{align}
\begin{align}
	\cJhat &\coloneqq \lim_{T\to\infty}\sum_{i=0}^{T}\langle  \hat z_i,  Q \hat z_i\rangle_{\mathcal \cH_1} 
	+ \ip{  \tuveco_i,  R\tuveco_i }_{\R^{n_u}},\nonumber\\
	&\quad \hat z_{i+1} = A \hat z_i + B \tuveco_i, \tuveco_i=\Ktilde \zhat_i
		\label{eq:nobj}.
\end{align}
Then, we upper bound the error $\cJhat - \cJ$.
\begin{theorem}[Convergence rate for $\cJhat - \cJ$]
	\label{thm:convergence_lqr_objective}
	Let $z_0$ be the initial state of the dynamical system of interest. Let $\rho(L)$ be the spectral radius of the closed-loop operator $L=A + BK$, where K is given by~\eqref{e:def_K}. Moreover, let $\cJ$ be as in~\eqref{eq:obj}, and $\cJhat$ as in~\eqref{eq:nobj}. Let $\new\zeta$ be a real number such that $\rho(L)\leq \new\zeta < 1$. 
	Lastly, let $\Gamma = 1 + \max\{\norm A, \norm B, \norm P, \norm K\}$.  Under Assumptions~\ref{a:bounded_kernel},~\ref{a:stabilizability}, and~\ref{a:detectability}, if $\norm{A - \Atilde}\leq \epsilon$, $\norm{B - \Btilde} \leq \epsilon$, 
	\new{assuming there exists some function $g$ such that} 
	$\norm{P - \Ptilde}\leq g(\epsilon)$ \new{for some value of $\epsilon$ chosen small enough so that} $g(\epsilon)\leq \frac{1 - \new\zeta}{6\norm{B}\tau(L, \new\zeta)\Gamma^2}$, and $\sigma_{\textrm{min}}(R) \geq 1$, we have that
	\begin{equation*}
		\cJhat - \cJ\leq 36 \sigma_{\textrm{max}}(R)\Gamma^9g(\epsilon)^2\kappa^2\frac{\tau(L, \new\zeta)^2}{1 - \new\zeta^2}.
	\end{equation*}
\end{theorem}
\begin{proof}
	The proof of this result can be obtained by following the same steps of~\cite[Theorem 1]{mania2019certainty}. Although deployed for $\A$ and $\B$ being matrices, these steps do not depend on the dimensionality of the state space, i.e., can be used also when considering function-valued dynamics, as we do in our work.
	\new{
	We begin by computing an upper bound on the error for the Riccati gain $\norm{K-\Ktilde}$. We then show that when this error is small enough, $\Ktilde$ stabilizes the system with exact kernel defined in~\eqref{eq:dynamics_compound}. Lastly, we use~\cite[Lemma 10]{fazel2018global} to upper bound the error on the objective functions.
	}
\end{proof}
\new{\begin{remark}
	Considering the constants appearing in the bounds of this section, they can be interpreted as follows: $\kappa$ is the upper bound on the kernel function's values, i.e., the kernel variance; $\gamma$, which is the regularization weight, could be \nnew{heuristically tuned based on the variance of the noise affecting the system in \eqref{eq:controlled_dynsys}}. $\delta$ should in principle be chosen to be as small as possible, as our bound in Theorem 5 holds with probability at least $1 - \delta$. The other constants (e.g., the norms of the learned operators) depend on the data used for training.
\end{remark}}
\section{Simulation results}
\label{s:experiments}
\begin{figure*}[t]
	\centering
	\begin{subfigure}{.32\linewidth}
		\centering
		\includegraphics[width=\linewidth]{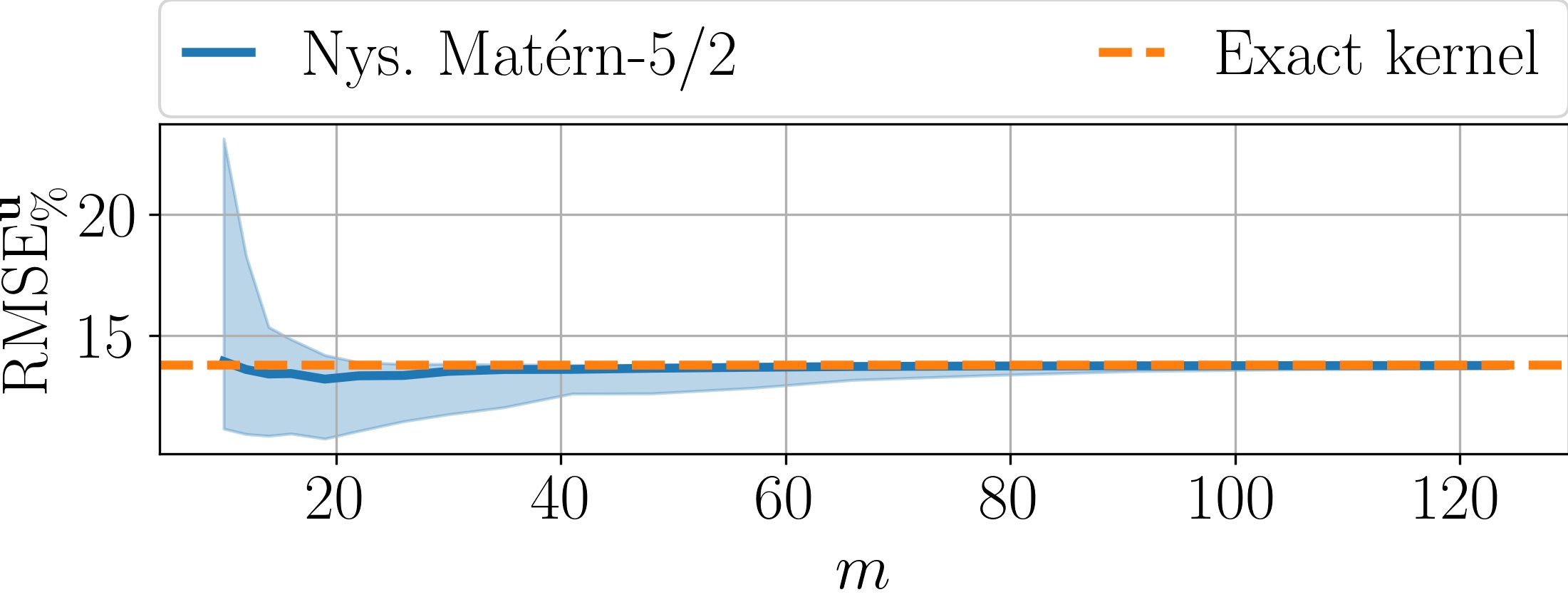}
		\caption{Error on the control input.}
		\label{subfig:opt_ctr_convergence_hjb}
	\end{subfigure}
\begin{subfigure}{.32\linewidth}
	\centering
	\includegraphics[width=\linewidth]{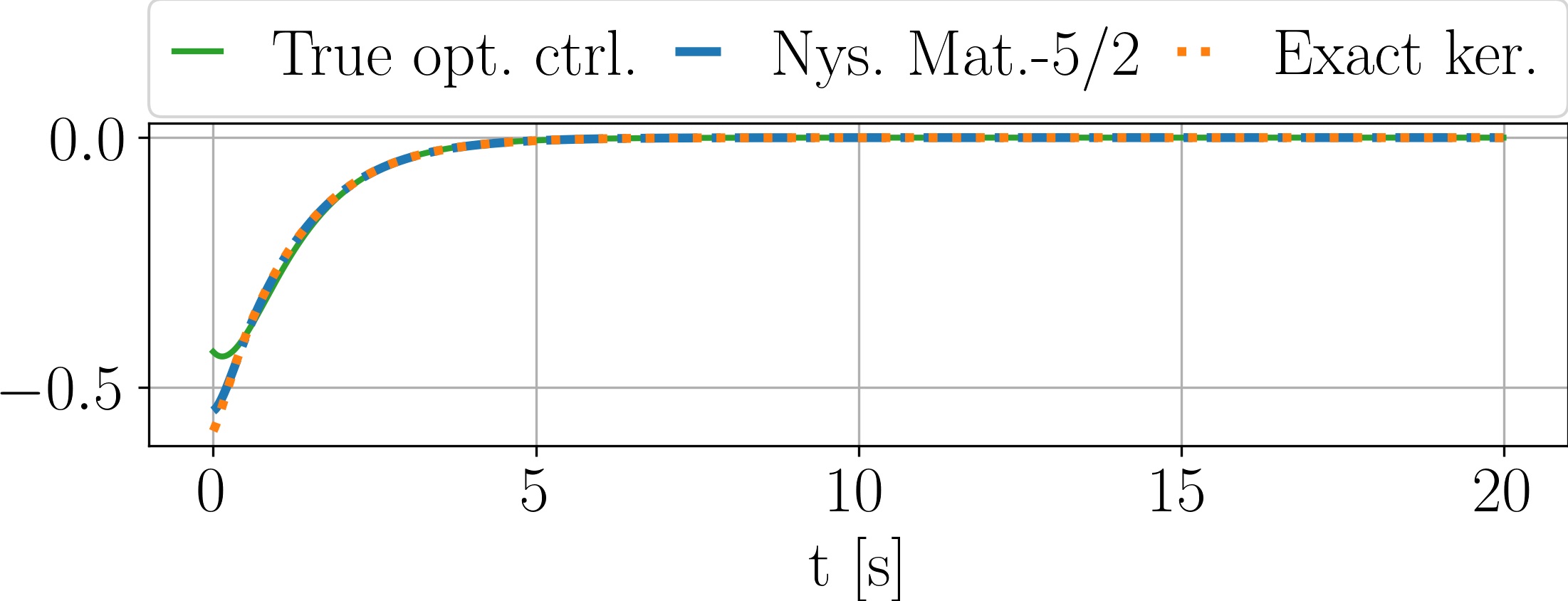}
	\caption{Example control law for $x_0=0.9$.}
	\label{subfig:ctrl_qual_vis}
\end{subfigure}
	\begin{subfigure}{.32\linewidth}
		\centering
		\includegraphics[width=\linewidth]{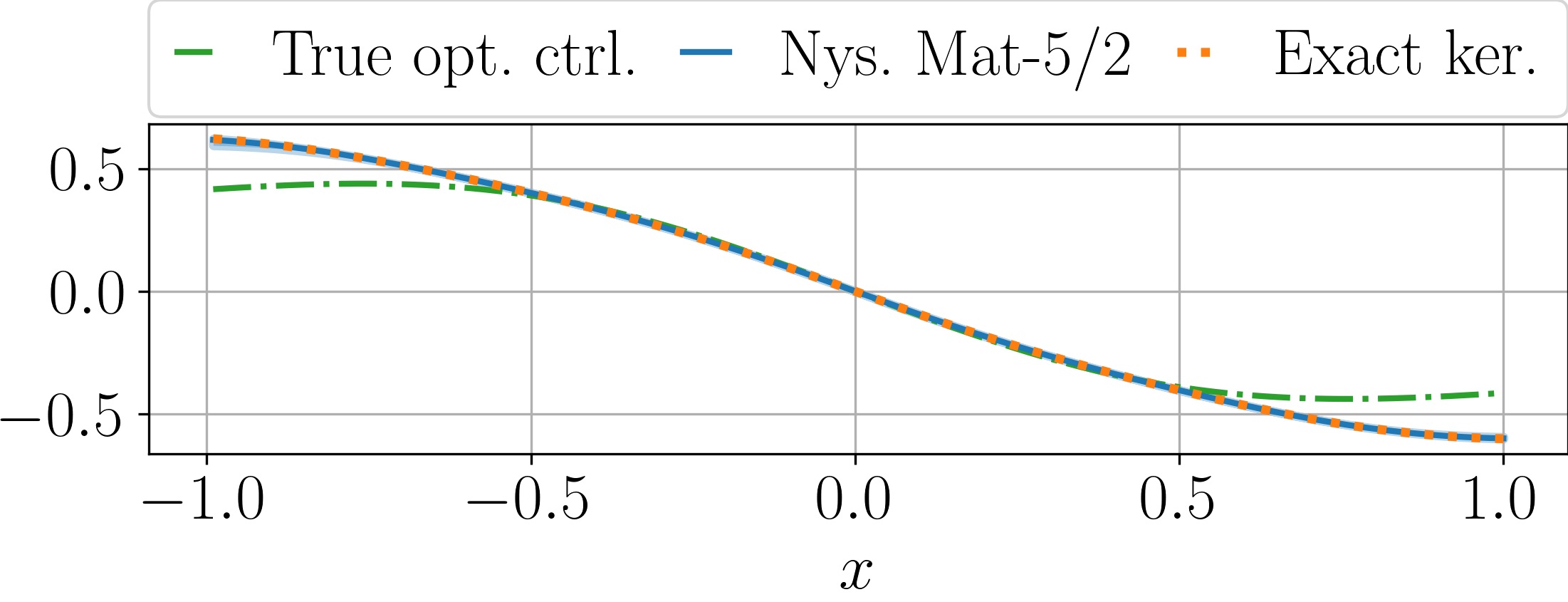}
		\caption{\nnew{Control law over domain.}}
		\label{subfig:opt_ctr_domain_hjb}
	\end{subfigure}
	\caption{\new{(a): Evaluation of the error between the control law defined in \eqref{e:ctrl_policy} and the true optimal control for the system in \eqref{e:hjb_sys}. Median, $15^{th}$ and $85^{th}$ percentile computed across 200 seeds. (b): A qualitative visualization of the optimal control retrieved, for $m=100$. \nnew{(c): A comparison between the true optimal control, the one defined in \eqref{e:ctrl_policy}, and its version obtained with an exact kernel, on the state space of the system \eqref{e:hjb_sys}, for $m=100$. For the Nyström approach, we show median, $15^{th}$ and $85^{th}$ percentile computed across 200 seeds.}}}
\end{figure*}
\begin{figure*}[t]
	\centering
	\begin{subfigure}{.33\linewidth}
		\centering
		\includegraphics[width=\linewidth]{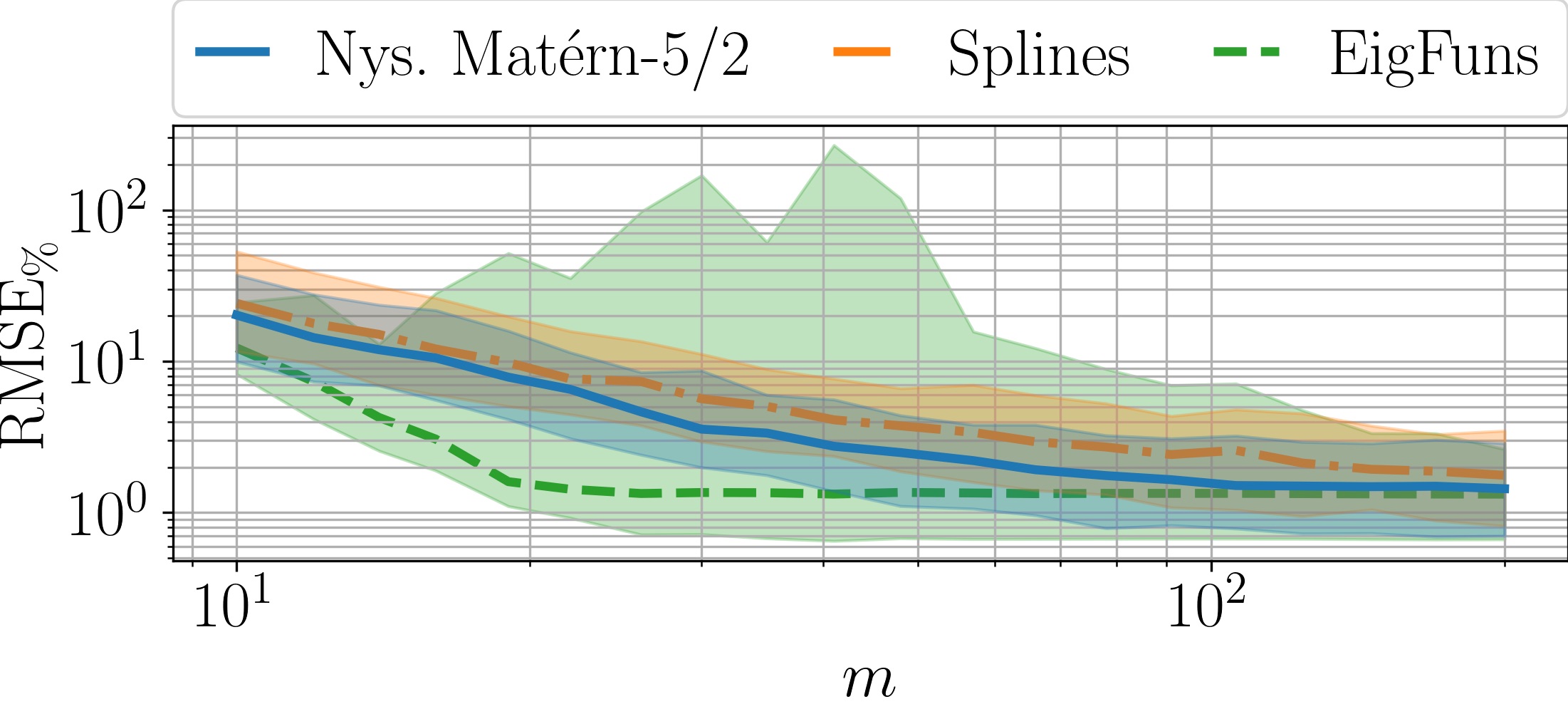}
		\caption{Open-loop forecasts.}
		\label{subfig:open_loop_duffing}
	\end{subfigure}\hfill
	\begin{subfigure}{.33\linewidth}
		\centering
\includegraphics[width=\linewidth]{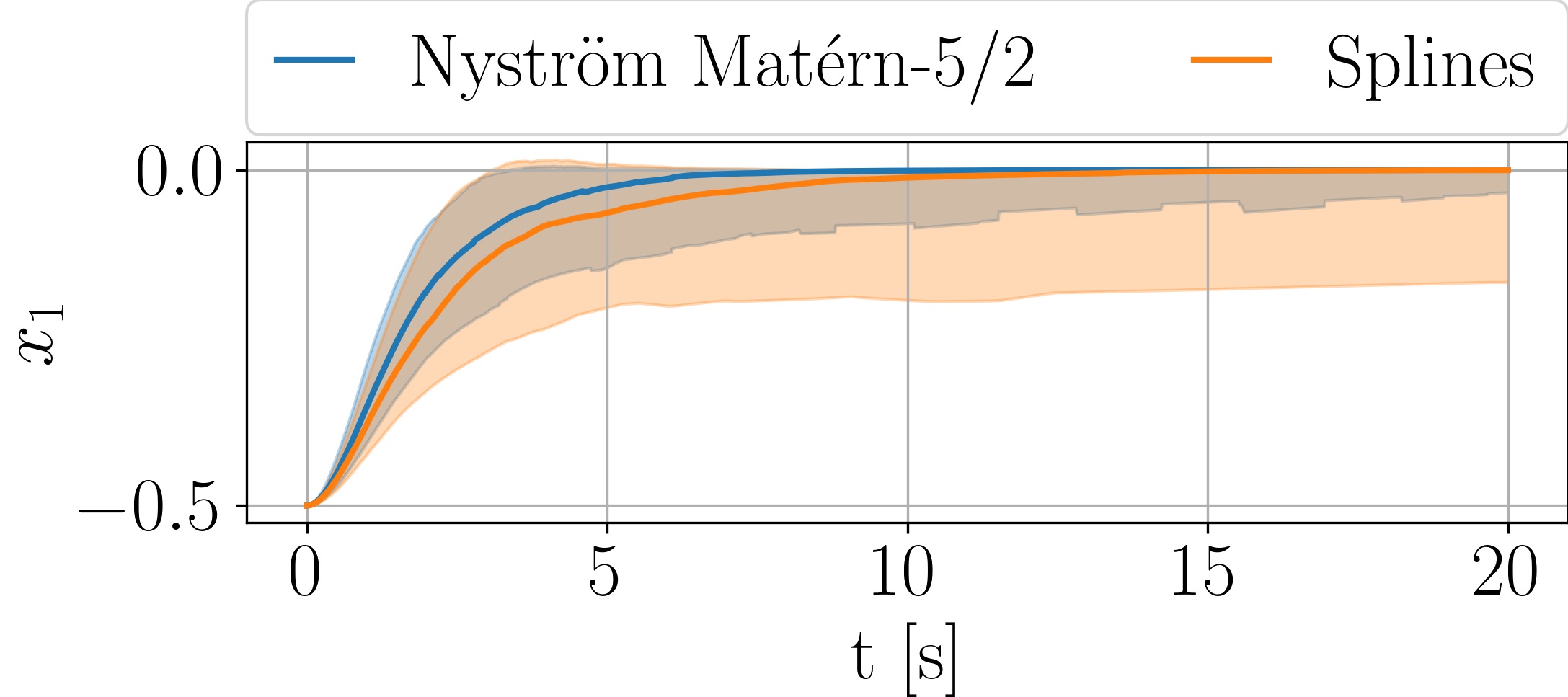}
\caption{Stabilization of $x_1$.}
\label{subfig:first_state_stabilization}
	\end{subfigure}
\begin{subfigure}{.33\linewidth}
	\centering
	\includegraphics[width=\linewidth]{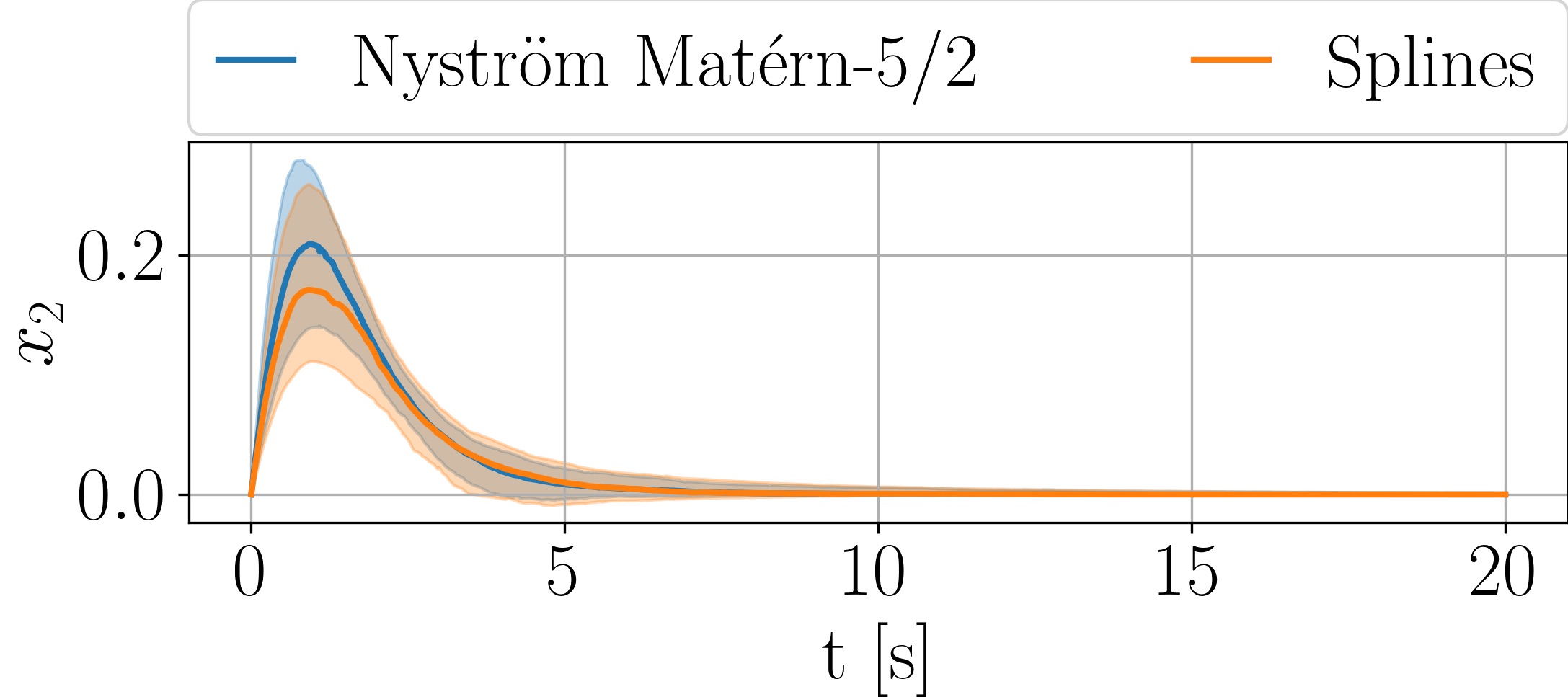}
\caption{Stabilization of $x_2$.}
\label{subfig:second_state_stabilization}
\end{subfigure}
	\caption{(a): $\mathrm{RMSE}_{\%}$ between the true trajectory and the one forecasted in open-loop for the Duffing oscillator, with three different feature representations (splines, eigenfunction approximation by~\cite{korda2020optimal}, and Nyström approximation of the Matérn-5/2 kernel), as a function of the dimensionality of the feature vector, $m$. Median, $15^{th}$ and $85^{th}$ percentile computed across 200 test trajectories with random initial conditions from the \new{unit ball}. 
	(b)-(c): The LQR control strategy to stabilize towards the origin, with $m=20$, starting from the initial conditions $[-0.5, 0.0]^T$. Most importantly, when the splines are used, in 2 cases the LQR gain yields unstable nonlinear dynamics (not included in the percentile range). Median, $15^{th}$ and $85^{th}$ percentile computed across 200 seeds.}
	\label{fig:duffing_experiments}
\end{figure*}
\begin{figure}[t]
	\centering
	\includegraphics[width=.8\linewidth]{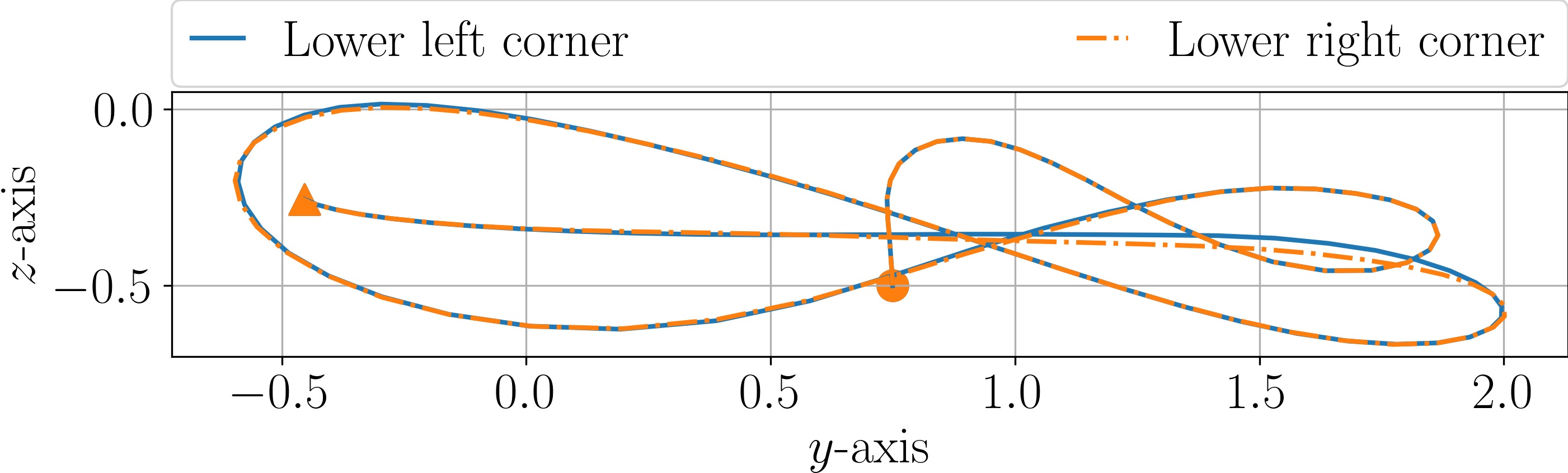}
		\caption{\new{An example trajectory of the left and right lower corners of the cloth in the $y$-$z$ plane. The circle denotes the starting position, while the triangle the final one.}}
	\label{fig:cloth_trajectory_sample}
\end{figure}
\begin{figure*}[t]
	\centering
	\begin{subfigure}{.33\linewidth}
		\centering
		\includegraphics[width=\linewidth]{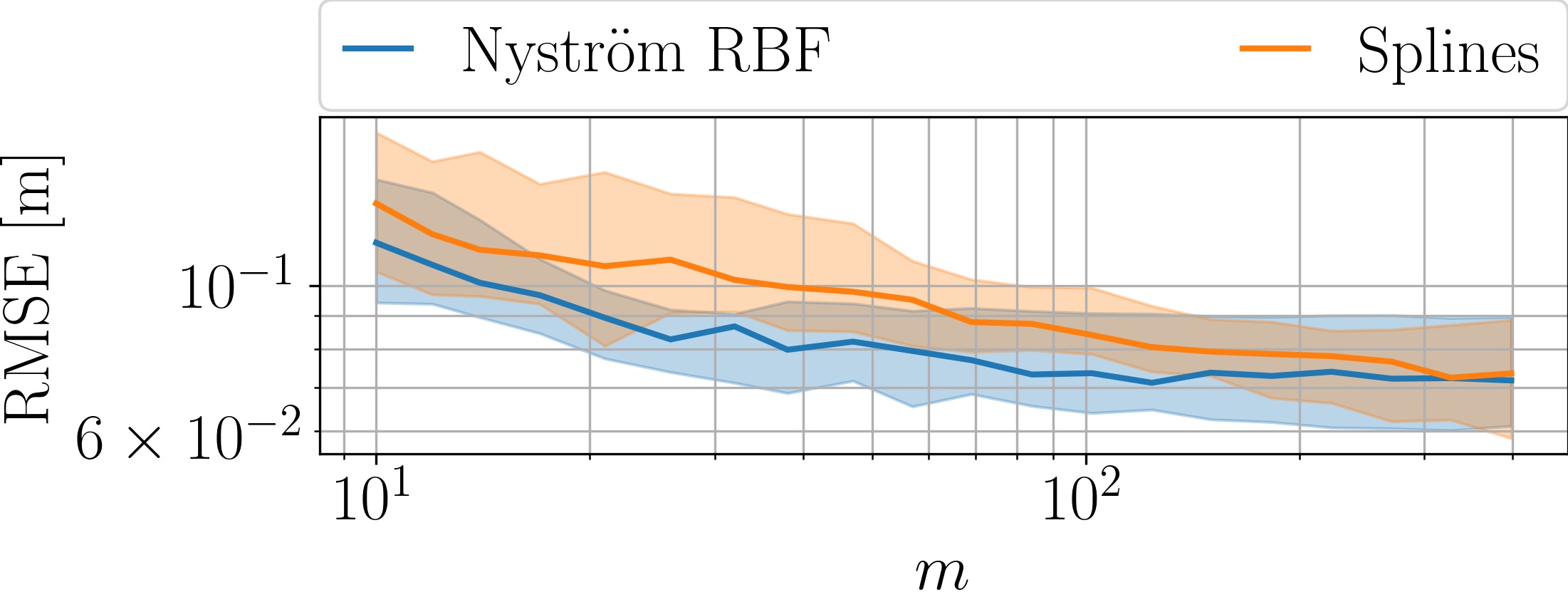}
		\caption{Open-loop forecast.}
		\label{subfig:cloth_open_loop_forecast}
	\end{subfigure}\hfill
\begin{subfigure}{.33\linewidth}
	\centering
	\includegraphics[width=\linewidth]{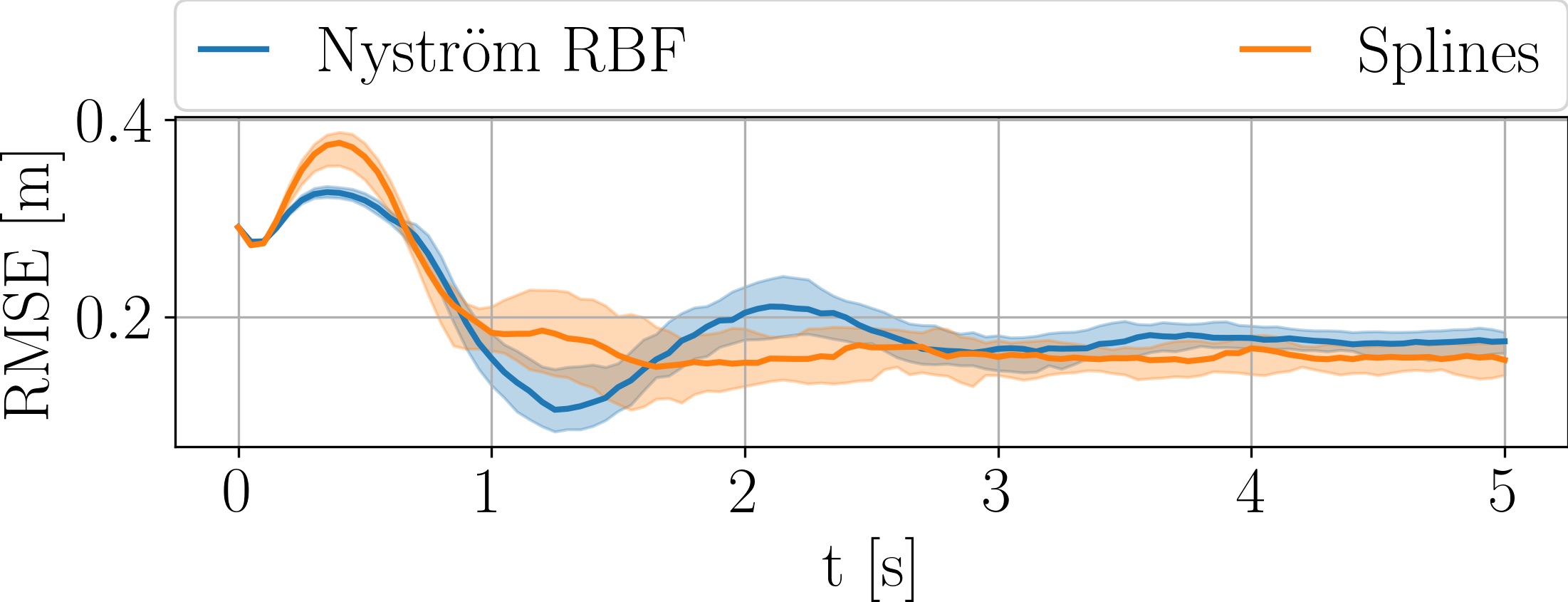}
	\caption{Regulation error.}
	\label{subfig:cloth_regulation}
\end{subfigure}
\begin{subfigure}{.33\linewidth}
	\centering
	\includegraphics[width=\linewidth]{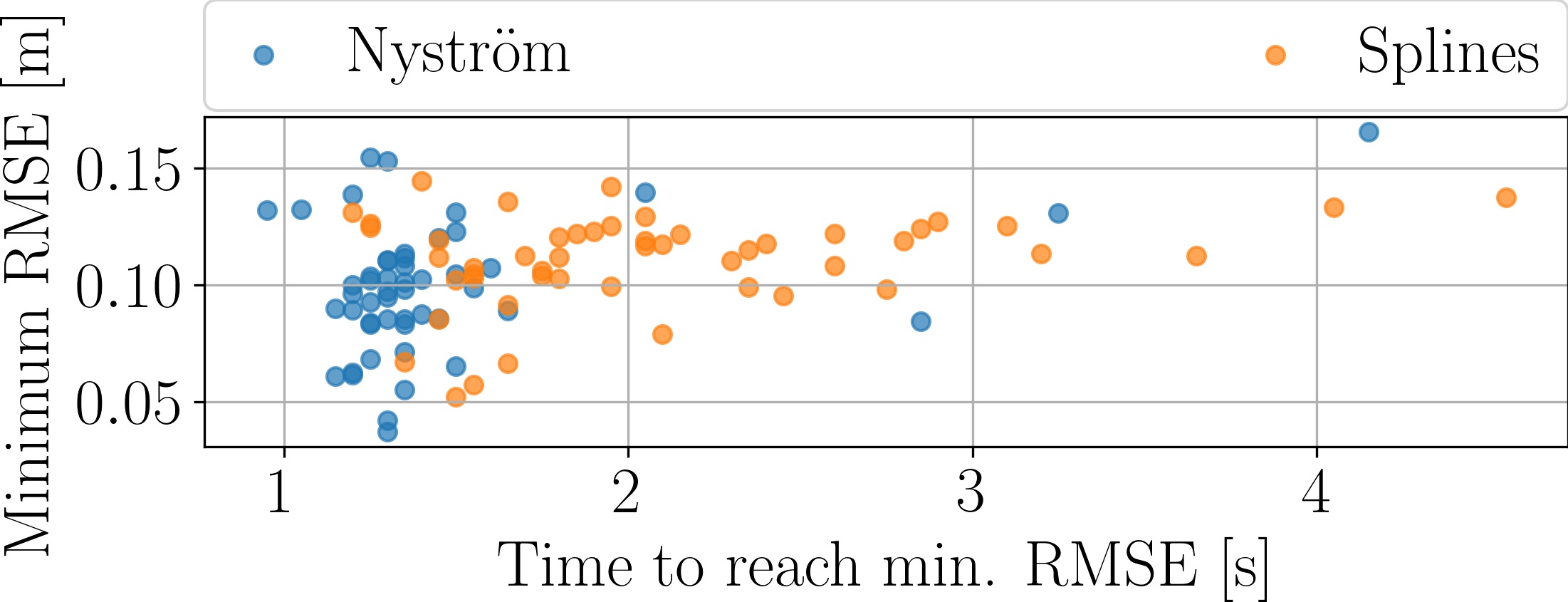}
	\caption{Time-to-min.\ RMSE.}
	\label{subfig:cloth_scatter}
\end{subfigure}
	\caption{(a): the RMSE computed between the true cloth trajectory and the one forecasted in open loop, with two different feature representations (splines vs.\ Nyström approximation of the RBF kernel). Median, $15^{th}$ and $85^{th}$ percentile computed across 10 testing trajectories, sampled with 20 different seeds.
	(b): regulation error between the target pose of the cloth and the actual one. Median, $15^{th}$ and $85^{th}$ percentile computed across 50 seeds. \new{(c): Scatter plot showing the time time needed by each simulation of the cloth experiment to reach the minimum distance from the target.}}
	\label{fig:cloth_simulations}
\end{figure*}
\begin{figure*}
	\centering
	\begin{subfigure}{.33\linewidth}
		\centering
		\includegraphics[width=\linewidth]{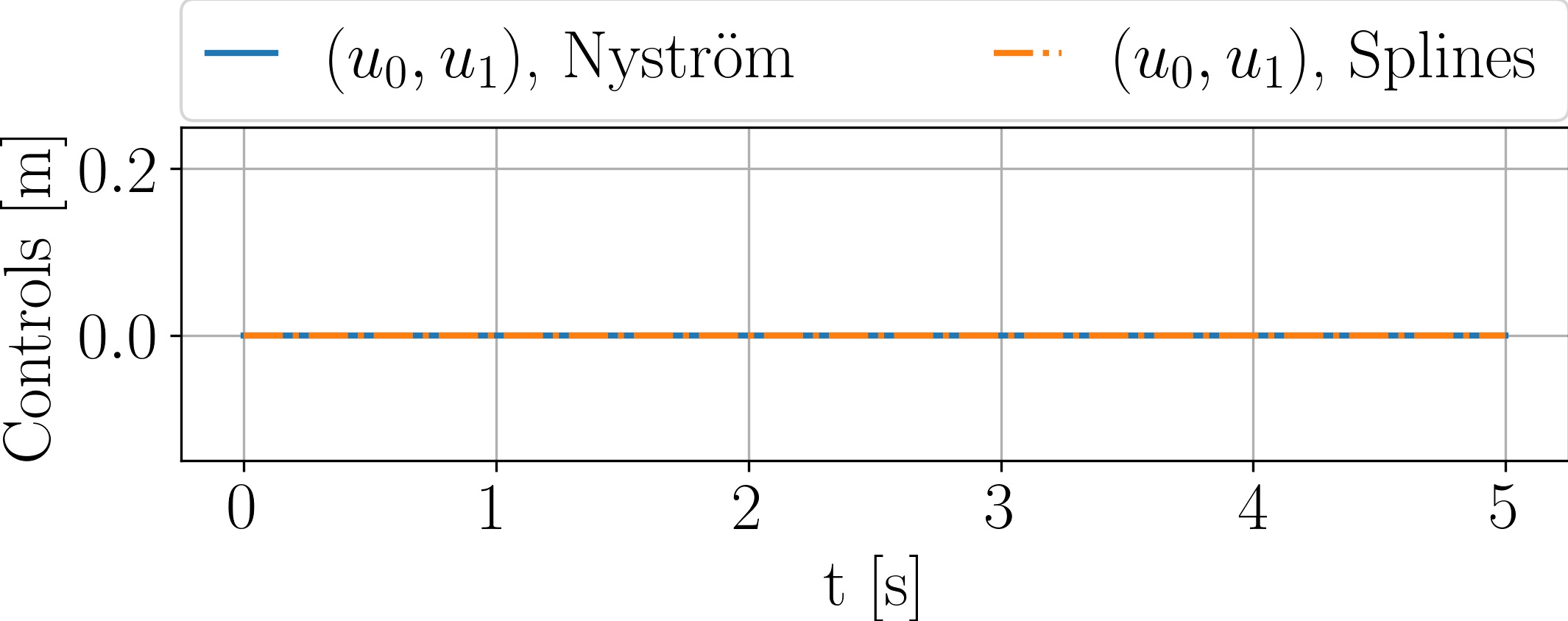}
		\caption{Displacement along $x$-axis.}
	\end{subfigure}\hfill
	\begin{subfigure}{.33\linewidth}
		\centering
		\includegraphics[width=\linewidth]{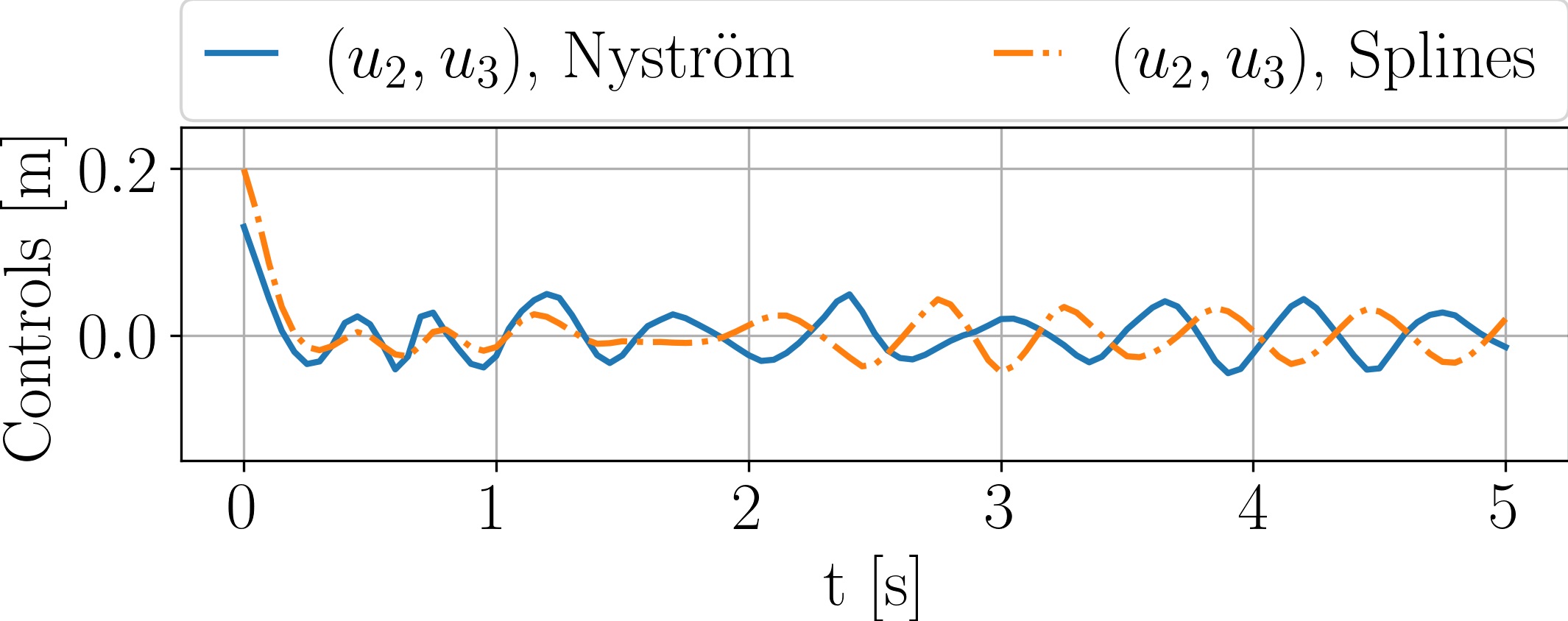}
		\caption{Displacement along $y$-axis.}
	\end{subfigure}
	\begin{subfigure}{.33\linewidth}
		\centering
		\includegraphics[width=\linewidth]{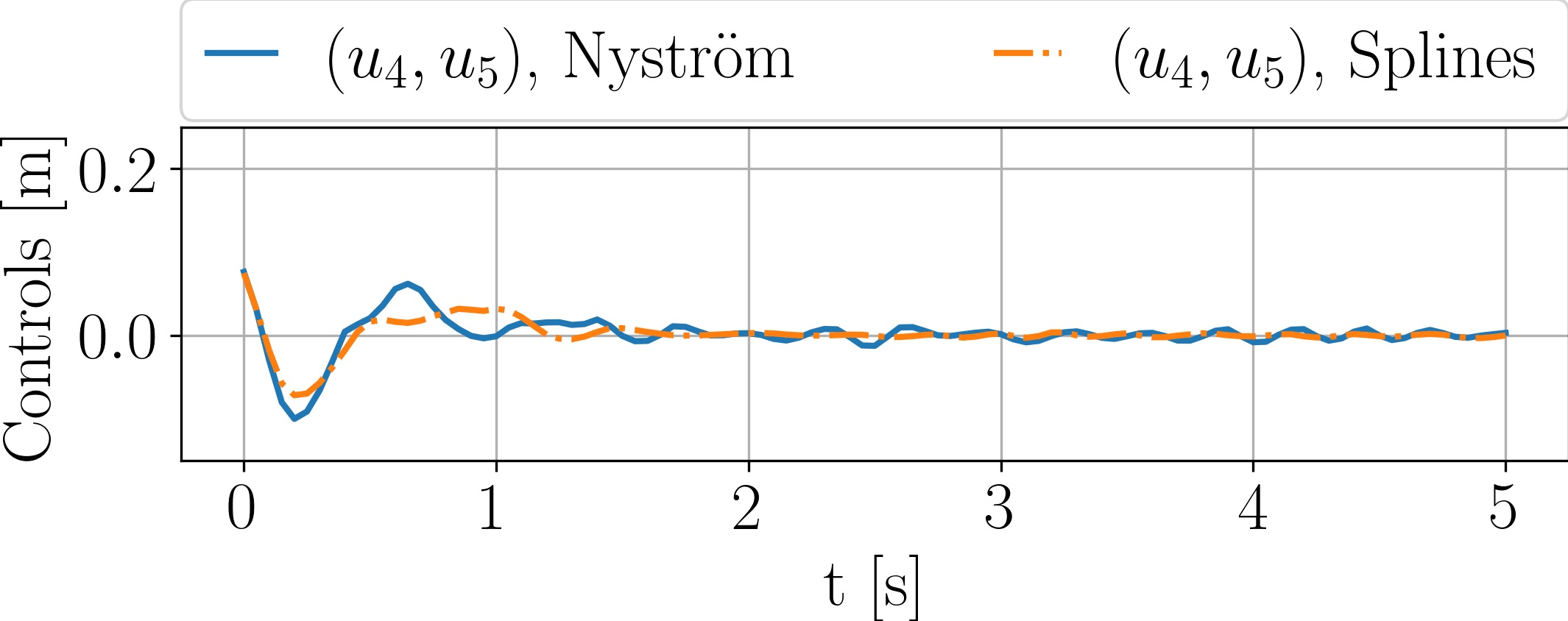}
		\caption{Displacement along $z$-axis.}
	\end{subfigure}
	\caption{\nnew{An illustrative example of the components of the control signal, for the cloth manipulation simulation. After approaching the target, the system enters a stable limit cycle. The elements in the pair $(u_i, u_{i+1})$ are identical. This is because, for each coordinate, the training controls (which are the variations in position of the upper cloth's corners) are the same for the left and the right corner, as exemplified in Fig.\ \ref{fig:cloth_trajectory_sample}. This bias is reflected in the trajectory obtained using the LQR gain.}}
	\label{fig:oscillating_control_cloth}
\end{figure*}
In this section, we propose a numerical validation of the Nyström-based system identification and the subsequent LQR algorithm discussed in the previous sections. \new{We will start in Section \ref{sec:poc_opt_ctrl} with a proof-of-concept evaluation of the control law retrieved by our Nyström-based approach, on the nonlinear dynamics with known optimal control discussed in \cite{guo2022tutorial}. Then}, we will consider \new{the Duffing oscillator} in Section~\ref{s:duffing}, and we will study the application of cloth manipulation in Section~\ref{s:cloth}.

In all our case-studies, the kernel lengthscale and the regularization parameters are determined with a grid search combined with 5-fold cross validation. For simplicity, we use the same regularization parameter for the estimate of the dynamics and the state reconstruction, i.e.~$λ=γ$. 
\new{\subsection{Proof-of-concept dynamics}
	\label{sec:poc_opt_ctrl}
%
In this section, we assess how the control law defined by \eqref{e:ctrl_policy}, denoted as $u$, compares to the known optimal control $u^{\textrm{true, opt}}$ for an illustrative single-input dynamical system. This simulation serves as a proof-of-concept to show that the proposed control law, retrieved by an LQR-based methodology, is not too far from the true optimal control. The continuous dynamics defined in the following are discretized with the Runge-Kutta method.

We consider the nonlinear dynamics studied by Guo et al.\ \cite{guo2022tutorial}:
\begin{equation*}
	\dot x = -x^3 + u.\numberthis\label{e:hjb_sys}
\end{equation*}
The known optimal control, for $Q=1$ and $R=1$, stabilizing the origin is $u^{\textrm{true, opt}}=x^3 - x\sqrt{1 + x^4}$. In the simulation, we started at $x_0=0.9$. The Nyström-based system was defined using the Matérn-5/2 kernel with lengthscale equal to 1 (\ie $k(x,y)\de (1+\sqrt{5}r +5r²/3)\exp(-\sqrt{5}r)$ where $r=‖x-y‖$), and regularization constant equal to $10^{-6}$. It was trained with 20 trajectories of length 2 s, with a sampling time of 0.01 s. The initial states for the training trajectories, and $u_t$ were uniformly sampled from $[-1, 1]$. The Nyström-based LQR is solved with $Q=C^TC$ and $R=1$. Given the size of the training dataset, we further compare both the Nyström-based control law and the true optimal control with the exact kernel dynamics from \eqref{eq:dynamics_compound}. 
Note that the latter induces states belonging to $\ran(\Zx)$, hence when starting from an initial state in $\ran(\Zx)$ 
it can be seen as a special case of the Nyström dynamics \eqref{eq:nystrom_dynamics_compound} by taking $\Pin=I$, and choosing $m=n$ and $\ldmo{1}=\xvec_2,\ldmo{m}=\xvec_{n+1}$ in the definition of $\Stildeout$.
%
By doing so, we get
$
	\zvec_{t+1} = (Z_{\xvec}Z_{\xvec}^*)^{1/2}  (SS^*+γI)^{-1}S\vvec{ Z_\xvec^*[(Z_\xvec Z_\xvec^*)^{\dagger}]^{-1/2} \zvec_t, \\\uvec_t },
$
and $\zvec_0 \de \Bout^*\psi(\xvec_0)$.
\begin{table}[t]
	\caption{\new{Infinite-horizon LQR costs for data-driven and optimal control laws, on proof-of-concept dynamics \eqref{e:hjb_sys} (for the Nyström approximation we show \textbf{median}, $15^{th}$ and $85^{th}$ percentiles across 200 seeds).}}
	\label{table:cum_costs}
	\centering
	\renewcommand{\arraystretch}{1.0}
	\setlength{\tabcolsep}{2.7pt}
	\begin{tabular}{lll}
		\toprule
		\textbf{Cost of \eqref{e:ctrl_policy}, Nystr.}
		\multirow{3}{*} &$m=10$ & $\mathbf{57.12},\ 57.00,\ 58.19$ \\
		&$m=50$ &$\mathbf{57.10},\ 57.05,\ 57.10$\\
		&$m=100$ &$ \mathbf{57.10},\ 57.09,\ 57.10$ \\
		\midrule
		\textbf{Cost of \eqref{e:ctrl_policy}, exact ker.} & & $57.10$\\
		\midrule
		\textbf{Cost of opt.\ ctrl.} && $56.74$\\
		
		%
		%
		\bottomrule
	\end{tabular}
\end{table}
Besides computing the infinite-horizon LQR cost, we assess the distance between the two control laws with this key performance indicator:
$
	\mathrm{RMSE}_{\%}^{\uvec} = \sqrt{{\sum_{t=1}^{200} (u_{t} - u^{\textrm{true, opt}}_{t})^2}/{\sum_{t=1}^{200} (u^{\textrm{true, opt}}_{t})^2}}\cdot 100.
$
As shown in Fig.~\ref{subfig:opt_ctr_convergence_hjb}, the retrieved control law is close to the optimal one, achieving a final median $\mathrm {RMSE}_{\%}^{\uvec}$ of approximately 13\%, both for the exact and for the Nyström kernel. A qualitative comparison of the control laws involved in this simulation is offered in Fig.\ \ref{subfig:ctrl_qual_vis}. \nnew{Fig.\ \ref{subfig:opt_ctr_domain_hjb} shows the optimal and data-driven control laws as a function of the state $x\in[-1, 1]$ of system \eqref{e:hjb_sys}.} 
Even though the data-driven control laws have an error w.r.t.\ the optimal one near the boundaries of the training domain, Fig.\ \ref{subfig:ctrl_qual_vis} shows that convergence to the true optimal control law happens fast even when starting in those regions (e.g., for $x=0.9$). 
Lastly, Table \ref{table:cum_costs} compares the infinite-horizon LQR cost achieved by the control law in \eqref{e:ctrl_policy}, and by the optimal control for the task considered.
}
\subsection{\new{Duffing oscillator}}
\label{s:duffing}
We \new{now consider} a classic benchmark for nonlinear control, namely the \emph{damped Duffing oscillator}~\cite{korda2020optimal}. The nonlinear dynamics are defined in continuous time domain and given by the following differential equations:
\begin{align}
	\dot x_1 &= x_2,\qquad	\dot x_2 = -0.5x_2 - x_1(4x_1^2 - 1)+ 0.5 u\label{eq:duffing_eqs}.
\end{align}
Similarly to~\cite{korda2020optimal}, the dynamics are discretized with the Runge-Kutta method to obtain a difference equation of the form of~\eqref{eq:controlled_dynsys}, with a discretization step of 0.01 s. For $u=0$, this dynamical system is known to have an unstable equilibrium point at the origin, and two stable equilibria at $[-0.5, 0]^T$ and $[0.5, 0]^T$. The training set consists of 100 trajectories of length 5 s and without forcing, and 100 trajectories of length 2 s with forcing, all staring from inside the unit ball around the origin. For the forced data, the system is excited with inputs sampled uniformly at each time step from the interval $[-1, 1]$.

Firstly, we are interested in assessing how well the Nyström-based dynamics in~\eqref{eq:nystrom_dynamics_vector_partial} approximate the ones of the true dynamical system in~\eqref{eq:duffing_eqs}. To do so, we test the open-loop forecast of the dynamics subject to a square wave with unitary amplitude and frequency 3.33 Hz, starting from a random initial condition in the unit ball. This means that we excite the data-driven model of~\eqref{eq:nystrom_dynamics_compound} with the square wave, reconstruct the state via matrix $C$ from~\eqref{eq:reconstruction}, and compare 
the states reconstructed using $C$
 against the states visited by the true nonlinear dynamics. The length of the open-loop forecast is set to 2 s. The results can be observed in Fig.~\ref{subfig:open_loop_duffing}. 
 The Nyström identification is compared against the splines introduced in~\cite{korda2018linear}, with centers sampled uniformly at random from the unit circle, and against the eigenfunction-based method by~\cite{korda2020optimal}. 
  \new{For the latter approach, we opted for the implementation without optimization of the eigenvalues, which is more robust for large values of $m$, and
 	practically equivalent to the one with optimization, based on the empirical evaluation carried out in \cite{korda2020optimal}.}
The Nyström identification uses a Matérn-5/2 kernel, which seems empirically to perform better than other Matérn or RBF kernels in this setting. The lengthscale and variance are set to 1 by cross-validation%
. 
We set $\gamma=10^{-6}$. Both the splines and the Nyström approach use the same samples for approximating the input and output spaces. The Nyström landmarks are sampled uniformly at random from the output training set. We consider the normalized root-mean-squared error (RMSE) between the actual states $x$ and the forecasted ones $\tilde x$, defined as
$
	  \mathrm{RMSE}_{\%} = \sqrt{[\sum_{i=1}^2\sum_{k=1}^{200} (\tilde x_{i, k} - x_{i,k})^2]/[\sum_{i=1}^2\sum_{k=1}^{200} x_{i, k}^2]}\cdot 100,
$
where $x_{i,k}$ denotes the $i$-th dimension of the state at time step $k$. 
We can observe that all the representations converge in median to the same minimum error, but the Nyström method yields a smaller error when a few landmarks are used. Moreover, we can observe that the approach in~\cite{korda2020optimal} suffers from a much larger variance, probably due to the short horizon of the unforced training trajectories employed (w.t.r.\ to 8 s used in~\cite{korda2020optimal}).

The goal of the control strategy is to stabilize the system towards the origin $[0, 0]^T$, starting from the initial condition $[-0.5, 0.0]^T$, for $m=20$. As discussed in Section~\ref{s:kernel_lqr}, the optimal LQR gain is retrieved by solving~\eqref{eq:lqr_koopman_nystrom_finite_dimensional}, and the optimal state-feedback law $u_k = \Ktilde \Pout\psi(\xvec_k)$ is plugged in~\eqref{eq:duffing_eqs}. 
The $R$ weighting matrix is set to the identity, while $\tilde Q= C^TC$, where $C$ is obtained by solving~\eqref{eq:reconstruction} both with the Nyström and the splines features. 
The values of the states $x_1$ and $x_2$ as a function of time, for a state-feedback control, are shown in Figures~\ref{subfig:first_state_stabilization} and~\ref{subfig:second_state_stabilization}. 
We can observe that the Nyström representation yields a better behavior of the closed-loop system, especially when considering the stabilization of the first state of the oscillator. Note that the method reported in~\cite{korda2020optimal} builds a complex-valued data-driven system, which would yield to a complex-valued Riccati gain and control input. 
\nnew{While it is certainly possible to project the complex-valued LQR-based control law retrieved in this way to be a real-valued signal, we observed in preliminary experiments that such a control law yields a suboptimal behavior, such as steady-state errors. Note that this issue is circumvented by~\cite{korda2020optimal}, who directly constrain the control to be a real-valued signal, and solve such an optimal control problem online with a model predictive control loop.}
\subsection{Cloth manipulation}
\label{s:cloth}
Cloth manipulation is arguably one of the most active and challenging fields in robotic manipulation. Recent work has tried to address this problem in a model-based fashion, proposing approximate, data-driven methods that can identify the dynamics of cloth, and then combine such approximate models with real-time predictive controllers~\cite{amadio2023controlled,zheng2022mixtures,luque2022model}. Similarly to
\cite{amadio2023controlled}, in this work we consider a squared piece of cloth, modeled based on the simulator deployed by~\cite{coltraro2022inextensible}. The training dataset is created by moving the upper corners of the cloth in a butterfly-like shape in the $y$-$z$ plane \new{(see Fig.\ \ref{fig:cloth_trajectory_sample})}, with different angles w.r.t.\ the $x$-$y$ plane. The angle is sampled from the interval $[-60\degree, 60\degree]$\footnote{\nnew{We refer the interested reader to Figs.\ 4 and 5 by Amadio et al.\ \cite{amadio2023controlled} for a clear visualization of the task considered here.}}. The goal is to identify the cloth dynamics in this setup. The cloth is modeled as a square $8\times 8$ mesh. The state is represented by the 3-dimensional position of the mesh points, and the control input is given by the variation in position of the upper corners. The control input is in $\R^6$, while the state is in $\R^{192}$. \new{The sampling time is equal to 0.05 s.}

 We first assess the performance of the proposed approximate kernel-based identification on this dataset. In particular, the training set consists of 30 trajectories randomly sampled from a total of 40. The length of the trajectories is 5 s, and the sampling time is 0.05 s. We test the open-loop system forecasts in the remaining ten testing scenarios. We compare the Nyström approximation of the RBF kernel against the thin plate splines employed by~\cite{korda2018linear}. 
 \nnew{Note that, in~\cite{korda2018linear} and in Section~\ref{s:duffing}, the feature vector contains elements of the form $\lVert {\mathbf x - \mathbf x_0}\rVert^2\log{\lVert{\mathbf x-\mathbf x_0}\rVert}$, for $\mathbf x_0$ being a center sampled uniformly from a box in $\mathbb R^d$. However, given the extremely large dimensionality of the cloth's state-space representation, the centers sampled in this way might be quite far from the real cloth configurations visited during the trajectories. We observed, in preliminary experiments, that this unbounded distance between centers and training points yields a highly numerically unstable behavior of splines for $m\geq 200$. For this reason, we propose to sample the splines' centers uniformly at random from the training data, making splines competitive for large values of $m$ as well.}
 Similarly, the Nyström landmarks are sampled uniformly at random from the output training set. 
The RBF kernel lengthscale is set to $10$ and $\gamma$ to $10^{-7}$ by cross-validation. A comparison is offered in Fig.~\ref{subfig:cloth_open_loop_forecast}, showing the \emph{unnormalized} $\mathrm{RMSE}$ on the testing trajectories. While both feature representations converge to the same RMSE, the RBF representation converges faster and leads to a smaller error in the low-feature regime. Also, note that, given the high dimensionality of the state space, in this case the feature transformation can be seen as a dimensionality reduction technique.
\begin{figure}
	\centering
	\includegraphics[width=.8\linewidth]{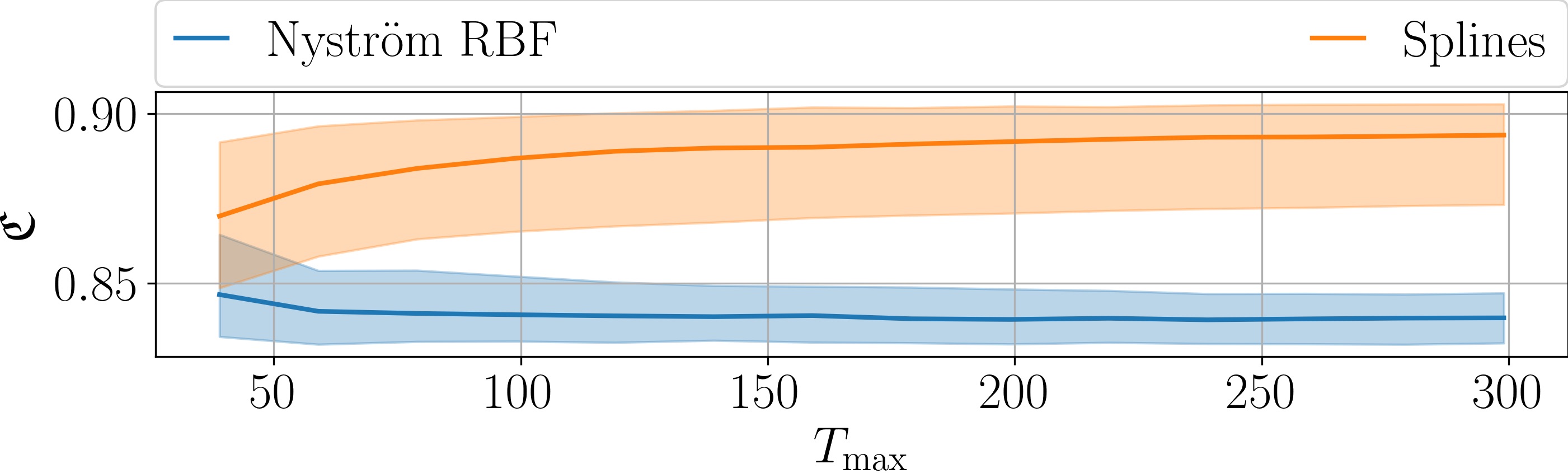}
	\caption{\nnew{Cost \eqref{e:avgd_running_cost} as a function of the horizon $T_{\mathrm{max}}$.}}
	\label{fig:averaged_running_cost}
\end{figure}

 In order to use an LQR approach, we define the control goal as moving the cloth to a target pose. Such a pose is given by the initial pose, rotated by $45\degree$ about the upper corners. 
This simulation captures how well the Koopman models the cloth's deformations, which can be used in fast, dynamic motions. In this case, the number of Nyström landmarks and spline centers is set to 100. \new{Moreover, $\tilde Q = 0.0075 C^TC$, and $R$ is the identity matrix of suitable dimensions.}
As we can observe in Fig.~\ref{subfig:cloth_regulation}, the Nyström method improves the regulation error, getting closer to the target pose. \new{The better performance of the Nyström features is confirmed by Fig.\ \ref{subfig:cloth_scatter}. There, we can observe that the cloth gets closer to the target in a shorter amount of time.}
Note that due to gravity and to the deformable nature of cloth, the cloth will not stay in the target pose once reached, which is why the RMSE increases again for both methods after reaching the target. \new{After the target is reached, we observe that the system enters a stable limit cycle with an RMSE at around 19 cm for the Nyström approach, and 18 cm for the splines. In this cycle, the cloth oscillates around a vertical position that is closer to the target one than the initial pose, for which the RMSE is approximately 30 cm. The limit cycle can be visualized, for instance, in Fig.\ \ref{fig:oscillating_control_cloth}.}

\nnew{As a final assessment, we can consider the following averaged running cost computed on the true nonlinear dynamics:
\begin{equation}
	\mathfrak C =\frac{1}{T_{\textrm{max}}}\sum_{t=0}^{T_{\textrm{max}}}0.0075  (\mathbf x_t - \mathbf r)^T\mathbf (\mathbf x_t - \mathbf r) + \mathbf u_t^T\mathbf u_t ,\label{e:avgd_running_cost}
\end{equation}
where $\mathbf r$ is the target state, $T_{\textrm{max}} = \mathrm{T\ s} / \mathrm{0.05\ s}$, $\mathrm T$ is a variable time horizon, $\mathrm{0.05\ s}$ is the sampling time, and $\mathbf u_t$ is the control from \eqref{e:ctrl_policy}, computed with splines or the Nyström features. This metric is related to the infinite-horizon LQR cost used to retrieve the state-feedback gain, since $\tilde Q=0.0075C^TC$, and $\mathbf x_t \approx C\tilde{\mathbf z}_t$ ($\tilde{\mathbf{z}}$ indicating either the Nyström or the splines features). Considering this metric, we observe in Fig.\ \ref{fig:averaged_running_cost} that the Nyström features yield a consistently smaller cost than splines.
}

\section{Conclusions}
\label{s:conclusion}
In this paper, we have studied the combination of the Koopman operator framework with reproducing kernels and the Nyström method, to design linear control for nonlinear dynamical systems. In a novel theoretical analysis, we \new{have linked} novel error rates, due to the Nyström approximation, on the dynamical system's representation to its effect on the linear quadratic regulator problem. Lastly, we have evaluated the effectiveness of the proposed method both on two classic control benchmark, and on the task of dynamic cloth manipulation.
\section*{Acknowledgments}                              
	E.\ Caldarelli, A.\ Colomé and C.\ Torras acknowledge support from the project CLOTHILDE (“CLOTH manIpulation Learning from DEmonstrations”), funded by the European Research Council (ERC) under the European Union’s Horizon 2020 research and innovation programme (Advanced Grant agreement No 741930).
	E.\ Caldarelli acknowledges travel support from ELISE (GA No 951847). 
	C. Molinari is part of “GNAMPA” (INdAM) and has been supported by the projects MIUR Excellence Department awarded to DIMA UniGe CUP D33C23001110001, AFOSR FA8655-22-1-7034, MIUR-PRIN 202244A7YL and PON “Ricerca e Innovazione” 2014–2020.
	C.\ Ocampo-Martinez acknowledges the support from the project SEAMLESS (PID2023-148840OB-I00) funded by MCIN/AEI. 
	 L. Rosasco acknowledges the financial support of the European Research Council (grant SLING 819789), the European Commission (Horizon Europe grant ELIAS 101120237), the US Air Force Office of Scientific Research (FA8655-22-1-7034), the Ministry of Education, University and Research (FARE grant ML4IP R205T7J2KP; grant BAC FAIR PE00000013 funded by the EU - NGEU) and the Center for Brains, Minds and Machines (CBMM). 
\appendix

\section{Closed form for $\G$}
\label{a:risk}
We recall that we defined the risk in \eqref{e:def_risk} as
$
	\cR(W) \de \frac1n \sum_{i=1}^{n}\n*{ \psi(\xvec_{i+1}) - W ϕ(\wvec_i)}_{\cH_1}².
$

\paragraph*{Expression of the risk} 
For any orthonormal basis $(g_j)_{j∊\mathbb N}$ of $\Hspace_1$, it holds
\begin{align*}
	&\mathcal R(W)	\\
	&= \frac1n \sum_{i=1}^{n} \left[ \sum_{j\in\mathbb N}
	\innerprod{g_j, \psi(\xvec_{i+1})-W\phi(\wvec_i)}_{\mathcal H_1}^2\right] \\
	&= \frac1n\sum_{j\in\mathbb N}\left[\sum_{i=1}^{n}\left(\innerprod{g_j, \psi(\xvec_{i+1})}_{\Hspace_1}-\innerprod{W^*g_j, \phi(\wvec_i)}_{\Hspace}\right)^2\right] \\
	&= \frac1n\sum_{j\in\mathbb N}\norm{\sqrt n\Zx g_j-\sqrt nSW^*g_j}^2_{\R^n} \\
	&= \nHS{\Zx - SW^*}². 
\end{align*}
\paragraph*{Expression of $\G$} Recall $\G=\argmin_{W:\cH→\cH_1}\cR(W)+\gamma\new{\norm{W}_{\HS}²}$.
Computing the gradient of $\cR(W)$ gives
$
	\pd{\mathcal{R}(W)}{W} = - 2\Zx^*S + 2WS^*S.
$
The first-order optimality condition with additive Tikhonov regularization leads to
\begin{align}
	\G (S^*S + \gamma I) &= \Zx^*S, \nonumber\\
	\G &= \Zx^*S (S^*S + \gamma I)^{-1} \label{eq:a_b_compound_expr}.
\end{align}
\new{This gives the expression claimed in \eqref{e:def_G} by applying the Woodbury identity.}

\section{Convergence rate for $\nG-\G$}

\new{We prove here the convergence rate for $\nG-\G$ given in Theorem~\ref{r:bound_G_nG_opnorm}, using bounds on the Nyström approximation derived in Section~\ref{ss:nystroem_error}.}

\subsection{Proof of Theorem~\ref{r:bound_G_nG_opnorm}}
\label{appendix:nystrom_rate}
Via \eqref{e:def_G} and \eqref{eq:a_b_compound_expr_nystr}, denoting $\rfK=SS^*+γI$ the (weighted) regularized kernel matrix, it holds
\newcommand\nH{\new{\tilde{H}_{γ}}}
\newcommand\nHi{\new{\tilde{H}_{γ}^{-1}}}
\begin{align*}
	\G &= \Zx^*\rfK^{-1}S, \qquad
	\nG 
	=\Pout\Zx^* \nHi S\Pin,
\end{align*}
where
$
	\nHi
	\de (S\Pin S^* + \gamma I)^{-1}.
$
In particular, we can observe that
\begin{align}
	\n{\nG-\G}
	&≤  \norm{\Zx^*\rfK^{-1} S(I -  \Pin)}\nonumber \\
	&\quad+\n{\Zx^* (\rfK^{-1}-\nHi) S\Pin}\nonumber\\
	&\quad+\n{(I-\Pout)\Zx^* \nHi S\Pin} .
	\label{e:bound_G_nG_dec}
\end{align}
The first addend can be split as follows: 
\begin{align}
	\norm{\Zx^* \rfK^{-1} S\Pin^\perp}
	&≤ \norm{\Zx^* \rfK^{-1}}\norm{S\Pin^\perp}≤ \frac\kappa\gamma\norm{S \Pin^\perp}. 
	\label{e:bound_G_nG_t1}
\end{align}
By leveraging the identity $A^{-1} - B^{-1} = A^{-1}(B - A) B^{-1}$ for two invertible operators $A$ and $B$, the second addend in \eqref{e:bound_G_nG_dec} can be bounded as
\begin{align*}
	\medmath{\norm{\Zx^* ( \rfK^{-1} - \nHi ) S\Pin}}&\medmath{= \norm{\Zx^* \rfK^{-1} (S\Pin S^* - SS^*) \nHi S\Pin}}\\
	&\medmath{≤ \frac{κ}{γ} \n{\Pin^\perp S^*}² \n{ \nHi S\Pin }.} 
\end{align*}
Note that, using a polar decomposition, $S\Pin = (S\Pin S^*)^{1/2}U$ for some partial isometry $U$, and thus
\begin{align}
	\hspace{-0.3cm}\n{\nHi S\Pin }
	&≤ \n{(S\Pin S^* + \gamma I)^{-1/2}}\nonumber\\
	&\quad\cdot\!	\n{(S\Pin S^* + \gamma I)^{-1/2}(S\Pin S^*)^{1/2} }\nonumber\\
	&≤ γ^{-1/2},\!\!\!\!
	\label{e:bound_G_nG_hS}
\end{align}
so that 
\begin{align}
	\norm{\Zx^* ( \rfK^{-1} - \nHi ) S\Pin}
	&≤ \frac{κ}{γ^{3/2}} \n{\Pin^\perp S^*}².
	\label{e:bound_G_nG_t2}
\end{align}
Eventually, using again \eqref{e:bound_G_nG_hS} the third addend in \eqref{e:bound_G_nG_dec} can be bounded as
\begin{align}
	\n{\Pout^\perp \Zx^* \nHi S\Pin} 
	&≤ \n{\Pout^\perp \Zx^*}\n{ \nHi S\Pin} \nonumber\\
	&≤ \new{\gamma^{-1/2}}\n{\Pout^\perp \Zx^*} .
	\label{e:bound_G_nG_t3}
\end{align}
Putting together \eqref{e:bound_G_nG_t1}, \eqref{e:bound_G_nG_t2} and \eqref{e:bound_G_nG_t3}, we get
$
	\n{\nG-\G}
	≤ \frac{κ}{γ} \norm{\Pin^\perp S^*} + \frac{κ}{γ^{3/2}} \n{\Pin^\perp S^*}²
	+ \new{\gamma^{-1/2}}\n{\Pout^\perp \Zx^*}.
$

We recall that $\Pin = \vvec{\Pinx & 0 \\ 0 & I}$, so that $\n{\Pin^\perp S^*}=\n{\Pinx^\perp \Sx^*}$, and both terms $\n{\Pinx^\perp \Sx^*}$ and $\n{\Pout^\perp \Zx^*}$ can be bounded using Lemma \ref{r:bound_Pin_Pout_halfcov}, which is itself derived from \cite{rudi2015less}. We thus get
\begin{align}
	\n{\nG-\G}
	&≤ \prt*{\frac{κ}{\nnew{γ}}+\new{\frac{1}{\gamma^{1/2}}}} 4\supfmap\sqrt{\frac{3}{m}\log\prt*{\frac{8m}{5δ}}} \nonumber\\
	&\quad+ \frac{48\supfmap^3}{γ^{3/2}} \frac{1}{m}\log\prt*{\frac{8m}{5δ}}.
\end{align}

\subsection{Error induced by the Nyström approximation}
\label{ss:nystroem_error}

The following lemma is a restatement of~\cite[Lemma 6]{rudi2015less} with slightly different constants. Albeit being originally written in a random design scenario, the same result holds when conditioning on the data.
\new{In the following, we denote $\kron$ the outer-product in $\cH_1$, \ie the rank-one operator defined as $(a \kron b) c= \ip{b,c}_{\cH_1} a$.}
\begin{lemma}%
	\label{uniform_nys_approx}
	Let $x₁,…,x_n∊\bR^d$ and denote $\cC=\tfrac{1}{n} \sum_{i=1}^n ψ(x_i)\kron ψ(x_i)$, where $ψ$ is the canonical feature map associated to a kernel satisfying Assumption \ref{a:bounded_kernel}.  
	Let $\ldm{1},…,\ldm{m}$ be drawn uniformly from all partitions of cardinality $m$ of $\cb{x₁,…,x_n}$. 
	Denoting $\Pi_m$ the orthogonal projection onto $\spa(ψ(\ldm{1}),…,ψ(\ldm{m}))$, 
	for any $γ∊]0,\n{\cC}]$, it holds 
	$ \n{(I-\Pi_m) (\cC+γI)^{1/2}}^2 ≤ 3γ,$
	with probability at least $1-δ$ provided 
	$ m\geq ( 2+ 5 \tfrac{κ²}{γ} )\log\prt*{\tfrac{4κ²}{γδ}}. $ 
\end{lemma}
\begin{proof}
	We apply exactly the same proof as in \cite[Lemma 6]{rudi2015less} but conditioning on the data (i.e., in particular applying \cite[Prop. 8]{rudi2015less} with the $v_i$ drawn \tiid according to the empirical distribution and $Q$ being the empirical covariance), with $\cN_∞(γ)=κ²/γ$. 
	Denoting $w=\log\prt*{\tfrac{4\supk}{γδ}}$, we end up with the bound
	$
		\n{(I-\Pi_m)(\cC+γI)^{1/2}}² ≤ \frac{γ}{1-β(γ)} ,
	$
	where
	$
		β(γ) ≤ \frac{2w}{3m} + \sqrt{\frac{2w\supk}{γ m} }.
	$
	We now derive slightly better constants than in the original lemma in order to satisfy $β(γ)≤2/3$ (which gives the claimed result). Indeed,
$
		β(γ) ≤ 2/3
		⇔ m - \frac{3}{2} \sqrt{\frac{2w\supk}{γ}} \sqrt{m} - w ≥ 0. 
$
	We solve this inequality as a second-order equation in $\sqrt{m}$, with discriminant $Δ=w(\frac{9\supk}{2γ}+4)$, which is positive whenever $γ<4\supk δ^{-1}$.  
	A sufficient condition to satisfy $β(γ)≤2/3$ is thus
	$
		m ≥ \prt*{\frac{3}{4} \sqrt{\frac{2w\supk}{γ}}+ \frac{1}{2} \sqrt{w(\frac{9\supk}{2γ}+4)} }².
	$
	Using the identity $2(a²+b²)≥(a+b)²$, we get the sufficient condition
	$
		m 
		≥ w \prt*{ 5 \frac{\supk}{γ} + 2 }.
	$
\end{proof}

\begin{lemma}
	\label{r:bound_Pin_Pout_halfcov}
	Under Assumption \ref{a:bounded_kernel}, 
	provided that $\Pinx$ and $\Pout$ are built by sampling uniformly $m$ samples respectively from $\cb{\bf x₁,…,\bf x_n}$ and $\cb{\bf x₂,…,\bf x_{n+1}}$, it holds with probability $1-δ$
$
		\max(\n{\Pinx^\perp \Sx^*}, \n{\Pout^\perp \Zx^*}) 
		≤ 4\supfmap\sqrt{\frac{3}{m}\log\prt*{\frac{8m}{5δ}}}\nonumber .
$
\end{lemma}
\begin{proof}
	Using a polar decomposition and applying Lemma \ref{uniform_nys_approx} with $m ≥ ( 2+ 5 \tfrac{κ²}{γ} ) \log\prt*{\tfrac{4κ²}{γδ}}$, it holds
	\begin{align}
		\n{\Pinx^\perp \Sx^*} 
		&= \n{\Pinx^\perp (\Sx^*\Sx)^{1/2}}  \\
		&≤ \n{\Pinx^\perp (\Sx^*\Sx+γI)^{1/2}} 
		 ≤ \sqrt{3γ},
	\end{align}
	and a similar argument holds for $\n{\Pout^\perp \Zx^*}$. \\
	The condition
	$
		m 
		≥ ( 2+ 5 \tfrac{κ²}{γ} ) \log\prt*{\tfrac{4κ²}{γδ}}
	$
	is in particular satisfied provided that 
	$
	m/2 ≥ \max(2, 5 \tfrac{κ²}{γ}) \log\prt*{\tfrac{4κ²}{γδ}}
	$, 
	which holds taking $γ=\frac{16κ²}{m}\log(\tfrac{4m}{5δ})$. 
	This yields the claimed result via a union bound.
\end{proof}

\section{Proof sketch for Lemma~\ref{lemma:convergence_p_operator}}
\label{appendix:rates_riccati}
Let $L=A + BK$ be the closed-loop operator of the system specified by~\eqref{eq:dynamics_compound}, where $K$ is the stabilizing gain defined in~\eqref{e:def_K}. The proof of the result, which closely follows the proof of~\cite[Proposition 2]{mania2019certainty}, starts by defining a set $\cS$ of perturbations of the Riccati operator $P$ as
$
	\cS = \{X:\cH_1\to\cH_1, X = X^*, P + X \succcurlyeq 0\}.
$
Then, we can derive a fixed-point operator equation
\begin{equation}
	X = \Phi X,\label{e:fixed_point_eq_riccati}
\end{equation}
where 
$
	\Phi X = \cTinv[F(X + P, A, B) - F(X + P, \Atilde, \Btilde) - \mathcal R X]$,
	$\cT X = I - L^*XL$, and 
	 $\cR X = L^*X[I + B^*PB(P + X)]^{-1}B^*PBXL.$
Let us define $\CLHo$ as the space of continuous linear operators from $\cH_1$ to $\cH_1$, and $\cL(\CLHo)$ as the space of continuous linear operators from $\CLHo$ to $\CLHo$. The operator $\cT:\CLHo\to\CLHo$ is invertible in $\cL(\CLHo)$. This fact can be proven by showing that $1$ is not in the spectrum of operator $\cD:\CLHo\to\CLHo$, defined as $\cD X = L^*XL$. This is achieved by upper bounding the spectral radius of $\cD$ by means of Gelfand's formula, and by observing that the closed-loop dynamics are stable, i.e., $L$ has spectral radius smaller than 1. Moreover, $\cTinv$ is bounded. This can be proven by showing that the Neumann series $\sum_{k=0}^\infty \cD^k$ is absolutely convergent in operator norm, and therefore equal to $\cTinv$.
Then, it can be shown that the unique solution of~\eqref{e:fixed_point_eq_riccati} belonging to $\cS$ is \new{$\Ptilde-P$}, since, by construction of the fixed point equation, it must hold $F(X + P, \Atilde, \Btilde) = 0$. 

We can now define the following closed subset of $\cS$,
$
	\mathcal S_{\nu} = \{X :\norm{X}\leq \nu, X = X^*, P + X \succcurlyeq 0\}.
$
Let $\norm{\Atilde -A}\leq \epsilon$, $\norm{\Btilde -B}\leq \epsilon$, where $\epsilon$ is the error rate in Theorem~\ref{r:bound_G_nG_opnorm}. \new{Let $N\de BR^{-1}B^*$.} If we pick
\begin{align*}
	\nu= 	\min\left\{\norm{N}^{-1}, \frac12\right., &6\left.\epsilon \frac{\tau(L, \new\zeta)^2}{1 - \new\zeta^2}\normp A^2 \normp P^2\right.\\
	&\left. \cdot\normp B \normp {R^{-1}} \right\},\nonumber
\end{align*}
for $X, X_1, X_2\in \cS_\nu$, we can show that $\Phi X\in \cS_\nu$, and $\exists \eta <1$ such that $\norm{\Phi X_1 - \Phi X_2} \leq \eta \norm{X_1- X_2}$, that is, $\Phi$ is a contraction operator. Then, we can apply Banach fixed-point theorem, meaning that the self-adjoint fixed point $\Ptilde - P$ of $\Phi$ belongs to $\cS_{\nu}$, i.e., the error on the Riccati operator due to the Nyström approximation has bounded operator norm as a function of $\epsilon$.

\section{Computable expression for Nyström vector dynamics and $C$}
\label{appendix:computable_expressions}
In this appendix, we report the computable expression for the operator 
appearing 
in~\eqref{eq:nystrom_dynamics_vector_partial}, as well as the expression of the reconstruction matrix $C$ in \eqref{eq:reconstruction}. 
They are the formulas used in practice for the implementation of the method. To simplify the notation, for operators $M, N$, we denote a block-diagonal operator as $\bd{M, N} = \vvec{M &0\\ 0&N}$.
We introduce the following additional notation. In the sequel we refer to the nonlinear kernel $k$:
\begin{itemize}[topsep=0pt]
	\item $K_{nm, \textrm{out}}$ is the kernel between the regression outputs and the output landmarks;
	\item $K_{nm, \textrm{in}}$ is the kernel between the regression inputs (state only) and the input landmarks;
	\item $K_{m, \textrm{in}}$ is the kernel at the input landmarks;
	\item $K_{m, \textrm{in},\textrm{out}}$ is the kernel between the input landmarks and the output landmarks.
\end{itemize}

Let us define the sampling operator for the input Nyström landmarks as
$
	\Stildein:\Hspace_1\rightarrow\R^m,\
	\Stildein g = [g(\ldmi{1}), …,g(\ldmi{m})]ᵀ.
$
Note that $\ran(\Stildein^*)=\Hin$. 
We denote 
$\Stildein = U\Sigma V^*$ its 
singular value decomposition, where $U:\R^t\rightarrow\R^m$, $\Sigma:\R^t\rightarrow\R^t$, $V:\R^t\rightarrow \tilde \cH_1$. 
Here, $\Sigma$ is diagonal and strictly positive. It holds $\Pin = \bd{VV^*, I}$ and thus
\nnew{
\begin{align*}
	&\Bout^* \Zx^* (S\Pin S^* + \gamma I)^{-1} S\bd{\Pinx\Bout , I } \\
	&= \Bout^* \Zx^*S\bd{\Stildein^* , I } \bd{U\Sigma^{-1} , I } \\
	&\quad\cdot\prt*{\bd{V^* , I } S^*S \bd{V , I }+ \gamma I}^{-1}\\
	&\quad\cdot \bd{\Sigma^{-1}U^*, I } \bd{\Stildein , I }\bd{\Bout , I }.
\end{align*}}%
Now, we can observe that $V^*V=I$. Moreover, $\prt*{\bd{V^* , I } S^*S \bd{V , I }+ \gamma I}$ is full-rank, $\bd{U\Sigma^{-1} ,I }$ is full column-rank and $\bd{\Sigma^{-1} U^*, I }$ is full row-rank, 
then by \cite[Section 1.6]{adiben-israel2003GeneralizedInversesTheory} we have that
\nnew{
\begin{align*}
	&\bd{U \Sigma^{-1}, I }\prt*{\bd{V^* , I } S^*S \bd{V , I }+ \gamma I}^{-1}\bd{\Sigma^{-1}U^*, I }\\
	&\quad= \prt*{\bd{\Stildein , I } S^*S \bd{\Stildein^* , I }+ \gamma \bd{\Stildein\Stildein^*, I }}^\dagger.
\end{align*}}
This yields
\nnew{\begin{align*}
	&\Bout^*\nG \bd{\Bout , I}\\
	&\quad= (K_{m, \textrm{out}}^{\dagger})^{1/2}K_{mn, \textrm{out}}\vvec{K_{nm, \textrm{in}} &\sqrt n S_{\uvec}}\\
	&\qquad\cdot\prt*{\vvec{K_{mn, \textrm{in}}\\ \sqrt n\Suᵀ}
		\vvec{ K_{nm, \textrm{in}} &\sqrt n \Su }
		\!+\! \gamma n \bd{
			K_{m,\textrm{in}} , I } }^{\dagger} \\
	&\qquad\cdot\bd{K_{m, \textrm{in},\textrm{out}} (K_{m, \textrm{out}}^†)^{1/2}, I}.
\end{align*}}%
Moreover, having stacked the regression outputs in a matrix $X_{+1}\in \R^{d\times n}$, the matrix $C$ from \eqref{eq:reconstruction} takes the form
$
	C = X_{+1}K_{nm, \textrm{out}}
	(K_{mn, \textrm{out}}K_{nm, \textrm{out}} + \lambda n K_{m, \textrm{out}})^{-1}K_{m,\textrm{out}}^{1/2}.
$
\bibliographystyle{abbrv}   
\bibliography{biblio} 
\newpage
\onecolumn
\section{Table of notations}
\begin{table*}[hb!]
	\caption{Summary of the main notations used in our article.}
	\label{table:notation}
	{
		\renewcommand{\arraystretch}{1.0}
		\begin{tabular}{ll}
			\toprule
			\textbf{Variable} & \textbf{Meaning} \\
			\midrule
			$\cH_1$ & RKHS associated to $k$ \\
			$\cH=\cH₁×\bR^{n_u}$ & product space for lifted state and control \\
			\midrule
			$\xvec \in\bR^d$, 
			$\uvec \in\bR^{n_u}$ 
			& state,
			control input\\
			$\uveco_i \in\bR^{n_u}$,$\tuveco_i \in\bR^{n_u},\ i\in\mathbb N$ 
			& control sequences for the optimal gains, defined in \eqref{eq:obj}, \eqref{eq:nobj} \\
			$\ldmi_i \in\bR^{d}, \ldmo_i \in\bR^{d},\ i=1, \dots, m$ 
			& Nyström input and output landmarks\\
			$z \in\Hspace_1$ & lifted state (infinite dimensional), see \eqref{eq:dynamics_compound} \\
			$\tilde z \in\Hspace_1$ & Nyström lifted state (infinite dimensional), see~\eqref{eq:nystrom_dynamics_compound}\\
			$\tildezvec \in\bR^{ m}$ & lifted state (finite dimensional, Nyström features), see~\eqref{eq:nystrom_dynamics_vector_partial} \\
			$\wvec=[\xvec^T,\uvec^T]^T\in \bR^{d+n_y}$ & augmented states (concatenation of state and control input) \\
			$\psi: \bR^d \rightarrow \cH_1$ & transformation for the state \\
			$\phi : \bR^{d+n_u} \rightarrow \cH_1 \times \bR^{n_u} $ & transformation for state and control input, see \eqref{e:def_phi} \\
			\midrule
			$S:\Hspace\rightarrow\R^n$ & sampling on the input with normalization $n^{-1/2}$, see \eqref{e:def_S} \\
			$\Zx:\cH_1\rightarrow\bR^n$ & sampling on the output with normalization $n^{-1/2}$, w/o the control input, see \eqref{e:def_Zx}\\
			$\Sx:\cH_1\rightarrow\bR^n$ & sampling on the input with normalization $n^{-1/2}$, w/o the control input, see \eqref{e:def_Sx} \\
			$\Su:\bR^{n_u}\rightarrow\bR^n$ & sampling on the control input with normalization $n^{-1/2}$, see \eqref{e:def_Su} \\
			\midrule
			$\Pin:\Hspace\rightarrow\Hspace$ &block diagonal projection for input space (state and control input)\\
			$\Pinx:\Hspace_1\rightarrow\Hspace_1$&orthogonal projection on input state\\
			$\Pout:\Hspace_1\rightarrow\Hspace_1$&orthogonal projection on output state, can be equal to $\Pinx$\\
			$\Stildeout:\Hspace_1\rightarrow \R^m$&sampling operator for Nyström centers (regression output), see \eqref{e:def_Stout}\\
			\midrule
			$\G=[\A,\B]:\cH → \cH₁$ & estimated (state,input)-state transition operator\\
			$\A :\Hspace_1\rightarrow\Hspace_1$ &estimated state-state transition operator, also appearing as $A$\\
			$\B :\R^{n_u}\rightarrow\Hspace_1$ &estimated input-state transition operator, also appearing as $B$\\
			$\nG=[\nA, \nB]:\cH → \cH₁$ & Nyström (state,input)-state transition operator\\
			$ \nA :\Hspace_1\rightarrow\Hspace_1$ &Nyström approximation of state-state transition operator, also appearing as $\Atilde$\\
			$ \nB :\R^{n_u}\rightarrow\Hspace_1$ &Nyström approximation input-state transition operator, also appearing as $\Btilde$\\
			$ P :\cH_1\rightarrow\Hspace_1, \Ptilde :\cH_1\rightarrow\Hspace_1$ &Riccati operator with exact and Nyströmized kernel, see~\eqref{e:def_P} and \eqref{e:def_nP}\\
			$\K :\cH_1\to\R^{n_u}, \Ktilde :\cH_1\to\R^{n_u}$ &LQR optimal gain with exact and Nyströmized kernel, see~\eqref{e:def_K} and \eqref{e:def_nK}\\
			\bottomrule
		\end{tabular}
	}
\end{table*}
\section{Proof of Lemma~\ref{lemma:convergence_p_operator}}
	Note that in our case the operators $Q$ and $R$ are the same both for the exact and Nyströmized dynamics, and are assumed to be positive definite. The structure of the proof is as follows. First, we write down a fixed-point operator equation, $X=\Phi X$, for a suitable operator $\Phi$, that has a unique solution, given by the error on the Riccati operator due to the Nyström approximation, $\Ptilde - P$. Then, we define a set of perturbations $X$ with bounded norm, and we show that $\Phi$ maps this set to itself, and is a contraction. This implies that the fixed point of $\Phi$, which is shown to be $\Ptilde-P$, lies in the set we defined, and has bounded norm. We begin by defining 
	\begin{align}
		N:\cH_1\to\cH_1,\ N &= B R^{-1}B^*, 
			\label{e:def_N} \\
		\Ntilde:\cH_1\to\cH_1,\ \Ntilde &= \Btilde R^{-1}\Btilde^*.
			\label{e:def_nN}
	\end{align}

\begin{proposition}
	\label{prop:inversion_perturbed_riccati}
	Given $P$ as the self-adjoint positive semi-definite solution of~\eqref{e:def_P}, let us consider the following set of self-adjoint perturbations 
	\begin{equation}
		\cS = \{X:\cH_1\to\cH_1, X = X^*, P + X \succcurlyeq 0\}.
	\end{equation} 
Let $X \in \cS$. Let  $L = A + BK$, where $K$ is defined in~\eqref{e:def_K}. Then, it holds
	\begin{equation}
	F(P + X, A, B) = \underbrace{X - L^*XL }_{\eqqcolon \mathcal T X}+\underbrace{ L^*X[I + N(P + X)]^{-1}NXL}_{\eqqcolon \mathcal R X}.\nonumber\label{eq:perturbed_riccati}
	\end{equation}
\end{proposition} 
\begin{proof}
This statement can be found, in the matrix setting, in~\cite{mania2019certainty,konstantinov1993perturbation}. We report here the derivation for completeness. By definition of $P$ it holds $F(P,A,B)=0$ thus:
\begin{align}
	 &X - L^*XL + L^*X[I + N(P + X)]^{-1}NXL\nonumber \\
	 &= X - L^*XL + L^*X[I + N(P + X)]^{-1}NXL + F(P,A,B)\\
	 &= X - L^*XL + L^*X[I + N(P + X)]^{-1}NXL 
	 	+ P-A^*P(I+NP)^{-1}A -Q \\
	 &= (P+X)-Q + \brk*{ L^*X[I + N(P + X)]^{-1}NXL 
		 - L^*XL 
	 	 -A^*P(I+NP)^{-1}A } \\
	 &= (P+X)-Q + \brk*{ 
	 	 L^*X\prt*{[I + N(P+X)]^{-1}NX-I}L 
	 	 -A^*P(I+NP)^{-1}A 
	 	} \\
	 &= (P+X)-Q + \Big[
	 	 L^*X\prt*{[I + N(P+X)]^{-1}NX-[I + N(P+X)]^{-1}[I + N(P+X)]}L \\
	 	 &\quad -A^*P(I+NP)^{-1}A 
	 	\Big] \nonumber\\
	 &= (P+X)-Q - \brk*{ 
	 	 L^*X\prt*{[I + N(P+X)]^{-1}[I + NP]}L 
	 	 +A^*P(I+NP)^{-1}A 
	 	} \\
	 &= (P+X)-Q - A^*\brk*{ 
	 	 ((I+NP)^*)^{-1}X[I+N(P+X)]^{-1}
	 	 +P(I+NP)^{-1}}A 
	 	 \\
	 &= (P+X)-Q - A^*\brk*{ 
	 	 (I+PN)^{-1}X
	 	 +P(I+NP)^{-1}[I+N(P+X)]}[I+N(P+X)]^{-1}A 
	 	 \\
	 &= (P+X)-Q - A^*\brk*{ 
	 	 (I+PN)^{-1}X
	 	 +P(I+NP)^{-1}NX
	 	 +P
		 }[I+N(P+X)]^{-1}A 
	 	 \\
	 &= (P+X)-Q - A^*(P+X)(I+N(P+X))^{-1}A\\
	 &= F(P+X, A, B),
\end{align}
which concludes the proof.
\end{proof}

We can now study the operator $\mathcal T$ appearing in Proposition~\ref{prop:inversion_perturbed_riccati}.
\begin{proposition} 
	\label{prop:T_inv}
	Let $\CLHo$ denote the set of continuous linear operators from $\cH₁$ to itself.
	The operator $\mathcal T: \CLHo→\CLHo$ appearing in Proposition~\ref{prop:inversion_perturbed_riccati} is invertible. 
	Under assumptions~\ref{a:stabilizability} and~\ref{a:detectability}, the fixed point equation 
	\begin{equation}
	X = \mathcal T^{-1} [F(X + P, A, B) - F(X + P, \Atilde, \Btilde) - \mathcal R X] \nonumber
	\end{equation}
has a unique solution in $\cS$, namely $X = P-\Ptilde$.
\end{proposition}
\begin{proof} The operator $\mathcal T$ can be rewritten as 
\begin{equation}
	\mathcal T X =( I - \mathcal D)X
\end{equation}
where $\mathcal D X = L^*XL:\CLHo\to\CLHo$ is linear ($\mathcal D (X + Y) = L^*XL + L^*YL,  \mathcal D (\alpha X) =\alpha L^*XL$).
Let $\cL(\CLHo)$ denote the set of continuous linear operators from $\CLHo$ to itself. 
Then $\mathcal T∊\cL(\CLHo)$ (because $A, B, Q, R$, and $K$, defined in~\eqref{e:def_K}, appearing in the expression of $L$ are bounded). 
We denote the spectrum of $\mathcal D$ as 
\begin{equation}
	\sigma(\mathcal D) = \{\lambda : \mathcal D - \lambda I \text{ is not  invertible in $\cL(\CLHo)$}\},
\end{equation}
and its spectral radius $\rho(\mathcal D) = \sup_{\lambda \in \sigma(\mathcal D)} |\lambda|$.
The operator $\mathcal T$ is invertible in $\cL(\CLHo)$ iff 1 is not in the spectrum $\sigma(\mathcal D)$, and we now show that this holds by proving $ρ(\cD)<1$.

We can observe that since $K$ in~\eqref{e:def_K} is a stabilizing gain by assumption, it holds~\cite[page 4]{hager1976convergence} $\rho(L) <1$. The spectral radius of $\cD $ can be computed with the Gelfand's formula,
using the fact that $L^k, (L^*)^k$ are bounded for any $k$:\\
\begin{align}
	\rho(\mathcal D) &= \lim_{k\to\infty}\norm{\mathcal D^k}^{1/k} \\
	&=  \lim_{k\to\infty}\sup_{\norm X = 1}\norm{\mathcal D^k X}^{1/k} \\
	&= \lim_{k\to\infty}\sup_{\norm X = 1}\norm{ {L^*}^k	XL^k}^{1/k} \\
	&\leq \lim_{k\to\infty}\norm{ {L^*}^k}^{1/k}\norm{L^k}^{1/k}\\
	&= \rho(L)^2 \\
	&< 1.
\end{align}
Now, we can consider the following equation:
\begin{equation}
	\label{eq:perturbed_riccati_with_nystrom}
	F(X + P, A, B) - F(X + P, \Atilde, \Btilde) = \mathcal T X + \mathcal R X.
\end{equation}
By Proposition~\ref{prop:inversion_perturbed_riccati} we have
\begin{equation}
	F(X + P, A, B )= \cT X + \cR X,
\end{equation}
and thus it must hold 
\begin{align}
F(X + P, \Atilde, \Btilde) &= 0.
	\label{e:prop_nP}
\end{align}
By definition of $\nP$ in \eqref{e:def_nP},
$X=\nP-P$ satisfies this equation, 
but it is also known \cite[Theorem 9]{hager1976convergence} that under Assumptions~\ref{a:stabilizability} and~\ref{a:detectability}, \eqref{e:prop_nP} admits a unique self-adjoint positive semi-definite solution. Moreover as both $P$ and $\nP$ are self-ajdoint, $X=\nP-P$ is self-adjoint as well and thus the unique solution of \eqref{e:prop_nP} in $\cS$.

Also, note that~\eqref{eq:perturbed_riccati_with_nystrom}, for a suitable operator $\Phi:\cS\to\CLHo$, can be rewritten as 
\begin{equation}
	X = \underbrace{\mathcal T^{-1} [F(X + P, A, B) - F(X + P, \Atilde, \Btilde) - \mathcal R X]}_{\eqqcolon \Phi X},
	\label{e:def_Phi}
\end{equation} 
so that $\nP-P$ is the unique fixed point of $Φ$.
\end{proof}

We have shown that the operator $\Phi$ is well defined and has a unique fixed point in $\cS$ equal to $\Ptilde - P$, and we will now show that the latter is upper bounded by $\epsilon$ in operator norm, where $\epsilon$ is the error on the matrices of the dynamical system, similarly to~\cite{mania2019certainty}. 
As mentioned at the beginning of this section, in the following, we will define a set $\cS_{\nu}$ of perturbations $X\in\cS$ with bounded norm. Furthermore, we will show that the function $\Phi$ is a mapping from $\cS_ν$ to itself, and that it is a contraction (i.e. $\alpha$-Lipschitz, $\alpha < 1$). Therefore, since $\cS_\nu$ is a closed subset of $\cS$, by the Banach fixed point theorem, $\Phi$ has a fixed point in $\cS_{\nu}$. Since the fixed point of $\Phi$ in $\cS$ has been proven to be $ \Ptilde - P$, this means that the error on the $P$ operator due to the Nyström approximation is bounded, and the bound depends on the error rate of the system's matrices. 
More formally, let us define:
\begin{equation}
	\mathcal S_{\nu} = \{X :\norm{X}\leq \nu, X = X^*, P + X \succcurlyeq 0\}.
\end{equation}
We can now prove some technical bounds that are used in~\cite{mania2019certainty}. We can define
\begin{align}
	\Delta A &\coloneqq \nA - A,\\
	\Delta B &\coloneqq \nB - B,\\
	\Delta N &\coloneqq \nN - N,\\
	\Delta P &\coloneqq \nP - P.
\end{align}
where we recall that $N,\nN$ are defined in \ref{e:def_N},~\ref{e:def_nN}.
\begin{proposition}
		\label{prop:bound_norm_delta_N}
	Assume $\norm{\Delta A}\leq \epsilon$,  and $\norm {\Delta B} \leq \epsilon$. Then, for $\epsilon\leq \norm{B}$, the following bound holds: 
	\begin{equation}
	\norm{\Delta N}\leq 3\epsilon \norm{R^{-1}}\norm{B}.\nonumber
	\end{equation}
\end{proposition}
\begin{proof}We can observe that: 
\begin{align}
	\norm{\Delta N}& = \norm{BR^{-1} B^* - \Btilde R^{-1} \Btilde^*}\\
	&=  \norm{BR^{-1} B^* - \Btilde R^{-1} B^* + \Btilde R^{-1} B^*-\Btilde R^{-1} \Btilde^*}\\
	&= \norm{(B - \Btilde)R^{-1} B^*  + \Btilde R^{-1} (B^*-\Btilde^*)}\\
	&\leq \norm{(B - \Btilde)R^{-1} B^*}  + \norm{\Btilde R^{-1} (B^*-\Btilde^*)}\\
	&\leq \norm{B- \Btilde}\norm{R^{-1}}(\norm{B} + \norm{\Btilde})\\
	&\leq \epsilon \norm{R^{-1}}(\norm{B} + \norm{\Btilde}).
\end{align}
We can observe that, by choosing $\epsilon\leq \norm B$, by the reversed triangle inequality:
\begin{align}
	\norm{\Btilde} - \norm{B}
		& \leq \norm{\Btilde - B}
		 \leq \epsilon
		 \leq \norm{B},
\end{align}
meaning that 
\begin{equation}
	\norm{\Delta N}\leq 3\epsilon \norm{R^{-1}}\norm{B},
\end{equation}
as also shown in~\cite{mania2019certainty}.
\end{proof}
\begin{proposition}\label{prop:norm_T_inv}
	Let $L$, $\cT$ be as defined in Proposition~\ref{prop:inversion_perturbed_riccati}. Moreover, let
	\begin{equation}
		\tau(L, \rho)=\sup\{\norm {L^k}\rho^{-k} : k≥ 0\}.\nonumber
	\end{equation}
	By assumption $K$ stabilizes the system, \ie $ρ(L)<1$.
	For any $\rho(L)\leq \rho<1$ we have that
	\begin{equation}
		\norm{\cTinv} \leq \frac{ \tau(L, \rho)^2}{1-\rho^2}.\nonumber
	\end{equation}
\end{proposition}
\begin{proof}
	By definition we have that operator $\cT$ can be expressed as $\cT= I - \cD$,
where $\cD X = L^*XL$.
Note that $L$ is bounded because $A,B$ and $K$ are bounded,
and thus $\cD:\CLHo→\CLHo$ is also bounded as
\begin{align}
	\norm{\cD} &= \sup_{\norm X = 1}\norm{\cD X}
	= \sup_{\norm X = 1}\norm{L^* XL}
	\leq \norm L^2.
\end{align}
Now, observe that, since the spectral radius of $L$ is strictly smaller than 1 (stability), we can choose $\rho(L)< \rho<1$ so that the quantity $\tau(L, \rho)$ is finite~\cite{mania2019certainty}.
Indeed according to Gelfand's formula, we are guaranteed that there exists $k_0$ such that for any $k\geq k_0$, $\norm{L^k}\leq \rho^k$.
Hence, we can choose 
\begin{equation}
	\tau'(L, \rho) =\max \cb{1} \cup \cb{ \norm{L^k}/\rho^k : k∊\irange{0,…,k_0}}
\end{equation}
which involves a finite sequence of values, so that for any $k∊\bN$ it holds $\n{L^k}≤\tau'(L, \rho) ρ^k$, and $τ(L,ρ)≤τ'(L,ρ)<∞$.
The Neumann series $\sum_{k=0}^{\infty} \cD^k$ is (absolutely) convergent in operator norm, as
\begin{align}
	\sum_{k=0}^{\infty}\norm{\cD^k} 
	&=\sum_{k=0}^{\infty} \sup_{\norm X=1}\norm{{L^*}^k X L^k}\\
	&\leq \sum_{k=0}^{\infty} \norm{L^k}^2\\
	&\leq \sum_{k=0}^{\infty}\left[\tau(L, \rho)\rho^k\right]^2\\
	&=\frac{ \tau(L, \rho)^2}{1-\rho^2},
\end{align}
since we have a geometric series with $\rho^2 < 1$.
Therefore, we can express $\cTinv$ using the Neumann series as
\begin{align}
	\cTinv &= (I - \cD)^{-1}
	= \sum_{k=0}^{\infty}\cD^k.
\end{align} 
This yields the claimed result by the triangle inequality.
\end{proof}
\begin{proposition}
	\label{prop:norm_r_x}.
	Let $\cR$, $L$, $N$, $X$ be the operators defined in Proposition~\ref{prop:inversion_perturbed_riccati}. Then, we have
	\begin{equation}
	\norm{\cR X}\leq \norm{L}^2\norm{X}^2\norm{N}.\nonumber
	\end{equation}
\end{proposition}
\begin{proof} The first part of the proof is the same as the one of Lemma 7 in~\cite{mania2019certainty}. Note that $N$ and $P+X$ are self-adjoint, positive semi-definite operators. 
		Let us now consider $\alpha >0$. We have that
		\begin{align}
			0 \preccurlyeq 
			(N+\alpha I)[I + (P + X+\alpha I)(N+\alpha I)]^{-1}&=[(N+\alpha I)^{-1} + P + X+\alpha I]^{-1}\\
			&\preccurlyeq(N+\alpha I).
		\end{align}		
	This implies the following norm inequality:
	\begin{equation}
		\norm{(N+\alpha I)[I + (P + X+\alpha I)(N+\alpha I)]^{-1}}\leq\norm{N+\alpha I}.
	\end{equation}		
By definition, since limits preserve inequalities and by continuity of the norm, we have that
\begin{align}
	\norm{N[I + (P + X)N]^{-1}}
	&=\norm{\lim_{\alpha\to0} (N+\alpha I)[I + (P + X+\alpha I)(N+\alpha I)]^{-1}} \\
	&=\lim_{\alpha\to0}\norm{(N+\alpha I)[I + (P + X+\alpha I)(N+\alpha I)]^{-1}}\\
	&\leq \lim_{\alpha\to0}\norm{N+\alpha I}\\
	&=\norm{N}.
\end{align}

 Hence,
\begin{align}
	\norm{\cR X}&=\norm{L^*X[I + N(P + X)]^{-1}NXL} \\
	&= \norm{L^*XN[I + (P + X)N]^{-1}XL}\\
	&\leq \norm{L}^2\norm{X}^2\norm{N}.
\end{align}
\end{proof}
\begin{proposition}
	With the same notations and assumptions of Proposition~\ref{prop:inversion_perturbed_riccati}, for $X\in \cS_ν$, we have that 
\begin{equation}
\norm{F(P+X, \Atilde, \Btilde) - F(P+X,A,B)}\leq \norm{A}^2\norm{P+X}^2\norm{\Delta N}+2\norm{A}\norm{P+X}\epsilon + \norm{P+X}\epsilon^2.\nonumber
\end{equation}
\end{proposition}
\begin{proof}
Let $P_X$ be a shorthand for $P + X$.
We can consider the following decomposition~\cite{mania2019certainty}:
\begin{align}
	F(P_X, \Atilde, \Btilde) - F(P_X,A,B) =& A^*\PX(I + N\PX)^{-1}A- \Atilde^*\PX(I + \tilde N\PX)^{-1}\Atilde\\
	=& A^*\PX(I + N\PX)^{-1}A -  A^*\PX(I + \Ntilde\PX)^{-1}A\nonumber\\
		&+  A^*\PX(I + \Ntilde\PX)^{-1}A -  A^*\PX(I + \Ntilde\PX)^{-1}\Atilde\nonumber\\
		&+  A^*\PX(I + \Ntilde\PX)^{-1}A -  \Atilde^*\PX(I + \Ntilde\PX)^{-1}A\nonumber\\
		&+  A^*\PX(I + \Ntilde\PX)^{-1}\Atilde  +\Atilde^*\PX(I + \Ntilde\PX)^{-1}A\nonumber \\
		&-A^*\PX(I + \Ntilde\PX)^{-1}A  -\Atilde^*\PX(I + \Ntilde\PX)^{-1}\Atilde\\
	=& A^*\PX\left[(I + N\PX)^{-1} -  (I + \Ntilde\PX)^{-1}\right]A\nonumber\\
		&-  A^*\PX(I + \Ntilde\PX)^{-1}\Delta A\nonumber\\
		&-  \Delta A^*\PX(I + \Ntilde\PX)^{-1}A\nonumber\\
		&-  \Delta A^*\PX(I + \Ntilde\PX)^{-1}\Delta A\\
	=&A^*\PX(I + N\PX)^{-1} \Delta N \PX  (I + \Ntilde\PX)^{-1}A\nonumber\\
		&-  A^*\PX(I + \Ntilde\PX)^{-1}\Delta A\nonumber\\
		&-  \Delta A^*\PX(I + \Ntilde\PX)^{-1}A\nonumber\\
		&-  \Delta A^*\PX(I + \Ntilde\PX)^{-1}\Delta A.
\end{align}
By~\cite[Lemma 7]{mania2019certainty}, we get
\begin{align}
\norm{F(P_X, \Atilde, \Btilde) - F(P_X,A,B)}&\leq \norm{A}^2\norm{\PX}^2\norm{\Delta N}+2\norm{A}\norm{\PX}\epsilon + \norm{\PX}\epsilon^2.
\end{align}
\end{proof}

With the technical results introduced above, we are now ready to show that $\Phi$ maps $\mathcal S_{\nu}$ to itself, and is a contraction. We can firstly compute suitable upper bounds, merging the results above.
\begin{proposition}
\label{prop:bound_tinv_X}
	With the same notations and assumptions of Proposition~\ref{prop:inversion_perturbed_riccati}, for $X\in\cS_ν$, $\nu\leq 1/2$, $\epsilon\leq \norm{B}$, and $\rho(L) \leq\rho < 1$, it holds that 
	\begin{equation}
	\norm{\Phi X} \leq  \frac{ \tau(L, \rho)^2}{1-\rho^2}\left[\norm{L}^2\norm{N}\nu^2 +3\epsilon(\norm{P} + 1)^2(\norm{A} + 1)^2(\norm{R^{-1}} + 1)(\norm{B} + 1) \right].\nonumber
	\end{equation}
\end{proposition}
\begin{proof}
We already know from Proposition~\ref{prop:bound_norm_delta_N} that $\norm{\Delta N}\leq 3\epsilon \norm{R^{-1}}\norm{B}$. Moreover, $\norm{\PX}\leq\norm{P} + \nu ≤ \n{P}+1$ by the triangle inequality and given that $X∊\cS_ν$ and $ν≤1/2$. According to Proposition~\ref{prop:T_inv},~\ref{prop:norm_T_inv},~\ref{prop:norm_r_x} and using $\epsilon\leq \norm{B}$:
\begin{align}
	\norm{\Phi X} &= \norm{\mathcal T^{-1} [F(X + P, A, B) - F(X + P, \Atilde, \Btilde) - \mathcal R X]}\\
	&\leq \norm{\mathcal T^{-1}}\left[\norm{F(X + P, A, B) - F(X + P, \Atilde, \Btilde)} + \norm{ \mathcal R X}\right]\\
	&\leq \frac{ \tau(L, \rho)^2}{1-\rho^2}\left[\norm{L}^2\norm{N}\nu^2 +\norm{A}^2\norm{\PX}^23\epsilon \norm{R^{-1}}\norm{B}+2\norm{A}\norm{\PX}\epsilon + \norm{\PX}\epsilon^2 \right]\\
	&\leq  \frac{ \tau(L, \rho)^2}{1-\rho^2}\left[\norm{L}^2\norm{N}\nu^2 +\epsilon\norm{\PX}\left(3\norm{A}^2\norm{\PX} \norm{R^{-1}}\norm{B}+2\norm{A} +\epsilon\right) \right]\\
	&\leq  \frac{ \tau(L, \rho)^2}{1-\rho^2}\left[\norm{L}^2\norm{N}\nu^2 +\epsilon\norm{\PX}\left(3\norm{A}^2\norm{\PX} \norm{R^{-1}}\norm{B}+2\norm{A} +\norm{B}\right) \right].
\end{align}
We use a rough upper-bounded on the second term of this last expression to get claimed result, which is more convenient to work with:
\begin{align}
	3\epsilon(\norm{P} + 1)&^2(\norm{A} + 1)^2(\norm{R^{-1}} + 1)(\norm{B} + 1)\\
	& =3\epsilon(\norm{P} + 1)^2(\norm{A} + 1)^2(\norm{R^{-1}}\norm{B} +\norm{R^{-1}} + \norm{B} + 1)\\
	&=3\epsilon(\norm{P} + 1)^2(\norm{A}^2 + 2\norm{A} + 1)(\norm{R^{-1}}\norm{B} +\norm{R^{-1}} + \norm{B} + 1)\\
	&=3\epsilon(\norm{P}^2 + 2\norm{P} + 1)(\norm{A}^2 + 2\norm{A} + 1)(\norm{R^{-1}}\norm{B} +\norm{R^{-1}} + \norm{B} + 1)\\
	&\geq 3\epsilon\norm{A}^2(\norm{P}+1)^2\norm{R^{-1}}\norm{B} + 12\norm{A}(\norm{P} + 1)\epsilon + 6(\norm{P} + 1)\norm{B}\epsilon)\\
	&\geq \norm{A}^2\norm{\PX}^23\epsilon \norm{R^{-1}}\norm{B}+2\norm{A}\norm{\PX}\epsilon + \norm{\PX}\epsilon^2.
\end{align}	
\end{proof}

\begin{proposition}
	\label{prop:bound_lipschitz_invt}
	With the same assumptions and notations of Proposition~\ref{prop:inversion_perturbed_riccati}, if $X_1,X_2\in\mathcal S_{\nu}$, $\nu\leq \min\{1/2, \norm{N}^{-1}\}$, $‖ΔA‖≤ε$, $‖ΔB‖≤ε$, $\epsilon \leq \norm{B}$, and $\rho(L) \leq\rho < 1$, it holds that
\begin{align}
	\norm{\Phi X_1 - \Phi X_2} \leq 32 \frac{\tau(L, \rho)^2}{1 - \rho^2}[&(\norm{A} + 1)^2(\norm{P} + 1)^3( \norm{B} + 1)^3(\norm R^{-1} + 1)^2 \norm{X_1 - X_2}\epsilon\nonumber \\
	&+ \norm{L}^2\nu\norm{N}\norm{X_1 - X_2}].\nonumber
\end{align}
\end{proposition}
\begin{proof}Let us choose $\nu\leq\min\{1/2, \norm{N}^{-1}\}$, and observe that:
\begin{align}
\norm{\cR X_1 - \cR X_2} \leq& \norm{L^*X_1 N[I + (P + X_1)N]^{-1}X_1L - L^*X_2 N[I + (P + X_2)N]^{-1}X_2L}\\
\leq& \norm{L}^2\norm{X_1 N[I + (P + X_1)N]^{-1}X_1 - X_1 N[I + (P + X_1)N]^{-1}X_2\nonumber\\
&+X_1 N[I + (P + X_1)N]^{-1}X_2-X_2 N[I + (P + X_1)N]^{-1}X_2\nonumber\\
&+ X_2 N[I + (P + X_1)N]^{-1}X_2
	- X_2 N[I + (P + X_2)N]^{-1}X_2}\\
\leq& \norm{L}^22\nu\norm{N}\norm{X_1 - X_2} \nonumber\\
&+ \norm{X_2 N[I + (P + X_1)N]^{-1}\left(X_2N - X_1N\right)[I + (P + X_2)N]^{-1}X_2}\\
=&  \norm{L}^22\nu\norm{N}\norm{X_1 - X_2} \nonumber\\
&+ \norm{X_2 N[I + (P + X_1)N]^{-1}\left(X_2 - X_1\right)N[I + (P + X_2)N]^{-1}X_2}\\
\leq &\norm{L}^2\left\{ 2\nu\norm{N} +\nu^2\norm{N}^2\right\}\norm{X_1 - X_2}\\
\leq & 3\norm{L}^2\nu\norm{N}\norm{X_1 - X_2}.
\end{align}
Again, let $\PX$ be a shorthand for $ P + X$. Having defined for convenience $\cG X = F(\PX, \Atilde, \Btilde) - F(\PX, A, B)$,
and using the definition of $\Phi$ in \eqref{e:def_Phi}, it holds
\begin{align}
	\norm{\Phi X_1 - \Phi X_2} =& \norm{\cTinv}\norm{\cG X_1 - \cG X_2 + \cR X_2 - \cR X_1}\\
	\leq&  \frac{\tau(L, \rho)^2}{1 - \rho^2}[\norm{\cG X_1 - \cG X_2} + 3\norm{L}^2\nu\norm{N}\norm{X_1 - X_2}].
\end{align}
Now, we can study $\cG$. Note that, as in~\cite{mania2019certainty} $\norm{(I + N\PX)^{-1}}\leq 2\norm{\PX}$, since $\nu\leq \frac12$. Then,
{\small
	\begin{align}
	\norm{\cG X_1 - \cG X_2}=&\norm{A^*{\PX}_1(I + N{\PX}_1)^{-1}A- \Atilde^*{\PX}_1(I + \tilde N{\PX}_1)^{-1}\Atilde \nonumber\\
		&- A^*{\PX}_2(I + N{\PX}_2)^{-1}A+ \Atilde^*{\PX}_2(I + \tilde N{\PX}_2)^{-1}\Atilde}\\
	=&\lVert A^*{\PX}_1(I + N{\PX}_1)^{-1} \Delta N {\PX}_1  (I + \Ntilde{\PX}_1)^{-1}A- A^*{\PX}_2(I + N{\PX}_2)^{-1} \Delta N {\PX}_2 (I + \Ntilde{\PX}_2)^{-1}A\label{e:cerulean}\\
	&-  A^*{\PX}_1(I + \Ntilde{\PX}_1)^{-1}\Delta A+  A^*{\PX}_2(I + \Ntilde{\PX}_2)^{-1}\Delta A\label{e:orange}\\
	&-  \Delta A^*{\PX}_1(I + \Ntilde{\PX}_1)^{-1}A+  \Delta A^*{\PX}_2(I + \Ntilde{\PX}_2)^{-1}A\label{e:purple}\\
	&-  \Delta A^*{\PX}_1(I + \Ntilde{\PX}_1)^{-1}\Delta A+  \Delta A^*{\PX}_2(I + \Ntilde{\PX}_2)^{-1}\Delta A\rVert\label{e:sepia}.
\end{align}}%
To control the difference in~\eqref{e:cerulean}, we can use the following inequality. Note that, by definition, $\norm N\leq \norm B^2 \norm {R^{-1}}$. Then, leveraging~\cite[Lemma 7]{mania2019certainty}:
{\small
	\begin{align}
		\norm{A^*{\PX}_1(I + &N{\PX}_1)^{-1} \Delta N {\PX}_1  (I + \Ntilde{\PX}_1)^{-1}A - A^*{\PX}_2(I + N{\PX}_2)^{-1} \Delta N {\PX}_2 (I + \Ntilde{\PX}_2)^{-1}A}\nonumber\\
		=& \norm{A^*{\PX}_1(I + N{\PX}_1)^{-1} \Delta N {\PX}_1  (I + \Ntilde{\PX}_1)^{-1}A - A^*{\PX}_1(I + N{\PX}_2)^{-1} \Delta N {\PX}_1(I + \Ntilde{\PX}_1)^{-1}A\nonumber\\
		&+ A^*{\PX}_1(I + N{\PX}_2)^{-1} \Delta N {\PX}_1  (I + \Ntilde{\PX}_1)^{-1}A - A^*{\PX}_2(I + N{\PX}_2)^{-1} \Delta N {\PX}_1 (I + \Ntilde{\PX}_1)^{-1}A\nonumber\\
		&+A^*{\PX}_2(I + N{\PX}_2)^{-1} \Delta N {\PX}_1  (I + \Ntilde{\PX}_1)^{-1}A - A^*{\PX}_2(I + N{\PX}_2)^{-1} \Delta N {\PX}_2 (I + \Ntilde{\PX}_1)^{-1}A\nonumber\\
		&+A^*{\PX}_2(I + N{\PX}_2)^{-1} \Delta N {\PX}_2  (I + \Ntilde{\PX}_1)^{-1}A - A^*{\PX}_2(I + N{\PX}_2)^{-1} \Delta N {\PX}_2 (I + \Ntilde{\PX}_2)^{-1}A}\\
	\leq& \norm{A}^2(\norm{P} + 1)\norm{\Delta N}\norm{P_{X_1}(I +N{\PX}_1 )^{-1} - P_{X_1}(I +N{\PX}_2 )^{-1}}\nonumber\\
	&+ \norm{A}^2\norm{\Delta N}(\norm P + 1)^3 \norm{X_1 - X_2}\nonumber\\
	&+ \norm{A}^2\norm{\Delta N}(\norm P + 1)^3 \norm{X_1 - X_2}\nonumber\\
	&+ \norm{A}^2(\norm{P} + 1)\norm{\Delta N}\norm{P_{X_2}(I +N{\PX}_1 )^{-1} - P_{X_2}(I +N{\PX}_2 )^{-1}}\\
	\leq& \norm{A}^2(\norm{P} + 1)\norm{\Delta N}\norm{P_{X_1}(I +N{\PX}_1 )^{-1} (I + NP_{X_2} - I - NP_{X_1})(I +N{\PX}_2 )^{-1}}\nonumber\\
	&+ \norm{A}^2\norm{\Delta N}(\norm P + 1)^3 \norm{X_1 - X_2}\nonumber\\
	&+ \norm{A}^2\norm{\Delta N}(\norm P + 1)^3 \norm{X_1 - X_2}\nonumber\\
	&+ \norm{A}^2(\norm{P} + 1)\norm{\Delta N}\norm{P_{X_2}(I + NP_{X_2})^{-1}(I + NP_{X_2} - I - NP_{X_1})(I +N{\PX}_1 )^{-1} }\\
	\leq& \norm{A}^22(\norm{P} + 1)^3\norm{\Delta N}\norm{N}\norm{X_2-X_1}\nonumber\\
	&+ \norm{A}^2\norm{\Delta N}(\norm P + 1)^3 \norm{X_1 - X_2}\nonumber\\
	&+ \norm{A}^2\norm{\Delta N}(\norm P + 1)^3 \norm{X_1 - X_2}\nonumber\\
	&+ \norm{A}^22(\norm{P} + 1)^3\norm{\Delta N}\norm{N}\norm{X_2 -X_1}\\
	\leq& 12\norm{A}^2(\norm{P} + 1)^3\norm{B}^3\norm{R^{-1}}^2\norm{X_2-X_1}\epsilon\nonumber\\
	&+ 6\norm{A}^2\norm{B}\norm{R^{-1}}(\norm P + 1)^3 \norm{X_1 - X_2}\epsilon.
\end{align}}
With a similar reasoning, we can compute an upper bound for the difference in~\eqref{e:orange}, as follows:
\begin{align}
		\norm{A^*{\PX}_2(I +& \Ntilde{\PX}_2)^{-1}\Delta A - A^*{\PX}_1(I + \Ntilde{\PX}_1)^{-1}\Delta A }\nonumber\\
		=&\norm{A^*{\PX}_2(I + \Ntilde{\PX}_2)^{-1}\Delta A - A^*{\PX}_2(I + \Ntilde{\PX}_1)^{-1}\Delta A\nonumber \\
		&+A^*{\PX}_2(I + \Ntilde{\PX}_1)^{-1}\Delta A - A^*{\PX}_1(I + \Ntilde{\PX}_1)^{-1}\Delta A }\\
		\leq&4\norm{A}(\norm{P} + 1)^3\norm{X_2 - X_1}\norm{B}^2\norm{R^{-1}}\epsilon\nonumber\\
		&+ 2\norm{A}\norm{X_2 - X_1}(\norm P  + 1)^2\epsilon.
\end{align}
The difference in~\eqref{e:purple} is the adjoint of the difference in~\eqref{e:orange}, and can be upper bounded by the same factor. An upper bound on the difference in~\eqref{e:sepia} can be computed as follows:
\begin{align}
			\norm{\Delta A^*{\PX}_2(I +& \Ntilde{\PX}_2)^{-1}\Delta A - \Delta A^*{\PX}_1(I + \Ntilde{\PX}_1)^{-1}\Delta A }\nonumber\\
		=&\norm{\Delta A^*{\PX}_2(I + \Ntilde{\PX}_2)^{-1}\Delta A - \Delta A^*{\PX}_2(I + \Ntilde{\PX}_1)^{-1}\Delta A \nonumber\\
			&+\Delta A^*{\PX}_2(I + \Ntilde{\PX}_1)^{-1}\Delta A - \Delta A^*{\PX}_1(I + \Ntilde{\PX}_1)^{-1}\Delta A }\\
		\leq&4(\norm{P} + 1)^3\norm{X_2 - X_1}\norm{B}^2\norm{R^{-1}}\epsilon^2\nonumber\\
		&+ 2\norm{X_2 - X_1}(\norm P  + 1)^2\epsilon^2\\
		\leq & 4(\norm{P} + 1)^3\norm{X_2 - X_1}\norm{B}^3\norm{R^{-1}}\epsilon\nonumber\\
		&+ 2\norm{X_2 - X_1}(\norm P  + 1)^2\norm{B}\epsilon, 
	\end{align}
where we used on the last line the assumption $ε≤‖B‖$. 
The overall upper bound has 6 terms in which $\norm A$ appears at most with degree 2, $\norm P + 1$ with degree 3, $\norm B$ with degree 3, $\norm{R^{-1}}$ with degree 2. With the same reasoning as for the other upper bound, we can write
\begin{equation}
	\norm{ \cG X_1 - \cG X_2} \leq 12 (\norm{A} + 1)^2(\norm{P} + 1)^3( \norm{B} + 1)^3(\norm R^{-1} + 1)^2 \norm{X_1 - X_2}\epsilon.
\end{equation}
To conclude, 
\begin{align}
		\norm{\Phi X_1 - \Phi X_2} \leq 12 \frac{\tau(L, \rho)^2}{1 - \rho^2}[&(\norm{A} + 1)^2(\norm{P} + 1)^3( \norm{B} + 1)^3(\norm R^{-1} + 1)^2 \norm{X_1 - X_2}\epsilon\nonumber \\
		&+ \norm{L}^2\nu\norm{N}\norm{X_1 - X_2}].
\end{align}
\end{proof}

We are now ready to show that $\Phi$ is a contraction from $\mathcal S_{\nu}$ to itself. First, we show that it maps $\cS_{\nu}$ to $\cS_{\nu}$ (Lemma \ref{lemma:T_inv_from_S_to_S}). Then, we show that it is a contraction on $\cS_{\nu}$ (Lemma \ref{lemma:T_inv_contraction}).

\begin{lemma}
	\label{lemma:T_inv_from_S_to_S}
	 Let us choose $\nu$ as a function of $\epsilon$, namely, 
\begin{equation}
	\nu= 	\min\left\{6\epsilon \frac{\tau(L, \rho)^2}{1 - \rho^2}\normp A^2 \normp P^2 \normp B \normp {R^{-1}}, \norm{N}^{-1}, \frac12\right\}.\nonumber
\end{equation}
Moreover, let $\epsilon$ be small enough, namely,
\begin{equation}
	 \epsilon \leq\min\left\{ \frac{1}{12}\normp{L}^{-2} \frac{(1 - \rho^2)^2}{\tau(L, \rho)^4}\normp A^{-2} \normp P^{-2} \normp B^{-3} \normp {R^{-1}}^{-2}, \norm B\right\}.\nonumber
\end{equation}
Lastly, let $\sigma_{\textrm{min}}(P)\geq 1$, and $\rho(L) \leq\rho < 1$. Note that the condition on the singular values of $P$ can be achieved by rescaling $R$ and $Q$ accordingly, as discussed by~\cite{mania2019certainty}.
Then, for $X\in\mathcal S_{\nu}$, we have that
\begin{equation}
	\Phi X\in \mathcal {S}_{\nu}.\nonumber
\end{equation}
\end{lemma}
\begin{proof}
In order to prove the lemma, we need to first show that $\Phi X$ has norm upper bounded by $\nu$, for $X\in\cS_ν$ and a suitable choice of $\nu$. Then, we need to show that $P + \Phi X$ is self-adjoint and positive semi-definite. Starting from Proposition~\ref{prop:bound_tinv_X}, we see that:
\begin{align}
	\norm{\Phi X} &\leq \frac{ \tau(L, \rho)^2}{1-\rho^2}\left[\norm{L}^2\norm{N}\nu^2 +3\epsilon\normp{P}^2\normp{A}^2\normp{R^{-1}}\normp{B}  \right]\\
	&\leq	\frac{ \tau(L, \rho)^2}{1-\rho^2}\left[\norm{L}^2\norm{N}	\left(6\epsilon \frac{\tau(L, \rho)^2}{1 - \rho^2}\normp A^2 \normp P^2 \normp B \normp {R^{-1}}\right)^2\right.\nonumber\\
	 &\quad\left.+3\epsilon\normp{P}^2\normp{A}^2\normp{R^{-1}}\normp{B}  \right]\\
	&= \frac{ \tau(L, \rho)^2}{1-\rho^2}\left[\norm{L}^2\norm{N}12\epsilon \frac{\tau(L, \rho)^4}{(1 - \rho^2)^2}\normp A^2 \normp P^2 \normp B \normp {R^{-1}} + 1 \right]\nonumber\\
	&\quad\cdot3\epsilon\normp{P}^2\normp{A}^2\normp{R^{-1}}\normp{B} \\
	&\leq  \frac{ \tau(L, \rho)^2}{1-\rho^2}\left[\norm{L}^212\epsilon \frac{\tau(L, \rho)^4}{(1 - \rho^2)^2}\normp A^2 \normp P^2 \normp B^3 \normp {R^{-1}}^2 + 1 \right]\nonumber\\
	&\quad\cdot3\epsilon\normp{P}^2\normp{A}^2\normp{R^{-1}}\normp{B} .
\end{align}
For the choice of $\epsilon$ in the statement of this lemma, the bracket term is bounded by 2, meaning that:
\begin{align}
	\norm{\Phi X}
	\leq  \frac{ \tau(L, \rho)^2}{1-\rho^2}6\epsilon\normp{P}^2\normp{A}^2\normp{R^{-1}}\normp{B} .
\end{align}
Note that, for our choice of $\epsilon$,
\begin{align}
	\norm{\Phi X}
	&\leq \frac{1}{2}\frac{1-\rho^2}{ \tau(L, \rho)^2}\normp{R^{-1}}^{-1}\normp{B}^{-2}\normp{L}^{-2}\\
	&\leq\min\left\{\frac 12, \norm{N}^{-1}\right\},
\end{align}
since $‖N‖=‖BR^{-1}B^*‖≤‖B‖²‖R^{-1}‖$ and $\frac{1-\rho^2}{ \tau(L, \rho)^2}\leq 1$. This can be noted by observing that $\tau(L, \rho)$ is a decreasing function of $\rho$ and is equal to 1 for $\rho\geq\norm{L}$ (see~\cite[Page 3]{mania2019certainty}), and $1 - \rho^2 < 1$ by our choice of $\rho$. 
All in  all, this means that $\norm{\Phi X}\leq \nu$. 
To conclude the proof, we can observe that, since $P$ is positive semi-definite and $\sigma_{\textit{min}}(P)\geq 1$ by assumption, for $a\in \cH_1$,
\begin{align}
	\innerprod{(P + \Phi X)a, a}_{\cH_1}&= \innerprod{Pa, a}_{\cH_1} + \innerprod{\Phi Xa, a}_{\cH_1}\\
	&\geq \innerprod{Pa, a}_{\cH_1} - \norm{\Phi X}\norm{a}^2\\
	&\geq \innerprod{Pa, a}_{\cH_1} -\nu\norm{a}^2\\
	&= \norm{P^{1/2}a}^2 -\nu\norm{a}^2\\
	&\geq [\sigma_{\textit{min}}(P) - \nu]\norm a^2\\
	&\geq 0.
\end{align}
Hence, $P + \Phi X\succcurlyeq 0$. 
Moreover, $ΦX$ is self-adoint because $\cT,\cR$ and $X\mapsto F(P+X,A,B)$ preserve self-adjointness.
\end{proof}
\begin{lemma} 
	\label{lemma:T_inv_contraction}
	Let $\nu $ be defined as in Lemma~\ref{lemma:T_inv_from_S_to_S}, and $\rho(L) \leq\rho < 1$. Let 
\begin{equation}
	12\frac{\tau(L, \rho)^4}{(1 - \rho^2)^2}\left[\normp{P}+\normp{L}^2\right]\normp A^2 \normp P^2\normp B^3\normp{R^{-1}}^2\epsilon< 1.\nonumber
\end{equation}
Then, $\exists\eta < 1$ such that, $\forall X_1, X_2 \in \cS_{\nu}$,
	\begin{equation}
		\norm{\Phi X_1 - \Phi X_2}\leq\eta\norm{X_1 - X_2}.\nonumber
	\end{equation}
\end{lemma}
\begin{proof}Let us plug our choice of $\nu$ from Lemma~\ref{lemma:T_inv_from_S_to_S} in the bound of Proposition~\ref{prop:bound_lipschitz_invt}, to get
\begin{align}
	 12\frac{\tau(L, \rho)^2}{1 - \rho^2}&\left[ \normp{A}^2\normp{P}^3\normp{B}^3\normp{ R^{-1}}^2 \norm{X_1 - X_2}\epsilon+ \norm{L}^2\nu\norm{N}\norm{X_1 - X_2}\right]\nonumber\\
	 &=	 12\frac{\tau(L, \rho)^2}{1 - \rho^2}\left[\normp{P}\normp{B}^2\normp{ R^{-1}} + \frac{\tau(L, \rho)^2}{1 - \rho^2}\norm{L}^2\norm{N}\right]\nonumber\\
	 &\quad\cdot\normp A^2 \normp P^2\normp B\normp{R^{-1}}\epsilon\norm{X_1 - X_2}\\
	 &\leq 12\frac{\tau(L, \rho)^2}{1 - \rho^2}\left[\normp{P}\normp{B}^2\normp{ R^{-1}} + \frac{\tau(L, \rho)^2}{1 - \rho^2}\norm{L}^2\normp B^2 \normp{R^{-1}}\right]\nonumber\\
	 &\quad\cdot\normp A^2 \normp P^2\normp B\normp{R^{-1}}\epsilon\norm{X_1 - X_2}\\
	 &=12 \frac{\tau(L, \rho)^2}{1 - \rho^2}\left[\normp{P}+ \frac{\tau(L, \rho)^2}{1 - \rho^2}\norm{L}^2\right]\nonumber\\
	 &\quad\cdot\normp A^2 \normp P^2\normp B^3\normp{R^{-1}}^2\epsilon\norm{X_1 - X_2}\\
	 &\leq 12 \frac{\tau(L, \rho)^2}{1 - \rho^2}\left[ \frac{\tau(L, \rho)^2}{1 - \rho^2}\normp{P}+ \frac{\tau(L, \rho)^2}{1 - \rho^2}\norm{L}^2\right]\nonumber\\
	 &\quad\cdot\normp A^2 \normp P^2\normp B^3\normp{R^{-1}}^2\epsilon\norm{X_1 - X_2}\\
	 &\leq 12 \frac{\tau(L, \rho)^2}{1 - \rho^2}\left[ \frac{\tau(L, \rho)^2}{1 - \rho^2}\normp{P}+ \frac{\tau(L, \rho)^2}{1 - \rho^2}\normp{L}^2\right]\nonumber\\
 	 &\quad\cdot\normp A^2 \normp P^2\normp B^3\normp{R^{-1}}^2\epsilon\norm{X_1 - X_2}.
\end{align}
Again, in the last step we used the fact that $\tau(L, \rho)$ is a decreasing function of $\rho$ and is equal to 1 for $\rho\geq\norm{L}$ (see~\cite[Page 3]{mania2019certainty}), and $1 - \rho^2 < 1$ by our choice of $\rho$.
The operator $\Phi$ is a contraction if  $\epsilon$ is small enough, namely,
\begin{equation}
12\frac{\tau(L, \rho)^4}{(1 - \rho^2)^2}\left[\normp{P}+\normp{L}^2\right]\normp A^2 \normp P^2\normp B^3\normp{R^{-1}}^2\epsilon< 1.
\end{equation}
\end{proof}

Lemmas~\ref{lemma:T_inv_from_S_to_S},~\ref{lemma:T_inv_contraction} imply that, for $\epsilon$ small enough, the unique fixed point of $\Phi$ is in $\mathcal S_{\nu}$. According to Proposition~\ref{prop:T_inv}, that fixed point is exactly the error on the Riccati operator due to the Nyström approximation, whose norm scales linearly in $\epsilon$. To conclude the proof of Lemma~\ref{lemma:convergence_p_operator}, we can observe that the conditions stated in Lemmas~\ref{lemma:T_inv_from_S_to_S},~\ref{lemma:T_inv_contraction} yield an upper bound on $\epsilon$:
\begin{equation}
	\epsilon< \frac{1}{12} \frac{1}{\normp{L}^{2} +\normp{P}}\frac{(1 - \rho^2)^2}{\tau(L, \rho)^4}\normp A^{-2} \normp P^{-2} \normp B^{-3} \normp {R^{-1}}^{-2} ,
\end{equation}
guaranteeing
\begin{equation}
	\norm{P-\tilde P}\leq 6\epsilon \frac{\tau(L, \rho)^2}{1 - \rho^2}\normp A^2 \normp P^2 \normp B \normp {R^{-1}}.
\end{equation}

\section{Proof of Theorem~\ref{thm:convergence_lqr_objective}}

The proof follows the steps in~\cite[Theorem 1]{mania2019certainty}, adapted to the operator setting considered in this work. We begin by computing an upper bound on the error for the Riccati gain $\norm{K-\Ktilde}$. We then show that when this error is small enough, $\Ktilde$ stabilizes the system with exact kernel defined in~\eqref{eq:dynamics_compound}. Lastly, we use~\cite[Lemma 10]{fazel2018global} to upper bound the error on the objective functions.
We begin by re-stating two technical lemmas that appear in~\cite[Section 2.3]{mania2019certainty} generalized to the case in which the LQR objective contains operators.

	\begin{lemma}[Lemma 1 of~\cite{mania2019certainty}]
	\label{lemma:strongly_convex}
	Let $f_1$, $f_2$ be two $\mu$-strongly convex twice differentiable functions on a Hilbert space. Let $\xvec_1=\argmin_{\xvec} f_1(\xvec)$ and $\xvec_2 = \argmin_{\xvec}=f_2(\xvec)$. If $\norm{\frac{\partial f_1(\xvec)}{\partial \xvec}\rvert_{\xvec=\xvec_2}}\leq \epsilon$, then $\norm{\xvec_1 - \xvec_2}\leq \frac\epsilon\mu$.
	\end{lemma}
\begin{lemma}
	\label{lemma:bound_on_riccati_gain}
	Let $A_i:\cH_1\to\cH_1$, $B_i:\R^{n_u}\to\cH_1$, $i=1,2$.
	Let us define $f_i:\R^{n_u}\times\cH_1\rightarrow \R$, $f_i(\uvec, z) = \frac12\innerprod{\uvec, R\uvec}_{\R^{n_u}} + \frac12\innerprod{A_iz + B_i\uvec, P_i(A_iz + B_i\uvec)}_{\cH_1}$ where $R$, $P_i$ are positive definite operators, $i=1, 2$, $z\in \cH_1$, $\uvec\in\R^{n_u}$. Let $K_i:\cH_1\to\R^{n_u}$ be the unique operator s.t.\ $\uvec_i = \argmin_{\uvec}f_i(\uvec, z)=K_i z$ for any $z$. Let us define the quantity $\Gamma = 1 +\max\{\norm {A_1}, \norm {B_1}, \norm {P_1}, \norm {K_1}\}$, and $\norm{A_1 - A_2}\leq \epsilon$, $\norm{B_1 - B_2}\leq \epsilon$, and $\norm{P_1 - P_2}\leq \epsilon$, for $\epsilon\in[0, 1)$. Then, 
	\begin{equation}
		\norm{K_1 - K_2} \leq \frac{3\epsilon\Gamma^3}{\sigma_{\textrm{min}}(R)}.
	\end{equation}
\end{lemma}
\begin{remark}
	Each $f_i$ in the statement of this lemma is the so-called state-action value function in the language of~\cite{fazel2018global}. It corresponds to an LQR objective when picking an arbitrary control input $\uvec$ at time 0, and then proceeding with the optimal control policy (static state feedback gain) for the next time steps. It is known that, if at timestep 1 we start using the LQR optimal control, the LQR objective from that timestep onwards is equal to $\frac{1}{2}\innerprod{z_1, Pz_1}_{\cH_1}$. 
\end{remark}
\begin{remark} 
	The gains $K_1$ (resp.\ $K_2$) in the statement of this lemma are the Riccati gains, obtained with the dynamics given by $A_1$, $B_1$ (resp.\ $A_2$, $B_2$).
\end{remark}
\begin{proof}
The structure of the proof is as follows. We begin by showing that our choice of $f_1, f_2$ fulfills the hypotheses of~\ref{lemma:strongly_convex}. Then we apply such a lemma to upper bound the error of interest. 
	Computing the gradient of $f_i$ w.r.t.\ $\uvec$ yields:
	\begin{align}
		\frac{\partial f_i (\uvec, z)}{\partial \uvec} &= R\uvec + \frac{\partial}{\partial \uvec}\left\{\frac12\innerprod{A_iz + B_i\uvec, P_i(A_iz + B_i\uvec)}_{\cH_1}\right\} \\
		&= R\uvec +  \frac{\partial}{\partial \uvec}\left\{\frac12\innerprod{A_iz, P_iA_iz }_{\cH_1} + \innerprod{A_iz, P_iB_i\uvec}_{\cH_1} +  \frac12\innerprod{B_i\uvec, P_iB_i\uvec}_{\cH_1}\right\}\\
		&= R\uvec + B_i^*P_iA_iz + B_i^*P_i^{1/2}P_i^{1/2}B_i\uvec\\
		&= B_i^*P_iA_iz + (B_i^*P_iB_i + R)\uvec.\label{eq:derivative_objective}
	\end{align}
After having derived the gradient expression above, we can bound two differences that will appear in the remainder of the proof. In the same way as in~\cite{mania2019certainty}, we can observe that, having defined $\Lambda=\max\{\norm{P_1}, \norm{B_1}, \norm{A_1}\}$, and considering the assumption that $\epsilon\in[0, 1)$:
\begin{align}
	\norm{B_1^*P_1B_1 - B_2^*P_2B_2} =& \norm{(B_1^* - B_2^*)P_1B_1 + B_2^*(P_1 - P_2)B_1 + B_2^*P_2(B_1 - B_2)}\\
	=&\norm{(B_1^* - B_2^*)P_1B_1 + B_2^*(P_1 - P_2)B_1 + B_2^*P_2(B_1 - B_2)\nonumber\\
		&-B_1^*(P_1 - P_2)B_1 + B_1^*(P_1 - P_2)B_1\nonumber\\
		&-B_1^*P_2(B_1 - B_2) + B_1^*P_2(B_1 - B_2)}\\
	=&\norm{(B_1^* - B_2^*)P_1B_1 + (B_2^* - B_1^*)(P_1 - P_2)B_1 + (B_2^* - B_1^*)P_2(B_1 - B_2)\nonumber\\
		&+ B_1^*(P_1 - P_2)B_1+ B_1^*P_2(B_1 - B_2)\nonumber\\
		&- (B_2^* - B_1^*)P_1(B_1 - B_2) + (B_2 ^*- B_1^*)P_1(B_1 - B_2)\nonumber\\
		&-B_1^*P_1(B_1 - B_2) + B_1^*P_1(B_1 - B_2)}\\
	=& \norm{(B_1^* - B_2^*)P_1B_1 + (B_2^* - B_1^*)(P_1 - P_2)B_1 \nonumber\\
		&+ (B_2^* - B_1^*)(P_2 - P_1)(B_1 - B_2)\nonumber\\
		&+ B_1^*(P_1 - P_2)B_1+ B_1^*(P_2 - P_1)(B_1 - B_2)\nonumber\\
		&+ (B_2^* - B_1^*)P_1(B_1 - B_2)\nonumber\\
		&+ B_1^*P_1(B_1 - B_2)}\\
	\leq& \epsilon \norm{P_1}\norm{B_1} + \epsilon^2\norm{B_1} + \epsilon^3+ \norm{B_1}^2\epsilon+ \norm{B_1}\epsilon^2+ \epsilon^2\norm{P_1}+ \norm{B_1}\norm{P_1}\epsilon\\
	\leq& \epsilon\Lambda^2 + \epsilon^2\Lambda + \epsilon^3 + \epsilon\Lambda^2 + \epsilon\Lambda + \epsilon^2 \Lambda + \Lambda^2\epsilon\\
	\leq& \epsilon (3\Lambda^2 + 3\Lambda + 1)\\
	\leq& 3\epsilon\Gamma^2.
\end{align}
Similarly, we have  
\begin{align}
		\norm{B_1^*P_1A_1 - B_2^*P_2A_2} =& \norm{(B_1^* - B_2^*)P_1A_1 + B_2^*(P_1 - P_2)A_1 + B_2^*P_2(A_1 - A_2)}\\
	=&\norm{(B_1^* - B_2^*)P_1A_1 + B_2^*(P_1 - P_2)A_1 + B_2^*P_2(A_1 - A_2)\nonumber\\
		&-B_1^*(P_1 - P_2)A_1 + B_1^*(P_1 - P_2)A_1\nonumber\\
		&-B_1^*P_2(A_1 - A_2) + B_1^*P_2(A_1 - A_2)}\\
	=&\norm{(B_1^* - B_2^*)P_1A_1 + (B_2^* - B_1^*)(P_1 - P_2)A_1 + (B_2^* - B_1^*)P_2(A_1 - A_2)\nonumber\\
		&+ B_1^*(P_1 - P_2)A_1+ B_1^*P_2(A_1 - A_2)\nonumber\\
		&- (B_2^* - B_1^*)P_1(A_1 - A_2) + (B_2 ^*- B_1^*)P_1(A_1 - A_2)\nonumber\\
		&-B_1^*P_1(A_1 - A_2) + B_1^*P_1(A_1 - A_2)}\\
	=& \norm{(B_1^* - B_2^*)P_1A_1 + (B_2^* - B_1^*)(P_1 - P_2)A_1 \nonumber\\
		&+ (B_2^* - B_1^*)(P_2 - P_1)(A_1 - A_2)\nonumber\\
		&+ B_1^*(P_1 - P_2)A_1+ B_1^*(P_2 - P_1)(A_1 - A_2)\nonumber\\
		&+ (B_2^* - B_1^*)P_1(A_1 - A_2)\nonumber\\
		&+ B_1^*P_1(A_1 - A_2)}\\
	\leq& \epsilon \norm{P_1}\norm{A_1} + \epsilon^2\norm{A_1} + \epsilon^3\nonumber\\
		&+ \norm{B_1}\norm{A_1}\epsilon+ \norm{B_1}\epsilon^2+ \epsilon^2\norm{P_1}+ \norm{B_1}\norm{P_1}\epsilon\\
	\leq& \epsilon\Lambda^2 + \epsilon^2\Lambda + \epsilon^3 + \epsilon\Lambda^2 + \epsilon\Lambda + \epsilon^2 \Lambda + \Lambda^2\epsilon\\
	\leq& \epsilon (3\Lambda^2 + 3\Lambda + 1)\\
	\leq& 3\epsilon\Gamma^2.
\end{align}
We can then consider the following difference. For arbitrary $\uvec$, $z$, by using the gradient expression in~\eqref{eq:derivative_objective} and the triangle inequality,
\begin{equation}
	\left\lVert	\frac{\partial f_1 (\uvec, z)}{\partial \uvec} - 	\frac{\partial f_2 (\uvec, z)}{\partial \uvec}\right\rVert\leq 3\epsilon\Gamma^2(\norm \uvec + \norm z).\label{eq:grad_diff}
\end{equation}

In particular, for any $z$ \st $‖z‖≤1$, this inequality applied to $\uvec_1=K_1 z$ yields by the first-order optimality condition of $\uvec_1$:
\begin{align}
		\left\lVert	\frac{\partial f_1 (\uvec, z)}{\partial \uvec} \rvert_{\uvec=\uvec_1}- 	\frac{\partial f_2 (\uvec, z)}{\partial \uvec}\rvert_{\uvec=\uvec_1}\right\rVert
		&=	\left\lVert	\frac{\partial f_2 (\uvec, z)}{\partial \uvec}\rvert_{\uvec=\uvec_1}\right\rVert 
		\leq 3\epsilon\Gamma^2(\norm {\uvec_1} + 1).
\end{align}

Hence, still for any $z$ with $‖z‖≤1$, applying Lemma~\ref{lemma:strongly_convex} to the functions $f₁,f₂$ defined in the statement, which are $\mu$-strongly convex with $\mu\geq \sigma_{\textrm{min}}(R)$, we get
\begin{align}
	\n{K₁z-K₂z}
	&≤ \sigma_{\textrm{min}}(R)^{-1} 3εΓ^2(\norm{¯\uvec_1} +1) \\
	&\leq \sigma_{\textrm{min}}(R)^{-1}3εΓ^2(\norm{K_1}\norm{z} +1)\\
	&\leq \sigma_{\textrm{min}}(R)^{-1}3εΓ^2(\norm{K_1} +1)\\
	&\leq \sigma_{\textrm{min}}(R)^{-1}3εΓ^3.
\end{align}

This yields the claimed result given that
\begin{align}
	\norm{K_1 - K_2} 
	&= \sup_{\norm{z}\leq 1} \norm{(K_1 - K_2)z}. 
\end{align}
\end{proof}

Having proved the lemma above, we now specialize it to the case of exact vs.\ Nyström kernel-based gains. 
\begin{proposition}
	Let $\epsilon>0$ be s.t.\ $\norm{ A - \Atilde}\leq \epsilon$, $\norm{ B - \Btilde}\leq \epsilon$, and $\norm{P - \Ptilde}\leq g(\epsilon)$ with $g(\epsilon)\geq \epsilon$. Then, assuming $R$ and $Q$ are positive definite and $\sigma_{\textrm{min}}(R)\geq1$, we have
	\begin{equation}
		\label{eq:bound_gain_error}
		\norm{\Ktilde - K} \leq 3\Gamma^3g(\epsilon).
	\end{equation}
Moreover, choose $\rho$ such that $\rho(L) \leq \rho<1$. If $\epsilon$ is small enough s.t.\ the r.h.s.\ of~\eqref{eq:bound_gain_error} is smaller than $\frac{1 - \rho}{2\tau(L, \rho)}$, then
\begin{equation}
	\tau\left(A + B\Ktilde, \frac{1 + \rho}{2}\right)\leq \tau(L, \rho).
\end{equation}
\end{proposition}
\begin{proof}
	The first part of the proposition is a corollary of Lemma~\ref{lemma:bound_on_riccati_gain} based on the condition $g(\epsilon)> \epsilon$ and $\sigma_{\textrm{min}}(R) \geq 1$. The second part of the proposition follows from applying the first inequality of~\cite[Lemma 5]{mania2019certainty} with $M=A+BK$ and $Δ=B(\tilde K- K)$, which guarantees that, if 
	\begin{equation}
		\norm{\Ktilde - K}\leq \frac{1 - \rho}{2\tau(L, \rho)\norm B},
	\end{equation}
	then 
\begin{align}
	\norm{(A + B\Ktilde)^k}&= \norm{(A + BK + B\Ktilde - BK)^k}\\
	&= \norm{(A + BK + B(\Ktilde - K))^k}\\	
	&\leq\tau(L, \rho)\left(\frac{1 - \rho}{2} + \rho\right)^k\\
	&= \tau(L, \rho)\left(\frac{1 + \rho}{2}\right)^k.
\end{align}
\end{proof}

Now we are ready to conclude the proof. Let us consider the following decomposition of the error on the LQR objective, which corresponds to the well known \emph{performance difference lemma} in the reinforcement learning literature~\cite{kakade2002approximately}. The following lemma is a restatement of~\cite[Lemma 10]{fazel2018global}, with our notation and setup (i.e., with operators).
\begin{lemma}
Let $\cJ$ and $\cJhat$ be as of~\eqref{eq:obj} and~\eqref{eq:nobj}. Then, the following decomposition holds:
\begin{equation}
	\cJhat - \cJ = \lim_{T\to\infty}\sum_{i=0}^{T}\ \innerprod{(\Ktilde - K)z_i, (R + B^*PB)(\Ktilde - K)z_i}_{\cH_1}
\end{equation}
\end{lemma}
\begin{proof}
	We begin by using the same telescoping argument applied by~\cite{fazel2018global}. 
	Let $A$ and $B$ be defined as in~\eqref{e:def_A},~\eqref{e:def_B}. Let us define, for an arbitrary initial condition $\hat z_0\in\cH_1$:
	\begin{equation}
		\hat z_{i+1} = A\hat z_i + B\tuveco_i,\tuveco_i=\Ktilde\zhat_i, i\geq 0.
\end{equation} 
Similarly to~\eqref{eq:nobj}, let us define, for an initial condition, at timestep $i$, $\hat z_i\in\cH_1$:
	\begin{align}
	\hat z_{i,j+1}' &= A\hat z_{i,j}' + B\uveco_j, \uveco_j = K \zhat_j',  j\geq i,\\
	\hat z_{i,i}' &= \hat z_i.
\end{align}
We can decompose the error of interest as follows, noting that, by the definition in~\eqref{eq:obj},
\begin{equation}
\cJ=\sum_{j=0}^T \left(\innerprod{\hat z_{0, j}', Q\hat z_{0, j}'}_{\cH_1}+ \innerprod{\uveco_j, R\uveco_j}_{\R^{n_u}}\right):
\end{equation}
\begin{align}
	\cJhat - \cJ =& \lim_{T\to\infty}\left\{\sum_{i=0}^T \left[\innerprod{\hat z_i, Q\hat z_i}_{\cH_1} + \innerprod{\tuveco_i, R\tuveco_i}_{\R^{n_u}}\right] - \cJ\right\} \\
	=& \lim_{T\to\infty}\left\{\sum_{i=0}^T \left[\innerprod{\hat z_i, Q\hat z_i}_{\cH_1} + \innerprod{\tuveco_i, R\tuveco_i}_{\R^{n_u}} \right.\right.\nonumber\\
& +\sum_{j=i}^T \left(\innerprod{\hat z_j', Q\hat z_j'}_{\cH_1} + \innerprod{\uveco_j, R\uveco_j}_{\R^{n_u}}\right)\nonumber\\
	&- \left.\left.\sum_{j=i}^T \left(\innerprod{\hat z_j', Q\hat z_j'}_{\cH_1}+ \innerprod{\uveco_j, R\uveco_j}_{\R^{n_u}}\right)\right] - \cJ\right\}\\
	=&\lim_{T\to\infty}\left\{\sum_{i=0}^T \left[\innerprod{\hat z_i, Q\hat z_i}_{\cH_1} + \innerprod{\tuveco_i, R\tuveco_i}_{\R^{n_u}} +\sum_{j=i+1}^T \left(\innerprod{\hat z_{i,j}', Q\hat z_{i,j}'}_{\cH_1} + \innerprod{\uveco_j, R\uveco_j}_{\R^{n_u}}\right)\right.\right.\nonumber\\
	&- \left.\left.\sum_{j=i}^T \left(\innerprod{\hat z_{i,j}', Q\hat z_{i,j}'}_{\cH_1}+ \innerprod{\uveco_j, R\uveco_j}_{\R^{n_u}}\right)\right]\right\}.
\end{align}
Now we can consider a single addend in the outer-most sum. 
Consider that by definition
	\begin{equation}
		\tuveco_i = \Ktilde \zhat_i.
	\end{equation}
Moreover, note that
\begin{align}
	\zhat_{i, i+1}' &= A \zhat_{i, i}' + B \Ktilde\zhat_{i, i}'\\
	&=A \zhat_i + B \Ktilde\zhat_i.
\end{align}
Lastly, note that, subject to the Riccati-optimal state feedback law $\uveco$, the following simplification is allowed, for $T\to \infty$:
\begin{align}
	\sum_{j=i+1}^T \left(\innerprod{\hat z_{i,j}', Q\hat z_{i,j}'}_{\cH_1} 
	+ \innerprod{\uveco_j, R\uveco_j}_{\R^{n_u}}\right)&= \innerprod{\zhat_{i, i+1}' , P \zhat_{i, i+1}' }_{\cH_1} \\
	&=  \innerprod{(A + B\Ktilde)\hat z_i, P (A + B\Ktilde) \hat z_i}_{\cH_1}.
\end{align}
Similarly,  observe that, for $T\to \infty,$
\begin{align}
	\sum_{j=i}^T \left(\innerprod{\hat z_{i,j}', Q\hat z_{i,j}'}_{\cH_1} 
	+ \innerprod{\uveco_j, R\uveco_j}_{\R^{n_u}}\right)&= \innerprod{\zhat_{i, i}' , P \zhat_{i, i}' }_{\cH_1} \\
	&=  \innerprod{\hat z_i, P \hat z_i}_{\cH_1}.
\end{align}
Hence, we can rewrite the addend as the following, for $T\to\infty$:
\begin{align}
	\innerprod{\hat z_i, Q\hat z_i}_{\cH_1}& + \innerprod{\tuveco_i, R\tuveco_i}_{\R^{n_u}} 
		+\sum_{j=i+1}^T \left(\innerprod{\hat z_{i,j}', Q\hat z_{i,j}'}_{\cH_1} 
		+ \innerprod{\uveco_j, R\uveco_j}_{\R^{n_u}}\right) \nonumber\\
	&-\sum_{j=i}^T \left(\innerprod{\hat z_{i,j}', Q\hat z_{i,j}'}_{\cH_1}+ \innerprod{\uveco_j, R\uveco_j}_{\R^{n_u}}\right)\nonumber\\
	=& \innerprod{\hat z_i, (Q + \Ktilde^*R\Ktilde) \hat z_i}_{\cH_1} + \innerprod{(A + B\Ktilde)\hat z_i, P (A + B\Ktilde) \hat z_i}_{\cH_1} -\innerprod{\hat z_i, P\hat z_i}_{\cH_1}\\
	=& \innerprod{\hat z_i, (Q + (\Ktilde - K + K)^*R(\Ktilde- K + K) \hat z_i}_{\cH_1} \nonumber\\
	&+ \innerprod{(A + B\Ktilde -BK + BK)\hat z_i, P (A + B\Ktilde - BK + BK)\hat z_i}_{\cH_1} -\innerprod{\hat z_i, P\hat z_i}_{\cH_1}\\
	=& \innerprod{\zhat_i, Q\zhat_i}_{\cH_1} + \innerprod{\zhat_i, (\Ktilde - K)^*R(\Ktilde - K)\zhat_i}_{\cH_1} + \innerprod{\zhat_i, K^*R(\Ktilde - K)\zhat_i}_{\cH_1}\nonumber \\
	&+ \innerprod{\zhat_i, (\Ktilde - K)^*RK\zhat_i}_{\cH_1} + 
	\innerprod{\zhat_i, K^*RK\zhat_i}_{\cH_1}\nonumber\\
	&+\innerprod{B(\Ktilde - K)\hat z_i, PB (\Ktilde - K)\hat z_i}_{\cH_1} + \innerprod{(A + BK)\hat z_i, P (A + BK)\hat z_i}_{\cH_1}\nonumber\\
	&+\innerprod{B(\Ktilde - K)\hat z_i, P (A + BK)\hat z_i}_{\cH_1}+ \innerprod{(A + BK)\hat z_i, PB (\Ktilde - K)\hat z_i}_{\cH_1}\nonumber\\
	&-\innerprod{\zhat_i, P\zhat_i}_{\cH_1}\\
	=& \innerprod{\zhat_i, (\Ktilde - K)^*(R + B^*PB)(\Ktilde -K)\zhat_i}_{\cH_1}\nonumber\\
	&+ 2\innerprod{\zhat_i, (\Ktilde - K)^*[RK  + B^*P(A+ BK)]\zhat_i}_{\cH_1}\\
	=&\innerprod{\zhat_i, (\Ktilde - K)^*(R + B^*PB)(\Ktilde -K)\zhat_i}_{\cH_1}\nonumber\\
	&+ 2\innerprod{\zhat_i, (\Ktilde - K)^*[(R + B^* P B)K  + B^*PA]\zhat_i}_{\cH_1}\\
	=&\innerprod{\zhat_i, (\Ktilde - K)^*(R + B^*PB)(\Ktilde -K)\zhat_i}_{\cH_1}\nonumber\\
	&+ 2\innerprod{\zhat_i, (\Ktilde - K)^*[(R + B^* P B)(-1)(R + B^*PB)^{-1}B^*PA  + B^*PA]\zhat_i}_{\cH_1}\\
	=& \innerprod{\zhat_i, (\Ktilde - K)^*(R + B^*PB)(\Ktilde -K)\zhat_i}_{\cH_1}.
\end{align}
\end{proof}

To conclude the proof of the theorem, we can apply Cauchy-Schwarz inequality, Lemma~\ref{lemma:bound_on_riccati_gain}, and the fact that, since $\sigma_{\textrm{min}}(R)\geq 1$, it holds
\begin{align}
	\norm{R + B^*PB}&\leq \norm {R} + \norm{B^*PB}\\
	&\leq \norm {R} + \norm R\norm{B^*PB}\\
	&\leq \sigma_{\textrm{max}}(R)\Gamma^3,
\end{align}
where we recall that $Γ\de 1 + \max(‖A‖,‖P‖,‖K‖,‖B‖)$ in the statement of Theorem~\ref{thm:convergence_lqr_objective}.  
We have that
\begin{align}
	\cJhat - \cJ &= \lim_{T\to\infty}\sum_{i=0}^T\innerprod{\zhat_i, (\Ktilde - K)^*(R + B^*PB)(\Ktilde -K)\zhat_i}_{\cH_1}\\
	&\leq \lim_{T\to\infty}\sum_{i=0}^T\norm{\zhat_i}_{\cH_1}^2\norm{\Ktilde - K}^2\norm{R + B^*PB}\\
	&\leq 9 \sigma_{\textrm{max}}(R)\Gamma^9g(\epsilon)^2\lim_{T\to\infty}\sum_{i=0}^T\norm{\zhat_i}_{\cH_1}^2\\
	&=9 \sigma_{\textrm{max}}(R)\Gamma^9g(\epsilon)^2\lim_{T\to\infty}\sum_{i=0}^T\norm{(A + B\Ktilde)^iz_0}_{\cH_1}^2\\
	&\leq 9 \sigma_{\textrm{max}}(R)\Gamma^9g(\epsilon)^2\norm{z_0}_{\cH_1}^2\lim_{T\to\infty}\sum_{i=0}^T\norm{(A + B\Ktilde)^i}^2\\
	&\leq 9 \sigma_{\textrm{max}}(R)\Gamma^9g(\epsilon)^2\norm{z_0}_{\cH_1}^2\lim_{T\to\infty}\sum_{i=0}^T\tau(A + BK, \rho)^2\left(\frac{1 + \rho}{2}\right)^{2i}\\
	&\leq 9 \sigma_{\textrm{max}}(R)\Gamma^9g(\epsilon)^2\norm{z_0}_{\cH_1}^2\tau(A + BK, \rho)^2\frac{1}{1 - \left(\frac{1 + \rho}{2}\right)^2}\\
	&\leq 36 \sigma_{\textrm{max}}(R)\Gamma^9g(\epsilon)^2\norm{z_0}_{\cH_1}^2\frac{\tau(A + BK, \rho)^2}{1 - \rho^2}.
\end{align}
To conclude, note that, according to Assumption~\ref{a:bounded_kernel},
	\begin{align}
		\norm{z_0}^2_{\cH_1}&=‖k(x_0,·)‖²_{\cH_1}\\
		&=\ip{k(x_0,·), k(x_0,·)}_{\cH_1}\\
		&=k(x_0,x_0)\\
		&≤\kappa^2.
	\end{align}

\end{document}